\documentclass[onefignum,onetabnum]{siamonline190516}

\usepackage{lipsum}
\usepackage{amsfonts}
\usepackage{graphicx}
\usepackage{epstopdf}
\usepackage{amsmath}
\usepackage{amssymb}
\usepackage[utf8]{inputenc}
\usepackage[english]{babel}
\usepackage{graphicx}
\usepackage{tcolorbox}
\usepackage{pgfplots} 
\usepackage{algorithm}
\ifpdf
  \DeclareGraphicsExtensions{.eps,.pdf,.png,.jpg}
\else
  \DeclareGraphicsExtensions{.eps}
\fi

\newcommand{\specificthanks}[1]{\@fnsymbol{#1}}

\newsiamremark{hypothesis}{Hypothesis}
\crefname{hypothesis}{Hypothesis}{Hypotheses}
\newsiamthm{claim}{Claim}
\newsiamremark{remark}{Remark}
\newsiamremark{assumption}{Assumption}

\headers{Second-order geometric quasilinear hyperbolic PDEs}{G. Dong, M. Hinterm\"uller, Y. Zhang}

\title{A class of second-order geometric quasilinear hyperbolic PDEs and their application  in imaging
\thanks{ Received by the editors September 10, 2020; accepted for publication (in revised form) February 1, 2021.
	\funding{The work of GD and MH is supported by the Deutsche Forschungsgemeinschaft (DFG, German Research Foundation) under Germany's Excellence Strategy – The Berlin Mathematics Research Center MATH+ (EXC-2046/1, project ID: 390685689). 
	The work of YZ is supported by the Guangdong fundamental and applied research fund [No. 2019A1515110971].}
} }
\author{Guozhi Dong\footnotemark[3]\,\;\thanks{Weierstrass Institute for Applied Analysis and Stochastics (WIAS),	Mohrenstr. 39, 10117 Berlin, Germany, \email{guozhi.dong@wias-berlin.de}; \email{michael.hintermueller@wias-berlin.de}.}
\and Michael Hinterm\"uller\footnotemark[2]\,\;\thanks{Institute for Mathematics, Humboldt University of Berlin, Unter den Linden 6, 10099 Berlin, Germany, \email{guozhi.dong@hu-berlin.de}; \email{hint@hu-berlin.de}.}
\and Ye Zhang\thanks{School of Mathematics and Statistics, Beijing Institute of Technology, 100081 Beijing, China; Shenzhen MSU-BIT University, Longgang District, 518172 Shenzhen, China, \email{ye.zhang@smbu.edu.cn}.}}
\usepackage{amsopn}

\ifpdf
\hypersetup{
  pdftitle={Second-order geometric quasilinear hyperbolic PDEs]{A class of second-order geometric quasilinear hyperbolic PDEs and their application in imaging},
  pdfauthor={Guozhi Dong, Michael Hinterm\"uller, Ye Zhang}
}}
\fi

\newcommand{\norm}[1]{\left\| #1 \right\|}
\newcommand{\abs}[1]{ \left\vert #1 \right\vert}
\newcommand{\set}[1]{\left\{ #1 \right\}}
\newcommand{\f}[1]{\mathbf{#1}}

\newcommand{\R}{ \mathbb{R}}

\newcommand{\N}{ \mathbb{N}}

\newcommand{\ud}{u^{d}}

\begin{document}

\maketitle

\begin{abstract}
Motivated by important applications in image processing, we study a class of second-order geometric quasilinear hyperbolic partial differential equations (PDEs).
This is inspired by the recent development of second-order damping systems associated to gradient flows for energy decaying. In numerical computations, it turns out that the second-order methods are superior to their first-order counter-parts.
We concentrate on (i) a damped second-order total variation flow for e.g., image denoising, and (ii) a damped second-order mean curvature flow for level sets of scalar functions. The latter is connected to a non-convex variational model capable of correcting displacement errors in image data (e.g. dejittering).
For the former equation, we prove the existence and uniqueness of the solution and its long time behavior, and provide an analytical solution given some simple initial datum. 
For the latter, we draw a connection between the equation and some second-order geometric PDEs evolving the hypersurfaces,
and show the existence and uniqueness of the solution for a regularized version of the equation.
Finally, some numerical comparison of the solution behavior for the new equations with first-order flows are presented.
\end{abstract}



\section{Introduction}
Total variation flow (TVF) and mean curvature flow for level sets of scalar functions (called level-set MCF in what follows) are important nonlinear evolutionary geometric  partial differential equations (PDEs) which have been of interest in many fields during the last three decades. In the literature, they have been intensively investigated either analytically  \cite{AmbSon96,AndBalCas01,AndCasMaz04,BalCasNov02,CheGigGot91,EvaSpr91} or from a computational viewpoint \cite{BarGarNue08,BarDieNoc18,DecDziEll05,MalSet96,OshSet88}, to name just a few.
In particular, they both find application in imaging science and geometry processing, and they are of common interest to variational and PDE methods in image processing and analysis.
This is due to the fact that an image (or more general data) can be treated as a function defined on a bounded domain in $\R^n$, or more specifically a rectangular domain  $\Omega\subset \R^2$.
This is also the particular focus of the current paper in the applications, where we consider $u:\R^2 \to \R$ as an image function, and $\ud$ as given degraded image data. In our practical context, distorted images are (i) subject to some additive noise in which case $\ud= u+ \delta$, where $u$ denotes the true image, or (ii) corrupted by displacement errors $d:\R^2 \to \R^2$ which gives $\ud=u(x+d(x))$.

The first case is a fundamental problem in image processing and has been continuously and intensively studied from many perspectives. Mathematical methods are also developed from several different points of view, and many of them are based on the well-known Rudin-Osher-Fatemi (ROF) model \cite{RudOshFat92,ChaLio97}, where total variation (TV) is used for removing additive noise from image data. It is associated with a non-smooth energy functional, and it has the beneficial property of preserving the discontinuities (edges) of an image, which are often considered important features.
Accordingly, TVF, the gradient flow of the TV functional, has been studied in this context and also beyond, see for instance \cite{AndBalCas01,AndCasMaz04,BalCasNov02,BarDieNoc18} and the references therein.

Problems with displacement errors have, mathematically, been the subject of several recent studies.
This kind of error is not linearly separable like additive noise but rather it constitutes a nonlinear phenomenon calling for new ideas for correction.
In the literature, studies are mostly focused on specific sub-classes such as, e.g. image dejittering which restricts the error $d:\R \to \R$ to occur on only one direction.
In the work of \cite{LenSch11,DonPatSchOek15,DonSch17}, it is found out that the level-set MCF and some of its variants are capable of correcting displacement errors.
An intuitive understanding is that the displacement errors interrupt the level lines of image functions, and level-set MCF is in fact a minimizing flow for the perimeter of the level lines of the functions.
By setting $u_0=\ud$, the evolution of the level-set MCF produces a regularized solution which remedies the displacement errors in $\ud$.
A proper application of the level-set MCF in this context needs, however, an appropriate stopping.
Similarly, let the initial data be a noisy image, the TVF is able to decrease the total variation of the noisy image, and thus regularize the image when it is, again, properly stopped.

To summarize, both cases mentioned above fall into the common framework:
\begin{equation}
\label{eq:GF}
\dot{w} =-\partial \Phi(w) \quad \text{ and } \quad w(0)=w_0,
\end{equation}
where $\Phi$ is a general convex functional, and $\partial$ denotes the gradient (or subgradient) operator.
Throughout this paper, we will use Newton's notation for the partial derivative with respect to time. For TVF, we can think of $w$ as the evolutionary image function, and $\Phi(w)$ denotes the total variation of $w$, and $w_0=u^d$.
In the context for the level-set MCF, we understand $w$ in \eqref{eq:GF} to be an immersion of a hypersurface representing a level set of a proper function, and $\Phi$ denotes the area functional of the hypersurface, and $w_0$ be level sets of $u^d$.

More recently, second-order dynamics of the following form
\begin{equation}
\label{eq:DF}
\ddot{w}+ \eta(t)\dot{w}=-\partial \Phi(w), \quad w(0)=w_0 \quad \text{ and }\quad \dot{w}(0)=0,
\end{equation}
have been of great interest in the field of (convex) optimization; see, e.g. \cite{SuBoyCan16}. By some of the authors and other colleagues \cite{AttChbHas18,ZhaHof18,BotDonElbSch18}, it has also been applied as regularization methods for solving inverse problems.
The damped second-order dynamics are supposed to be superior to the first-order gradient flows.
The case of $\eta(t)$ being a constant is sometimes called a Heavy-Ball-with-Friction system (HBF) in the literature, see, e.g., in~\cite{Attouch-2000}. This system is an asymptotic approximation of the equation describing the motion of a material point with positive mass, subject to remaining on the graph of $\Phi(w)$,  which moves under the action of the gravity force, and the friction force ($\eta>0$ is the associated friction parameter). The introduction of the inertial term $\ddot{w}(t)$ to the dynamic  system allows to overcome some of the drawbacks of gradient descent methods, such as the well-known zig-zag phenomenon.
However, in contrast to gradient descent methods, the HBF system is not necessarily a descent method for the potential energy $\Phi$.
Instead, it decreases the total energy (kinetic+potential). The damping parameter $\eta$ may control the kinetic part. Larger values of $\eta$ in \eqref{eq:DF} result in more rapid evolution, while smaller values yield \eqref{eq:DF} more wave-like characteristics.
The optimization properties of the HBF system have been intensively studied in \cite{Alvarez-2000,Alvarez-2002,Attouch-2000}, and the references therein.
In particular in the recent works, e.g., \cite{SuBoyCan16,AttCab17,AttCabPeyRed17}, well posedness of the general framework \eqref{eq:DF}, and also the convergence and convergence rates of its solutions have been carefully analyzed, particularly connections to Nestrov's acceleration algorithms \cite{Nes04} have been established and discussed.

Within such a trend, second order PDE dynamics emerge naturally as a topic also in imaging, see for instance the very recent work  \cite{BarSveGulZha19,BenCalSunYez20}. 
We note that the standard theory on HBF does typically not apply to PDEs, particularly, when $\partial \Phi$ gives rise to some nonlinear and nonsmooth partial differential operators.
In our context, however, we are confronted with quasilinear hyperbolic PDEs.
In fact, in this paper we investigate the damped second-order dynamics for both TVF and level-set MCF.
Such PDEs have not been studied in the existing works which mostly focus on the algorithmic aspects.
Our goal is to understand the new equations and their solutions from a theoretical point of view, on the one hand, and to explore their applications in imaging, on the other hand.
In so doing, there are several mathematical challenges to overcome.
First of all, difficulties arising due to the non-linearity and non-smoothness in both the second-order TVF and the second-order level-set MCF have to be addressed.
Second, both the second-order dynamics are of PDEs of quasilinear hyperbolic type, which are in general subtler than first-order ones of parabolic type as the maximum principle is out of reach in the former case.
Moreover, for the level-set MCF, no convex energy associated to the level-set function has been found so far. Consequently, convex analysis techniques can not be applied here.
Compared to TVF, fundamental mathematical questions such as the existence of solutions and also uniqueness of the solution require more efforts, or even need to introduce a new concept for solutions of the second-order level-set MCF.
Therefore, the results of this paper will not only provide novel PDE methods for image processing, but also contribute and propose interesting research questions to the fields of PDE and geometric analysis.

As already sketched, one of the main contributions of this paper is to provide theoretical understanding of the two proposed novel PDEs. More precisely, we prove existence and uniqueness of the solution to the Cauchy problem for the damped second-order TVF. In doing so, we take advantage of the TV energy functional being convex.
We employ the Yosida approximation to show the existence of a solution, and develop an iterative scheme for proving the uniqueness of the solution.
We are able to show that there is no finite extinction time for the second-order TV flow, which is different to its first-order counterpart.
By choosing a simple initial datum, i.e., an eigenfunction of the subgradient of the total variation functional, we are able to derive analytical solutions to the second-order TVF,
which show more intuitive understanding on the behavior of the solution of this equation.
For the damped second-order level-set MCF, we find a connection between the equation and another novel second-order geometric PDE which evolves hypersurfaces.
This provides insight into  the behavior of  solutions of the second-order level-set MCF if we take the hypersurfaces to be the level sets of our function.
The damped second-order level-set MCF is a fully degenerate quasilinear hyperbolic PDE, for which a general theory seems to be elusive at this point.
In the literature, hyperbolic type MCF has been pointed out in \cite{Yau00}, and studied in \cite{GinSva16,HeKonLiu09,LefSmo08}. However, still very little of its analytical properties is known. In the current paper, we study level set formulations of some damped hyperbolic MCFs. We find interesting connections between level-set hyperbolic MCFs and the ones for hypersurfaces, although these are analytically open questions.
As a first step towards a solution concept, we show the existence and uniqueness of the solution to a regularized version of the damped second-order level-set MCF, which is also used in our numerical studies.

In view of applicability, hyperbolic type PDE tools are still new to imaging applications, at least in comparison to parabolic type PDEs. Therefore, this paper provides some new alternative PDE-based methods to imaging problems and opens up some new perspectives.
On the other hand, it is known that the first-order level-set MCF is also a minimizing flow  of the total variation of the initial data, however, it exhibits different behavior than the first-order TVF.
In fact, while the first-order TVF is known to decrease the contrast (height), the first-order level-set MCF shrinks the scale of image features (geometry).
Their second-order counterparts that we study in this work are able to preserve these features. This has been verified in our numerical examples.
We notice that applications of first-order TV flow and MC flow have been extensively explored. In this paper, as a starting point and a proof of concept, we consider two fundamental applications (denoising and correcting displacement errors) for the corresponding second-order flows. In particular, the solution behavior of the second-order MC flow is not trivial as we will explain in Section \ref{sec:mcf}, as well as the second-order TV flow with inhomogeneous initial velocities, due to some of the still open questions on their analytical properties. Therefore their numerical tests are meaningful, since they can serve as evidence for theoretical investigations.
In terms of computational efficiency, in the setting of the current paper, algorithms from second-order equations are numerically superior in comparison with their first-order counterparts.
Another aspect of second-order dynamics which has not been widely studied so far is the impact of the initial velocity. Due to the connection with Nesterov's algorithm, homogeneous initial velocity is often considered as default. However, some of our intuitive result shows that, inhomogeneous initial velocities exhibits some new perspective,  which open new avenues for research beyond this paper.

\begin{table}[h!]

	\begin{center}
		\resizebox{\textwidth}{!}{
		\begin{tabular}{|c|c|}
			\hline
			Notation & Description   \\
			\hline
			$f(x)$ & if there is only spatial variable $x$, then $f:\Omega \to \R^n$  is time-independent  \\
			\hline
			$f(t)$ & if there is only temporal variable $t$, then $f$ maps to either real values in $\R^n$\\
			& or elements in some function spaces that is $f(t)=f(\cdot,t)$ \\
			\hline
			$f(x,t)$ & $f$ is both space and time-dependent \\
			\hline
			$\dot{f}$ \& $\ddot{f}$ &  first-order and second-order  time derivative of $f$, respectively \\
			\hline
			$\nabla f$ & spatial gradient (distributional sense) of function $f$  \\
			\hline
			$\operatorname{div}( \mathbf{v})$ & spatial divergence  (distributional sense) of vector field $\mathbf{v}$  \\
			\hline
			$(\cdot,\cdot)$ &  inner product of two elements in Hilbert spaces (mostly $L^2$ space)\\
			\hline
			$\langle\cdot,\cdot\rangle$ &  inner product of two elements in $\R^n$ \\
			\hline
		\end{tabular}}	
	\caption{Notations and abbreviations.}\label{tab:Notation}
	\end{center}
\end{table}

{\bf Notation}: Table \ref{tab:Notation} summarizes notations and abbreviations which will be frequently used in this paper.
The time dependence of the function $u$, which appears throughout the text, is not explicitly written. We remark that this should be clear in its context.
Similarly, we will often omit the spatial variable $x$ for functions which depend on space and time. Note that the spatial variable $x$ is considered as a pixel of 2-d images in our numerical examples. However, in our theoretical analysis,  $x$ is taken to be a general element in some bounded set of $\R^n$ or the entire  $\R^n$.

The remainder of the paper is structured as follows:
Section \ref{sec:tvf} provides the mathematical analysis of the total variation flows.
Section \ref{sec:mcf} investigates the level sets mean curvature flows.
Section \ref{sec:algorithm} presents an algorithm and the results of numerical comparisons.
A convergence analysis of the algorithm is included in the Appendix.

\section{Total variation flows}
\label{sec:tvf}
\subsection{The first-order total variation flow}
We start by reviewing the first-order total variation flow (TVF) and its corresponding variational method.
Total variation has become a standard tool in mathematical methods for image processing since the final decade of the last century,
which is attributed to the seminal work of Rudin, Osher and Fatemi \cite{RudOshFat92}, who introduced the following nonsmooth variational model for recovering noisy images
\begin{equation}
\label{eq:ROF}
\min_u \frac{1}{2}\int_{\Omega} \abs{u(x)-\ud(x)}^2 dx + \alpha \text{TV}(u).
\end{equation}
Here $\alpha>0$ is a regularization parameter, and $\text{TV}(\cdot)$ is known as the total variation functional. 
Problem \eqref{eq:ROF} is usually referred to as ROF model in the literature.
From a practical point of view,  $\text{TV}(\cdot)$ is preferable in image processing to the standard Tikhonov regularization  (quadratic smooth regularization) because it is able to keep sharp contrast (edges) in the image.

We recall the definition of the total variation functional $\text{TV}(\cdot)$ here. Let $\Omega$ be a compact subset of $\R^n$. For a function $u:\Omega \rightarrow \R$, the total variation is defined as the value
\begin{equation}
\label{eq:tv}
\norm{Du}:=\sup  \left\{\int_{\Omega} u(x)\operatorname{div}(\f v(x))dx: \f v\in  C_0^\infty (\Omega),\; \abs{\f v(x)}\leq 1 \text{ for all } x\in \Omega \right\},
\end{equation}
where $ C_0^\infty (\Omega)$ presents the set of infinitely continuously differentiable functions compactly supported in $\Omega$. The space of functions of bounded variation on $\Omega$, usually denoted by $BV(\Omega)$, is given by
\begin{equation}
\label{BVspace}
BV(\Omega):= \left\{ u\in L^1(\Omega):~ \norm{u}_{BV}<+\infty \right\}, \; \text{ where } \norm{u}_{BV}:=\norm{u}_{L^1}+\norm{Du}.
\end{equation}
It is well known that $BV(\Omega)$ is a Banach space, and the Sobolev space $W_1^1(\Omega)$ is embedded into $BV(\Omega)$.
We recall that for functions $u\in W_1^1(\Omega)$ the total variation is equally characterized by the $L^1$ norm of the spatial gradient of $u$, that is
\begin{equation}
\label{eq:tv2}
\norm{Du}=\int_\Omega \abs{\nabla u(x)}dx,\; \text{ for } u\in W_1^1(\Omega).
\end{equation}
In the following, we shall consider a Hilbert space for the function $u$.
In particular we assume $u\in L^2(\Omega)$ for the purpose of our studies.
Let $\text{TV}(u)$ be the total variation of the function $u\in L^2(\Omega)$, then
\[\text{TV}(u):=\left\{
\begin{array}{ll}
\norm{Du}, & u\in  L^2(\Omega) \bigcap BV(\Omega);\\
\infty ,&    u\in L^2(\Omega)\setminus BV(\Omega).
\end{array}\right.\]
It is not difficult to find that the functional $\text{TV}(\cdot)$ is convex, proper, and lower semi-continuous in the Hilbert space $L^2(\Omega)$.
Note that in the following we always use the gradient notation $\nabla$ instead of $D$ for functions also in $BV(\Omega)$.
Given the following minimization problem
\begin{equation}
\label{eq:min_tv}
\min \left\{ \text{TV}(u): u\in L^2(\Omega)\bigcap BV(\Omega)  \right\},
\end{equation}
recall that the first-order TVF can be expressed as the negative $L^2$ gradient flow for minimizing \eqref{eq:min_tv},
which reads
\begin{equation}\label{eq:tvf}
\left\{\begin{array}{ll}
\dot{u}(t)=\operatorname{div}\left( \frac{\nabla u(t)}{\abs{\nabla u(t)}} \right), \; &\text{ in } \Omega \times (0,\infty)\\
u(0)=u_0, \; &\text{ in } \Omega \times 0.
\end{array}\right.
\end{equation}
Note here  that $-\operatorname{div}\left( \frac{\nabla u}{\abs{\nabla u}}\right)$ is formally identified with an element of the subgradient of $\text{TV}(u)$. It is important to give a sense to \eqref{eq:tvf} as a partial differential ``equation". This was addressed in e.g. \cite{AndCasMaz04}.
The idea there is to introduce  some vector field $p(t)$ as an element in the space $X(\Omega):=\set{p(t)\in L^\infty(\Omega,\R^n): \operatorname{div}(p(t))\in L^2(\Omega)}$ for all $t\in (0,\infty)$. Then the equation \eqref{eq:tvf} is understood in the sense of $\dot{u}(t)=\operatorname{div}(p(t)) $, where $p$ has the form:
\begin{equation}\label{eq:sub_gradient_u}
p(t) = \left\{ \begin{array}{ll}
\frac{\nabla u(t)}{\abs{\nabla u(t)}}, &\text{ if }\quad  \nabla u(t) \neq 0 ;\\
\gamma(t)  \quad \text{ for some } \abs{ \gamma(t)}\leq 1, &\text{ if }\quad  \nabla u(t) = 0,
\end{array}\right.
\end{equation}
which provides a more detailed understanding of \eqref{eq:tvf}.
This also applies later to \eqref{eq:acc_tvf}, one of our target equations in this paper.
For details on the TVF, we refer to, e.g., \cite{AndBalCas01,AndCasMaz04,BalCasNov02}.
There, the existence and uniqueness of solutions problem \eqref{eq:tvf} with Neumann/Dirichelet boundary condition on $\Omega$ are established. Also, the more general case where $\Omega$ is the entire space $\R^n$ was studied.
These developments are mostly motivated by applications in image denoising.
Indeed, setting the initial value $u_0=\ud$, and running the flow stopped at a proper time, yields a regularized image.
Usually the filtering of TVF is less destructive to the edges in images than filtering with a Gaussian, i.e., solving the heat equation with the same initial value $u_0$.

A formal connection between the TVF \eqref{eq:tvf} and the ROF variational model \eqref{eq:ROF} can be drawn as follows.
Given the initial value $u_0$, we consider an implicit time discretization of the TVF \eqref{eq:tvf} using the following iterative procedure:
\begin{equation}\label{eq:time_dis_tvf}
u_m-u_{m-1} \in -\Delta t \partial \text{TV}(u_m), \quad \text{ for } m\in \N,
\end{equation}
where $\partial$ denotes the subdifferential operator.
Identifying the time step $\Delta t$ with the regularization parameter in \eqref{eq:ROF}, that is $\alpha=\Delta t$, we see that \eqref{eq:time_dis_tvf} is in fact the first order optimality condition of \eqref{eq:ROF}. Therefore each iteration in \eqref{eq:time_dis_tvf} can be equivalently approached by solving  \eqref{eq:ROF}, where we take $u=u_m$ and $u^d=u_{m-1}$.

\subsection{The damped second-order total variation flow}
Following the idea of damped second-order dynamics for gradient flows of convex functionals in Hilbert spaces, we introduce the following second-order TVF:
\begin{equation}\label{eq:acc_tvf}
\left\{\begin{array}{ll}
\ddot{u}(t)+ \eta \dot{u}(t) - \operatorname{div}\left( \frac{\nabla u(t)}{\abs{\nabla u(t)}} \right) =0 , & \textmd{in} \ \Omega \times (0,\infty),\\
u(0) = u_0, \quad  \dot{u}(0) = v_0 ,&  \textmd{in} \  \Omega \times 0,\\
\partial_\nu u(t) = 0, & \textmd{on} \  \partial \Omega \times  (0,\infty),
\end{array}\right.
\end{equation}
where $\eta>0$ is the so-called damping parameter, which  is assumed to be a constant, and $\partial \Omega$ denotes the boundary of the domain $\Omega $ which is Lipschitz continuous, $\partial_\nu$ is the normal derivative and  $\nu$ denotes the outward unit normal vector on $\partial \Omega$.

In order to study the resolvability of \eqref{eq:acc_tvf} we consider the following concept for its solutions.
\begin{definition}
	\label{def:Solution_STVF}
	A function $u \in \mathcal{V}:=L^\infty([0,\infty);D(\partial \text{TV}))\cap W_2^{\infty}([0,\infty);L^2(\Omega)) $, is called a weak solution of \eqref{eq:acc_tvf} provided
	\begin{equation}\label{eq:WeakSolu}
	\ddot{u}(t) + \eta \dot{u}(t) + \partial \text{TV}(u(t)) \ni 0 \;\textmd{~for almost every~} \; t\geq 0.
	\end{equation}
	given the initial and boundary conditions in \eqref{eq:acc_tvf}. 
\end{definition}
Note that $W_2^{\infty}([0,\infty);L^2(\Omega))\subset C^1([0,\infty);L^2(\Omega)).$

Before discussing the existence of solutions for \eqref{eq:acc_tvf}, we recall the resolvent operator as well as the Yosida approximation operator for the $\text{TV}$ functional. These are standard tools available in many classic textbooks (see e.g. \cite[Chapter 7]{Bre10}):
\begin{definition}
	\label{def:Yosida}
	\begin{itemize}
		\item[(i)] The resolvent operator $J_{\lambda}:L^2(\Omega) \to D(\partial \text{TV})$ is defined by $J_{\lambda}(w):=u$, where $u\in BV(\Omega)$ is the unique solution of
		\[u+\lambda \partial \text{TV}(u) \ni w.\]
		\item[(ii)] The Yosida approximation operator $A_{\lambda}:L^2(\Omega) \to L^2(\Omega) $ is defined as
		\[A_{\lambda}(w):=(w-J_{\lambda}(w))/\lambda.\]
	\end{itemize}
\end{definition}

The operators $J_{\lambda}$ and $A_{\lambda}$ have the following properties (see also  \cite[Chapter 7]{Bre10}):
\begin{proposition}
	\label{Prop:Yosida}
	\begin{enumerate}
		\item[(i)] For any fixed $\lambda>0$, $A_{\lambda}$ is a Lipschitz continuous mapping, i.e.
		\[\|A_{\lambda}(w_1)-A_{\lambda}(w_2)\|\leq \frac{2}{\lambda} \|w_1 - w_2\|\; \text{ for all } \;w_1,w_2\in L^{2}(\Omega) .\]
		\item[(ii)] $A_{\lambda}$ is a monotone operator, i.e. $(A_{\lambda}(w_1)-A_{\lambda}(w_2), w_1 - w_2)\geq0$ for all $w_1,w_2\in L^{2}(\Omega) $.
		\item[(iii)] $A_{\lambda}(w)\in \partial \text{TV}(J_{\lambda}(w))$ for all $w\in L^{2}(\Omega) $.
		\item[(iv)] For all $w\in D(\partial \text{TV}) $:
		\begin{equation}\label{supA}
		\sup_{\lambda>0} \|A_{\lambda}(w)\| \leq \left| (\partial \text{TV})^0 (w)\right| := \min_{v\in \partial \text{TV}(w)} \|v\|.
		\end{equation}
		\item[(v)] For every $w\in L^2(\Omega)$:
		\[\lim_{\lambda\to 0} J_\lambda(w)=w.\]
	\end{enumerate}
\end{proposition}

We make use of the following lemma from \cite[Proposition 1.10]{AndCasMaz04}.
\begin{lemma}
	\label{lem:homogenous}
	Let  $u\in D(\partial \text{TV})$, and $\beta \in \partial \text{TV}(u)$. Then,
	\[(\beta, u)=\text{TV}(u).\]
\end{lemma}
Note that this is a special case of a general result which states that the above equality still holds true for an arbitrary  convex functional homogeneous of degree $1$ besides the $\text{TV}$ functional.
Having these results at hand, we proceed with proving the existence of a solution to \eqref{eq:acc_tvf}.
\begin{theorem}
	\label{thm:existence_TVF}
	Given $u_0 \in D(\partial TV)$, $v_0\in  D(\partial \text{TV})\cap L^2(\Omega) $ there exists a solution of \eqref{eq:acc_tvf} in the sense of Definition \ref{def:Solution_STVF}.
\end{theorem}

\begin{proof}
	We first consider the following approximate problem with fixed $\lambda>0$:
	\begin{equation}\label{approximatePro}
	\left\{\begin{array}{ll}
	\ddot{u}_{\lambda}(t) + \eta \dot{u}_{\lambda}(t) + A_{\lambda}(u_{\lambda}(t)) =0 ,& \textmd{in} \  \Omega \times (0,\infty),\\
	u_{\lambda}(0) =u_0 , \quad \dot{u}_{\lambda}(0) = v_0 ,&  \textmd{in} \  \Omega \times 0,\\
	\partial_\nu u_\lambda(t)=0 ,& \textmd{on} \  \partial \Omega \times  (0,\infty).
	\end{array}\right.
	\end{equation}
	For simplicity, denote $H=L^{2}(\Omega) $, and introduce the space $H\times H$ with scalar product
	\[([u_1,v_1]^\top,[u_2,v_2]^\top)_{H\times H} = (u_1,u_2)+(v_1,v_2),\]
	and the corresponding norm $\|[u,v]^\top\|_{H\times H}=\sqrt{\|u\|^2+\|v\|^2}$.
	Note that in the following proof, if there is no specification, then $\norm{\cdot}$ always means the $H$ norm.
	Now, we define $v_{\lambda}(t)=\dot{u}_{\lambda}(t)$ and $\mathbf{z}(t)=[u_{\lambda}(t),v_{\lambda}(t)]^\top$, and then, rewrite \eqref{approximatePro} as a first-order dynamical  system in the phase space $H\times H$, i.e.
	\begin{equation}\label{vectorform}
	\left\{\begin{array}{ll}
	\dot{\mathbf{z}}(t) = F(\mathbf{z}(t)) & \textmd{in } \ \Omega \times (0,\infty),\\
	\mathbf{z}(0) = [u_0,v_0]^\top & \textmd{in } \  \Omega \times  0,
	\end{array}\right.
	\end{equation}
	where $F(\mathbf{z}(t))=[v_{\lambda}(t), - \eta v_{\lambda}(t) - A_{\lambda}(u_{\lambda}(t))]^\top$.
	
	We show first that $F$ is Lipschitz continuous for every fixed $\lambda>0$.
	This is true by using the Lipschitz continuity of the Yosida approximation operator $A_{\lambda}$, and we have the following inequalities
	\begin{eqnarray*}
		&& \| F(\mathbf{z}_1(t)) - F(\mathbf{z}_2(t)) \|_{H\times H} \\
		&& = \sqrt{\|v_{\lambda,1}(t)-v_{\lambda,2}(t)\|^2 + \|\eta (v_{\lambda,1}(t)-v_{\lambda,2}(t)) + (A_{\lambda}(u_{\lambda,1}(t))-A_{\lambda}(u_{\lambda,2}(t)))\|^2} \\
		&& \leq \sqrt{(1+2\eta^2)\|v_{\lambda,1}(t)-v_{\lambda,2}(t)\|^2 + \frac{8}{\lambda^2} \|u_{\lambda,1}(t)-u_{\lambda,2}(t)\|^2}  \\
		&& \leq \sqrt{1+2\eta^2 + \frac{8}{\lambda^2}} \left\| [u_{\lambda,1}(t),v_{\lambda,1}(t)]^\top - [u_{\lambda,2}(t),v_{\lambda,2}(t)]^\top \right\|_{H\times H}.
	\end{eqnarray*}
		
	The existence and uniqueness of the solution of \eqref{vectorform} follow from the Cauchy-Lipschitz-Picard (CLP) theorem (see e.g. \cite{Bre10}) for first-order dynamical systems.
	In particular we can infer that
	\[u_\lambda \in C^1([0,\infty);H) \cap L^\infty([0,\infty); D(\partial \text{TV})),\; \text{ and } \; \dot{u}_{\lambda}\in C^1([0,\infty);H),\]
	and $\dot{u}_{\lambda}\in C^1([0,\infty);H)$ indicates that $u_{\lambda}\in C^2([0,\infty);H)$.
	
	In the remaining part of the proof, we show that as $\lambda\to0^+$ the function sequence $(u_{\lambda})_\lambda$ converges to a
	solution of problem \eqref{eq:acc_tvf} in the sense of Definition \ref{def:Solution_STVF}.
	We prove this by the following steps.
	
	\textbf{Step 1}. Show that $\dot{u}_\lambda \in L^2([0,\infty);H)$.
	
	According to the definition of $J_{\lambda} $ and the assumption $u_0\in D(\partial \text{TV})$, we have
	\begin{equation}
	\label{boundedTVu0}
	\text{TV}(J_{\lambda}(u_0)) \leq \text{TV}(u_0) <\infty.
	\end{equation}
	
	Define first a Lyapunov function of the differential equation \eqref{approximatePro}, that is
	\begin{equation}
	\label{LyapnunovE}
	\mathcal{E}_\lambda(t) := \frac{1}{2} \norm{\dot{u}_\lambda(t)}^2 + \text{TV}(J_{\lambda}(u_\lambda(t))),
	\end{equation}
and similarly $\mathcal{E}(t) := \frac{1}{2} \norm{\dot{u}(t)}^2 + \text{TV}(u_\lambda(t))$.
It is not difficult to show that
	\begin{equation}
	\label{Decreasing1}
	\dot{\mathcal{E}}_\lambda(t) = - \eta \|\dot{u}_\lambda(t)\|^2
	\end{equation}
	by considering  \eqref{approximatePro}.
	Integrating both sides in \eqref{Decreasing1}, we obtain
	\begin{equation}
	\label{eq:derivative_bound1}
	\int^{\infty}_{0} \|\dot{u}_\lambda(t)\|^2 dt \leq  \mathcal{E}_\lambda(0) / \eta =  \text{TV}(J_{\lambda}(u_0))/\eta + \frac{\norm{v_0}^2_{L^2(\Omega)} }{2\eta}\leq \mathcal{E}(0)/\eta  <  \infty,
	\end{equation}
	which yields $\dot{u}_\lambda  \in L^2([0,\infty),H)$ for all $\lambda\geq 0$.
	
	\textbf{Step 2}. Prove that both $u_\lambda , \dot{u}_\lambda  \in L^\infty([0,\infty),H)$ are uniformly bounded.
	
	Since we consider $\Omega$ bounded, according to Poincare's inequality, see e.g. \cite[Theorem 3.47 and Remark 3.50]{Ambrosio2000}\footnote{The estimate in \eqref{eq:SobolevIneq} holds in the case that $\Omega$ is a bounded set in $\R^n$ of all integer $n>0$, but not for the whole $\R^n$ in general (e.g., when $n>2$). That means similar estimate in \eqref{eq:SobolevIneq} might not hold true for functions of which the domain is the whole space $\R^n$. }, a constant $C$ independent of $\lambda$ exists such that
	\begin{equation}
	\label{eq:SobolevIneq}
	\|J_\lambda(u_\lambda)(t)\| \leq C \cdot \text{TV}(J_\lambda(u_\lambda)(t)).
	\end{equation}
	
	On the other hand, according to the assertion $(\text{v})$ in Proposition  \ref{Prop:Yosida}, a constant $\lambda_1$ exists such that for all $\lambda\in(0,\lambda_1]$,
	\begin{equation}
	\label{DiffJlambda1}
	\|J_\lambda(u_\lambda(t)) - u_\lambda(t)\| \leq 1 \quad \text{ for all } t\in (0,\infty).
	\end{equation}
	Note that it follows from \eqref{Decreasing1} that $\mathcal{E}_{\lambda}$ is a non-increasing function. Thus, we obtain together with \eqref{eq:SobolevIneq} and \eqref{DiffJlambda1} that
	\begin{eqnarray*}
		&& \|u_\lambda(t)\| \leq \|J_\lambda(u_\lambda(t))\| + 1 \leq C\cdot \text{TV}(J_\lambda(u_\lambda(t))) + 1 \\ && \qquad
		\leq C \left( \text{TV}(J_\lambda(u_\lambda(t))) + \frac{1}{2}\|\dot{u}_\lambda(t)\|^2  \right) + 1 = C \mathcal{E}_\lambda(t) + 1 \leq C \mathcal{E}(0)  + 1 <  \infty,
	\end{eqnarray*}
	which yields $u_\lambda \in L^\infty([0,\infty);H)$ for all $\lambda\in(0,\lambda_1]$.
	
	The uniform boundedness of $\dot{u}_\lambda $ follows from the following inequality:
	\begin{eqnarray*}
		\|\dot{u}_\lambda(t)\|^2 \leq 2\text{TV}(J_\lambda(u_\lambda(t))) + \|\dot{u}_\lambda(t)\|^2 = 2\mathcal{E}_\lambda(t) \leq 2 \mathcal{E}(0) <  \infty.
	\end{eqnarray*}

	\textbf{Step 3}. Show that $A_{\lambda}(u_\lambda), \ddot{u}_\lambda \in L^\infty([0,\infty);H)$ are both uniformly bounded.
	
	We have shown in Step 2 that $J_\lambda(u_\lambda)\in L^\infty([0,\infty);H)$ is uniformly bounded for all $\lambda\in (0,\lambda_1]$.
 	Now we  show, by contradiction, that $A_{\lambda}(u_\lambda)\in L^\infty([0,\infty);H)$ is uniformly bounded. Assume that there exists $(\lambda_k)_k\in (0,\lambda_1]$ such that $\left(A_{\lambda_k}(u_{\lambda_k}(t))\right)_{k\in \N}$ is an unbounded sequence as $\lambda_k\to 0$ for some $t\in (0,\infty)$, i.e., $\lim_{k\to\infty} \|A_{\lambda_k}(u_{\lambda_k}(t))\| = \infty$.  
	 Now we consider the sequence $\left( J_{\lambda_k}(u_{\lambda_k})\right)_{k\in \N}$ corresponding to $\left(A_{\lambda_k}(u_{\lambda_k})\right)_{k\in \N}$.
	Note that there must exist a subsequence such that the elements of $\left( J_{\lambda_k}(u_{\lambda_k}(t))\right)_{k\in \N}$ are not constant functions, otherwise, since $A_{\lambda_k}(u_{\lambda_k}(t))=\partial \text{TV}(J_{\lambda_k}(u_{\lambda_k}(t)) )$ (Proposition \ref{Prop:Yosida} $iii$), we get $A_{\lambda_k}(u_{\lambda_k}(t))\equiv 0$ which is already a contradiction. We  consider in particular this subsequence and use the same notation.	
	Now, let us define $d_{k}(t):= A_{{\lambda_k}}(u_{\lambda_k}(t))/ \norm{A_{\lambda_k}(u_{\lambda_k}(t))} $ and its smooth approximation $\tilde{d}_k(t)\in  C_0^\infty(\Omega) $, such that for arbitrary $\epsilon\in (0,1)$ we have:
	\[(A_{\lambda_k}(u_{\lambda_k}(t)), \tilde{d}_k(t))\geq (A_{\lambda_k}(u_{\lambda_k}(t)), d_k(t)) - \epsilon .\]
	Note that such $\tilde{d}_k(t)$ always exists since $A_{\lambda_k}(u_{\lambda_k}(t))\in L^2(\Omega)$ and $C^\infty_0(\Omega)$ is dense in $L^2(\Omega)$.
	Since $d_k$ is unitary, therefore both  $d_k$ and $\tilde{d}_k$ are uniformly bounded in $L^\infty((0,\infty);L^2(\Omega))$ for all $k\in \N$.
	Moreover, $\text{TV}(\tilde{d}_k(t))$ is also uniformly bounded as  $\tilde{d}_k(t)\in  C_0^\infty(\Omega) $ and $\Omega\subset \R^n$ is bounded.
	Because $A_{\lambda_k}(u_{\lambda_k}(t))\in \partial \text{TV}(J_{\lambda_k}(u_{\lambda_k}(t)))$, we have that for every $k\in \N$,
	\begin{eqnarray*}
		\text{TV}(J_{\lambda_k}(u_{\lambda_k}(t)) + \tilde{d}_k(t)) - \text{TV}(J_{\lambda_k}(u_{\lambda_k}(t)) )
		 \geq (A_{\lambda_k}(u_{\lambda_k}(t)), \tilde{d}_k(t)) \geq  \norm{A_{\lambda_k}(u_{\lambda_k}(t))} -\epsilon, 
	\end{eqnarray*}
	for some $\epsilon\in (0,1)$ and all $t\in (0,\infty)$. This implies that 
	\[	\text{TV}(J_{\lambda_k}(u_{\lambda_k}(t)) + \tilde{d}_k(t)) \to \infty \quad \text{ as }\quad k\to \infty.\]
	However, due to the triangle inequality
	\begin{eqnarray*}
\text{TV}(J_{\lambda_k}(u_{\lambda_k}(t)) + \tilde{d}_k(t))\leq  \text{TV}(J_{\lambda_k}(u_{\lambda_k}(t)) )+ \text{TV}( \tilde{d}_k(t))
		\leq  \text{TV}(u_0)+\text{TV}(\tilde{d}_k(t)),
	\end{eqnarray*}
 where $\text{TV}(\tilde{d}_k(t))$ is uniformly bounded for  every $t \in (0,\infty)$.
	This means that 
	\[\text{TV}(J_{\lambda_k}(u_{\lambda_k}(t)) + \tilde{d}_k(t)))<\infty \]
	which gives the contradiction.
	Therefore, we have that $A_{\lambda}(u_\lambda(\cdot))\in L^\infty([0,\infty);H)$ is uniformly bounded for all $\lambda\in (0,\lambda_1]$.
	
	The uniform boundedness of $\ddot{u}_\lambda  \in L^\infty([0,\infty);H)$ for $\lambda\in(0,\lambda_1]$ follows from the obtained results $\dot{u}_\lambda , A_{\lambda}(u_\lambda )\in L^\infty([0,\infty);H)$ and equation \eqref{approximatePro}.

	\textbf{Step 4}.
	Show the existence of a function $u\in W_2^\infty([0,\infty);H)$ which is a solution to \eqref{eq:acc_tvf} in the sense of Definition \ref{def:Solution_STVF}.

	First, we claim that for every sequence $(\lambda_k)_{k\in \N}$ with $\lambda_k\to 0$, there exists a uniformly convergent subsequence $(u_{\lambda_k})_{k\in \N}\in C^1([0,T];H)$ (here we do not change the notation for the subsequence) and for every $t\in [0,T]$ of arbitrary $T\in (0,\infty)$, so that
	\begin{equation}
	\label{eq:udot_lambda_convergence}
	u_{\lambda_k}  \to u  \textmd{~in~}   C^0([0,T];H), \quad \text{ and } \quad
	\dot{u}_{\lambda_k}  \to \dot{u}  \textmd{~in~}  C^0([0,T];H).
	\end{equation}
	This follows from the Arzel\'a-Ascoli theorem by noting that
	\[u_{\lambda_k} \in L^\infty([0,T];H) \quad  \text{ and } \quad  \dot{u}_{\lambda_k}  \in L^\infty([0,T];H)  \] are uniformly bounded for all $\lambda\in (0,\lambda_1)$, as well as $\ddot{u}_{\lambda_k}  \in L^\infty([0,T];H)$. Therefore, all elements of both $(u_{\lambda_k}) $ and $(\dot{u}_{\lambda_k}) $ are Lipschitz continuous thus equicontinuous over $t\in [0,T]$ for arbitrary $T\in (0,\infty)$. 
	Note that subsequences have to be applied here whenever they are needed, and the  uniform convergence in \eqref{eq:udot_lambda_convergence} implies that 
	\[ u_{\lambda_k}  \to u  \textmd{~in~}   C^1([0,T];H), \;  \text{ uniformly. }\]
	
	Furthermore, the uniform boundedness of  $\ddot{u}_{\lambda_k}$ in $L^\infty([0,\infty);H)$ implies that there exists a subsequence $(\lambda_{k_j})_{j\in \N}$  such that for arbitrary $T<\infty$:
	\begin{equation}
	\label{eq:weakConv}
	\ddot{u}_{\lambda_{k_j}} \rightharpoonup \ddot{u} \textmd{~in~}  L^2([0,T];H) \textmd{~as~}\;  j\to \infty.
	\end{equation}
	
	Now we show that $u(t)\in D(\partial \text{TV})$ for every $t\in (0,T]$, and it holds that
	\[ \ddot{u}(t) +\eta \dot{u}(t) \in -\partial \text{TV}(u(t))  \quad \text{ for a.e.  }\quad  t\in (0,T]  . \]
	We first notice for each $t>0$ that
	\begin{equation*}
	- \ddot{u}_{\lambda}(t) - \eta \dot{u}_{\lambda}(t) = A_{\lambda}(u_\lambda) \in \partial \text{TV}(J_{\lambda}(u_\lambda)),
	\end{equation*}
	which means that for arbitrary but fixed $w\in H$, we have
	\begin{equation*}
	\text{TV}(w) \geq \text{TV}(J_{\lambda}(u_\lambda)(t)) - (\ddot{u}_{\lambda}(t) + \eta \dot{u}_{\lambda}(t), w- J_{\lambda}(u_\lambda(t))).
	\end{equation*}
	Consequently, for $0< s\leq t\leq \infty$, it holds that
	\begin{equation*}
	(t-s)\text{TV}(w) \geq \int^t_s \text{TV}(J_{\lambda}(u_\lambda(\tau))) d\tau - \int^t_s (\ddot{u}_{\lambda}(\tau) + \eta \dot{u}_{\lambda}(\tau), w- J_{\lambda}(u_\lambda(\tau)))  d\tau.
	\end{equation*}
	
	Using the triangle inequality and Definition \ref{def:Yosida} of the Yosida approximation operator, we get:
	\begin{equation*}
	\begin{aligned}
 \norm{J_{\lambda}(u_\lambda(t))- u(t)} &\leq \norm{J_{\lambda}(u_\lambda(t))- u_\lambda(t)}+\norm{u(t)- u_\lambda(t)} \\
 & =\lambda\norm{A_\lambda u_\lambda(t)} +\norm{u(t)- u_\lambda(t)}.
	\end{aligned}
	\end{equation*}
	It has been shown in Step 3 that  $\norm{A_\lambda u_\lambda(t)}$ is uniformly bounded for all $\lambda\in (0,\lambda_1)$ and all $t\in (0,\infty)$. In combination with the uniform convergence of $u_\lambda \to u$, we have that  $J_{\lambda}(u_\lambda) \to u$ as $\lambda\to 0$ uniformly for all $t\in (0,T]$.
	Using the lower semi-continuity of the $\text{TV}$ functional, Fatou's Lemma, and the convergence \eqref{eq:udot_lambda_convergence} and weak convergence \eqref{eq:weakConv}, we conclude upon sending $\lambda=\lambda_{k_j}\to 0$ that
	\begin{equation*}
	(t-s)\text{TV}(w) \geq \int^t_s \text{TV}(u(\tau)) d\tau - \int^t_s (\ddot{u}(\tau) + \eta \dot{u}(\tau), w- u(\tau))  d\tau.
	\end{equation*}
	Thus, when $t$ is a Lebesgue point of $\dot{u},\ddot{u}$ and $\text{TV}(u)$, it holds that
	\begin{equation*}
	\text{TV}(w) \geq \text{TV}(u(t)) - (\ddot{u}(t) + \eta \dot{u}(t), w- u(t)),
	\end{equation*}
	for all $w\in H$. Since $- (\ddot{u}(t) + \eta \dot{u}(t))\in H$ is bounded, by definition of subgradient we have $u(t)\in D(\partial \text{TV})$ and
	\[-(\ddot{u}(t) + \eta \dot{u}(t))\in \partial \text{TV}(u(t)) \quad \text{ for almost every }\quad t\in (0,\infty).\]
	
	Finally we show that for all $t>0$, $u(t)\in D(\partial \text{TV})$.
	For every $t>0$, let $t_n\to t$ and $u(t_n)\in D(\partial \text{TV})$, and $- (\ddot{u}(t_n) + \eta \dot{u}(t_n))\in \partial \text{TV}(u(t_n))$. Because of the uniform boundedness of both $\dot{u}$ and $\ddot{u}$, and the continuity of $\dot{u}$, we have, up to a subsequence:
	\[\dot{u}(t_n)\to  \dot{u}(t)\; \text{ in }\; H, \text{ and } \; \ddot{u}(t_n) \rightharpoonup v\; \text{ weakly in } \;H \text{ for some } v\in H.\]
	For every fixed $w\in H$, we have
	\[\text{TV}(w)\geq \text{TV}(u(t_n))-(\ddot{u}(t_n) + \eta \dot{u}(t_n),w-u(t_n)).\]
	Passing to the limit $n\to \infty$, and due to the continuity of $u$ and the lower semi-continuity of $\text{TV}$, we arrive at
	\[\text{TV}(w)\geq \text{TV}(u(t))-(v + \eta \dot{u}(t),w-u(t)).\]
Then we have shown that $u(t)\in D(\partial \text{TV})$ for all $t>0$ and $ -(v+\eta \dot{u}(t))\in \partial \text{TV}(u(t))$. This concludes the proof.
\end{proof}
We remark that for the first-order TVF \eqref{eq:tvf}, using tools from semi-group theory, the regularity of the initial data can be relaxed to $L^2(\Omega)$ or even $L^1(\Omega)$ to prove the existence and uniqueness of solutions \cite{AndCasMaz04}. However, this does not seem to hold true for the second-order TVF \eqref{eq:acc_tvf} as it is a nonlinear wave equation, and particularly the semi-group theory does not apply here.
Now we continue with the uniqueness of the solution.
\begin{theorem}
	\label{thm:uniqueness_TVF}
	Problem \eqref{eq:acc_tvf}  admits a unique weak solution given prescribed initial and boundary conditions.
\end{theorem}
\begin{proof}
	Let $u$ and $\bar{u}$ both be solutions of \eqref{eq:acc_tvf} that satisfy both the initial and boundary conditions. Further $p$ and $\bar{p}$ are the function forms in \eqref{eq:sub_gradient_u} corresponding to  $u$ and $\bar{u}$, respectively.
	For every $ s\in (0,T]$, define for every function $g\in \mathcal{V} $
	\begin{equation*}
	\phi^s_g(t):=
	\left\{
	\begin{array}{ll}
	-\int_t^s g(r)dr ,& \text{ for } t\in (0,s), \\
	0, & \text{ for } t\geq s.
	\end{array}\right.
	\end{equation*}
	It is not hard to see that $\phi^s_g(s)=0$, $\dot{\phi}^s_g(t)=g(t)$.
	Let $v=u-\bar{u}$.
	Compute \eqref{eq:acc_tvf} once for $u$ and then for $\bar{u}$, subtract the two PDEs, and then test the resulting equation by $\phi^s_v(t)$ to obtain:
	\begin{equation}\label{eq:test1}
	\int_0^s(\ddot{v}(t) + \eta \dot{v}(t) ,\phi^s_v(t))dt=\int_0^s(\operatorname{div}(p)- \operatorname{div}(\bar{p}),\phi^s_v(t))dt.
	\end{equation}
	Using integration by parts and the initial conditions $v(0)=\dot{v}(0)=0$, equation \eqref{eq:test1} becomes:
	\begin{equation}\label{eq:test2}
	\int_0^s \frac{d \norm{v(t)}^2}{2dt} + \eta \norm{v(t)}^2 dt= -\int_0^s (\operatorname{div}(p)- \operatorname{div}(\bar{p}),\phi^s_v(t))dt.
	\end{equation}
	Then \eqref{eq:test2} is explicitly written as
	\begin{equation}\label{eq:test3}
	\begin{array}{ll}
	\int_0^s \frac{d \norm{v(t)}^2}{2dt} + \eta \norm{v(t)}^2 dt
	&= \int_0^s\left( \operatorname{div}(p(t))-\operatorname{div}(\bar{p}(t)) ,\int_t^s v(r)dr\right)dt\\
	&= \int_0^s\left(  \operatorname{div}(p(t))-\operatorname{div}(\bar{p}(t)), (s-t) v(t+h_s)\right)dt,
	\end{array}
	\end{equation}
	with $t+h_s \in (t,s)$.
	The second equality holds thanks to the continuity of $v(t)$ and the mean value theorem.
	
	In the following, we prove by contradiction that $v\equiv 0$ over the time domain $(0,s)$.
	We first notice that because of  equation \eqref{eq:acc_tvf}, if $u(t)\neq \bar{u}(t)$ for $t\in (0,s)$, and $u_0\neq 0$, then $u(t) \neq c\bar{u}(t)$ for any nonzero constant $c$.
	As $v(0)=0$, let $t=\epsilon >0$  be the first occasion such that $v(\epsilon)\neq 0$.
	If no such $\epsilon$ exists, then we are done.
	In case $v$ is non-zero immediately after $t=0$, then we choose a sufficiently small $\epsilon >0$ such that  $v(\epsilon)\neq 0$.
	
	Then we have
	\[\int_0^\epsilon \frac{d \norm{v(t)}^2}{2dt}+ \eta \norm{v(t)}^2 dt \geq \norm{v(\epsilon)}^2/2  >0.\]
	On the other hand, since $-\operatorname{div}(p(\epsilon))\in \partial \text{TV}(u(\epsilon))$  and $-\operatorname{div}(\bar{p}(\epsilon)) \in \partial \text{TV}(\bar{u}(\epsilon))$, taking into account monotonicity of the subgradients,  we have
	\begin{equation}\label{eq:negtive}
	 \left( \operatorname{div}(p(\epsilon))-\operatorname{div}(\bar{p}(\epsilon))  , v(\epsilon)\right)
	<  0 .
	\end{equation}
	 Note that the left hand side of \eqref{eq:negtive} is a symmetric Bregman distance with respect to the $\text{TV}$ functional. For $ u(\epsilon)\neq \bar{u}(\epsilon)$, the symmetric Bregman distance equals to $0$ only if the two sets
	  $\partial \text{TV}(u(\epsilon)) $ and $ \partial \text{\text{TV}}(\bar{u}(\epsilon)) $ intersect, and $\operatorname{div}(p(\epsilon))=\operatorname{div}(\bar{p}(\epsilon)) $ in particular.
	  This means that the equation $\ddot{v}(t)+\eta \dot{v}(t)=0$ is satisfied for $t\in [0,\epsilon]$. However the homogeneous equation $\ddot{v}+\eta\dot{v}=0$ has a unique solution under the initial condition $v(0)=\dot{v}(0)=0$, that is $v(t)\equiv 0$ for all $t\in [0,\infty)$, which contradicts $ u(\epsilon)\neq \bar{u}(\epsilon)$.
	Thus the above inequality \eqref{eq:negtive} holds strictly.
	Recall that $v(t)\in C^1([0,T];H)$. Then, by continuity of $v(t)$, there exists a neighborhood $B(\epsilon,h_\epsilon):=(\epsilon-h_\epsilon,\epsilon +h_\epsilon)$ of $\epsilon$ such that for all $t\in B(\epsilon,h_\epsilon)$,  the following relation holds true:
	\begin{equation}\label{eq:test4}
	\left( \operatorname{div}(p(t))-\operatorname{div}(\bar{p}(t))) ,  v(t+h)\right) \leq 0, \text{ for all } \abs{h} \leq h_\epsilon.
	\end{equation}
	Now we return to the right-hand side of \eqref{eq:test3} with $s=\epsilon$, and find
	\[\int_0^\epsilon \left( \operatorname{div}(p(t))-\operatorname{div}(\bar{p}(t)) , (\epsilon-t) v(t+\bar{h})\right)dt  \leq  0.\]
	This implies
	\[\int_0^\epsilon  \frac{d \norm{v(t)}^2}{2dt} + \eta \norm{v(t)}^2dt \leq 0,\]
	which yields a contradiction. Therefore  $v(t)\equiv 0$ over $t\in [0,\epsilon]$.
	Note that $\epsilon$ here does not depend on the initial time. 
	Then we can repeatedly apply this procedure to the time domains $[n\epsilon,(n+1)\epsilon]$  for every $n\in \N$.
	This shows that $v(t)\equiv 0$ over $t\in [0,\infty)$.
	Thus, equation \eqref{eq:acc_tvf} admits a unique solution.
\end{proof}

Finally, we show a decay rate for the $\text{TV}$ energy when applying the second-order TVF \eqref{eq:acc_tvf} as a total variation minimizing flow. For the formulation of the results we use the Landau symbol $o(\cdot)$.
\begin{proposition}
	\label{acceleration}
	Let $u$ be the solution of the  second-order TVF \eqref{eq:acc_tvf}, then
	\[\text{TV}(u(t))=o\left(\frac{1}{t}\right)\; \text{ as } \; t\to\infty .\]
\end{proposition}
\begin{proof}
	We adopt an idea of \cite{Cabot}. Let us first introduce the auxiliary function
	\begin{eqnarray}\label{h}
	h(t): = \frac{\eta}{2} \| u(t) \|^2+ ( \dot{u}(t), u(t) ).
	\end{eqnarray}
	By elementary calculations, we derive that
	\begin{eqnarray*}
		\dot{h}(t) = \eta ( \dot{u}(t), u(t) ) + ( \ddot{u}(t), u(t)) + \| \dot{u}(t) \|^2 = \| \dot{u}(t) \|^2- ( \partial \text{TV}(u(t)), u(t)).
	\end{eqnarray*}
	Then we define the entropy functional
	\begin{equation}\label{eq:entropy_TV}
	\mathcal{E}(t):=\norm{\dot{u}(t)}^2/2+ \text{TV}(u(t)).
	\end{equation}
	Note that $\mathcal{E}(t)$ is absolutely continuous and
	 \begin{equation}\label{eq:entro_deriv}
		\frac{d\mathcal{E}(t)}{dt}=-\eta \norm{\dot{u}}^2 .
	\end{equation}
 which in combination with Lemma \ref{lem:homogenous} implies that
	\begin{eqnarray*}
		\frac{3}{2} \dot{\mathcal{E}}(t) + \eta \mathcal{E}(t)  + \eta \dot{h}(t)
		= \eta \left [ \text{TV}(u(t)) -  (\partial \text{TV}(u(t)),u(t)) \right] = 0.
	\end{eqnarray*}
	Integrating the above inequality over $[0,T]$ we obtain together with the non-negativity of $\mathcal{E}(t)$,
	\begin{equation}\label{chiIneq}
	\quad \int^T_0 \mathcal{E}(t) dt = \frac{3}{2\eta} \left( \mathcal{E}(0) - \mathcal{E}(T) \right) + (h(0)-h(T))
	\leq \left(  \frac{3}{2\eta} \mathcal{E}(0) + h(0) \right)- h(T).
	\end{equation}
	On the other hand, by Theorem \ref{thm:existence_TVF}, $u(t)$ and $\dot{u}(t)$ are uniformly bounded. Hence, there exists a constant $M$ such that $|h(t)|\leq M$ for all $t$. Letting $T\to \infty$ in (\ref{chiIneq}), we obtain
	\begin{equation}\label{chiIneq2}
	\int^\infty_0 \mathcal{E}(t)  dt < \infty.
	\end{equation}
	Moreover, since $\mathcal{E}(t)$ is non-increasing, combining with \eqref{chiIneq2} we deduce that
	\begin{equation*}
	\int^{t}_{t/2} \mathcal{E}(\tau) d\tau \geq \frac{t}{2}\mathcal{E}(t) \quad \Rightarrow \quad \lim_{t\to\infty} t\cdot \mathcal{E}(t)=0.
	\end{equation*}
	Hence, we conclude $\lim_{t\to\infty} t \text{TV}(u(t)) =0$, which yields that $\text{TV}(u(t))=o(\frac1t)$.
\end{proof}
It is known that for gradient flows (first-order) of one-homogeneous functionals, finite extinction times exist, see, e.g., \cite{BunBurChaNov20}. However, for the corresponding second-order damping flow as in \eqref{eq:acc_tvf}, this is not the case.
\begin{theorem}\label{thm:finite}
There exists no finite extinction time for the solution of \eqref{eq:acc_tvf} given nontrivial initial condition.
\end{theorem}
\begin{proof}
We prove it by contradiction. Let us take $\mathcal{E}(t)$ as in \eqref{eq:entropy_TV}, and assume $\mathcal{E}(t_0)>0$.
If there exists a finite extinction time, then it means we can find a real number $T>t_0$, such that $ \mathcal{E}(t)\equiv 0$ for all $t\geq T$.
In particular, we choose $T$ to be the smallest one, i.e.,
\begin{equation}\label{eq:finite_end}
T=\inf_{t_e}\set{ t_0<t_e \; |\mathcal{E}(t) \equiv 0\; \text{ for all }\; t\in [t_e,\infty)}.
\end{equation} 
Since both $u(t)$ and $\dot{u}(t)$ are continuous, and using \eqref{eq:entro_deriv}, we have that $\mathcal{E}(t)$ is continuous and monotonically decreasing. 
Let us choose $t_\eta\in(t_0,T)$ which satisfies $\eta(T-t_\eta)< \frac{1}{2}$, and $\mathcal{E}(t_\eta)>0$.
Then we have 
\begin{equation}\label{eq:controdiction_equ}
	\int_{t_\eta}^T \frac{d \mathcal{E}(t)}{dt}dt =- \mathcal{E}(t_\eta)=-\eta \int_{t_\eta}^T \norm{\dot{u}(t)}^2 dt,
\end{equation}
where the first equality is because of \eqref{eq:finite_end}, and the second one is using \eqref{eq:entro_deriv}.
Due to Lipschitz continuity of  $\dot{u}(t)$, we pick
\[t^*=\underset{t\in [t_\eta,T]}{\operatorname{argmax}}\quad \norm{\dot{u}(t)}^2,\; \text{ which satisfies }\; \mathcal{E}(t_\eta)\geq \mathcal{E}(t^*)\geq \frac{1}{2}\norm{\dot{u}(t^*)}^2.\]
Note that $\norm{\dot{u}(t^*)}^2>0$, otherwise we have $\norm{\dot{u}(t)}^2\equiv 0$ over $[t_\eta,T]$ which gives $\mathcal{E}(t_\eta)=0$.
However, since $\eta(T-t_\eta)< \frac{1}{2}$, we derive that
\[ -\eta \int_{t_\eta}^T \norm{\dot{u}(t)}^2 dt \geq  -\eta(T-t_\eta)\norm{\dot{u}(t^*)}^2>-\frac{1}{2}\norm{\dot{u}(t^*)}^2\geq - \mathcal{E}(t_\eta),\]
which contradicts the second equality in \eqref{eq:controdiction_equ}.
This concludes that there exists no such finite $T\in \R$.
\end{proof}
Note that the proof above does not rely on a specific potential energy, i.e. the $\text{TV}$ functional. Therefore it applies to general second-order damping flows for convex potential functions while the corresponding entropy $\mathcal{E}(t)$ is absolutely continuous.

\subsection{Analytic solution in the case of a simple initial datum}
\label{subsec:Ana_solution_stvf}
To have more intuition on the behavior of the solution of \eqref{eq:acc_tvf}, we calculate here an analytical solution to \eqref{eq:acc_tvf} given some simple initial datum. 
To do this, we consider $u_0:\Omega \to \R$ for $\Omega \subset \R^2$ to be an eigenfunction of $\partial \text{TV}$. That means $u_0\in \xi_0\partial  \text{TV}(u_0)$ which is the case when $u_0$ is the characteristic function of some calibrable set, see, for instance, the examples provided in \cite{BalCasNov02,BurGilOshXu06}.
Let $v_0=\psi_0 u_0$  for some $\psi_0 \in \R$. One can then quickly check that $u(t)=\xi(t)u_0$ is a solution of \eqref{eq:acc_tvf}, where $\xi$ solves the following ordinary differential equation
\begin{equation}\label{eq:ode_atvf}\left.
\begin{aligned}
&\ddot{\xi}(t)+ \eta \dot{\xi}(t) +\text{sign}(\xi(t))=0,\\
&\xi(0)=\xi_0,\\
& \dot{\xi}(0)=\psi_0,
\end{aligned}\right\}\;
\text{ where } \;
\text{sign}(\xi)= \left\{\begin{aligned}
& 1, & \quad  \text{ when } \xi >0;\\
& 0, & \quad  \text{ when } \xi =0;\\
&-1, &  \quad  \text{ when } \xi<0.
\end{aligned} \right.
\end{equation}
The solution of this equation can be calculated analytically, which has to be distinguished for different cases of $\text{sign}(\xi)$ due to the discontinuity. 
We start by assuming the initial value $\xi_0>0$, and compute the trajectory for the positive case ($\xi >0$), i.e., 
\[
\left\{\begin{aligned}
& \ddot{\xi}(t)+ \eta \dot{\xi}(t) +1=0;\\
& \xi(t_0)=\xi_0, \quad \dot{\xi}(t_0)=\psi_0,
\end{aligned} \right. \]
for some $\xi_0\in \R^+$, and $\psi_0\in \R$ at initial time $t_0$. 
This gives
\begin{equation}\label{eq:anal_sol_pos}
\xi(t) =-\frac{t-t_0}{\eta} -\frac{ \psi_0\eta +1}{\eta^2}\left(e^{\eta (t_0-t)} -1\right)+ \xi_0   \quad \text{ for } t\in (t_0,t_1) ,
\end{equation}
with 
\begin{equation}\label{eq:anal_der_pos}
 \dot{\xi}(t) =-\frac{1}{\eta}  + \frac{\psi_0\eta+ 1}{\eta}e^{\eta (t_0-t)}   \text{ for } t\in (t_0,t_1), 
\end{equation}
where $t_1$ is the time that $\xi(t_1)=0$. 
Using  $-\frac{t_1-t_0}{\eta} -\frac{ \psi_0\eta +1}{\eta^2}\left(e^{\eta (t_0-t_1)} -1\right)+ \xi_0=0$, we find \[ \psi_1:=\dot{\xi}(t_1)= - \frac{1}{\eta} +(\psi_0+\frac{1}{\eta})e^{\eta (t_0-t_1)} =-(t_1-t_0)+\psi_0+\eta\xi_0   .\]
Notice that if $\psi_0\leq 0$, it is clear that $\dot{\xi}(t_1)\leq 0$.
However for $\psi_0 >0$, the sign of $\dot{\xi}(t_1)$ is not immediately clear. 
We ignore such a case to simplify the discussion.
Another observation is that if $\psi_0=-\eta \xi_0$, then $\psi_1=t_0-t_1<0$.
We plot the solution to see the behaviors of the ODE under such conditions in Figure \ref{fig:ode_plot2} and \ref{fig:ode_plot3}.
This case turns out to be quite interesting in our applications and numerical experiments (see Section \ref{sec:algorithm}).
Then, after passing $t_1$, the value of $\xi(t)$ continues decreasing due to the momentum ($\dot{\xi}(t_1)<0$), and thus the sign function switches to $-1$ in \eqref{eq:ode_atvf} which becomes the equation
\[
\left\{\begin{aligned}
& \ddot{\xi}(t)+ \eta \dot{\xi}(t) -1=0;\\
& \xi(t_1)=\xi_1, \quad \dot{\xi}(t_1)=\psi_1.
\end{aligned} \right. \]
Note that $\xi_1=0$, and $\psi_1= -\frac{1}{\eta}  + \frac{\psi_0\eta+1}{\eta}e^{\eta (t_0-t_1)}<0$. 
This yields the solution 
\begin{equation}\label{eq:anal_sol_neg}
\xi(t) =\frac{t-t_1}{\eta} - \frac{\psi_1\eta-1}{\eta^2} \left( e^{\eta (t_1-t)}  -1\right) +\xi_1  \quad \text{ for } t\in (t_1,t_2),  
\end{equation}
 with 
 \begin{equation}\label{eq:anal_der_neg}  \dot{\xi}(t) =\frac{1}{\eta} +   \frac{\psi_1\eta-1}{\eta}e^{\eta (t_1-t)}  \text{ for } t\in (t_1,t_2) ,
 \end{equation}
 where $t_2$ is the occasion that $\xi$ returns to $0$ after $t_1$.
 For the inhomogeneous initial velocity $\psi_0=-\eta \xi_0$, we have $\psi_2:=\dot{\xi}(t_2)=t_2-t_1+\psi_1>0$.
This repeated process subdivides the time interval $[0,T]=\bigcup_k (T_k:= [t_k,t_{k+1}])$ for $k\in \N$ for every given $T>0$.
 The solution over $T_k$ for arbitrary $k\in \N$ will then alternate between the above two types of trajectories with respect to corresponding initial values, which can be obtained from the previous calculations. For $k\geq 1$, the initial values are typically of the form
 \begin{equation}\label{eq:velocity_relation}
\xi_{k+1}\equiv 0, \text{ and } \quad \psi_{k+1}
 =\psi_{k}e^{-\eta \abs{T_k}} +\frac{ (-1) ^{k+1}}{\eta} (1-e^{-\eta \abs{T_k}} ) ,
 \end{equation}
where $\abs{T_k}=t_{k+1}-t_k$.
Note that $\psi_{k}$ is alternating between negative and positive values for $k\geq 1$.
In the case of $\psi_0=-\eta \xi_0$,  we have $\psi_{k+1}=\sum_{i=0}^k (-1)^{i+1} \abs{T_k}$ for all $k\geq 0$. 
Next we further look into the two types of explicit trajectories. First, let us consider the positive case, i.e.,
\eqref{eq:anal_sol_pos} and \eqref{eq:anal_der_pos} with $\xi_0\geq 0$. Assume $t>t_0$ be the time spot such that $\xi(t)=\dot{\xi}(t)=0$.
This then requires that the following system of equations has a real-valued solution:
\begin{equation*}\left\{
	\begin{aligned}
		-t+t_0 + \psi_0 +\eta \xi_0	=0,\\
		-1+ (\psi_0\eta+ 1)e^{\eta (t_0-t)}  =0,
	\end{aligned}\right.
\end{equation*}
where the first equation is based on some rearrangement of the terms in  \eqref{eq:anal_sol_pos} according to the second equation. Through some algebra computation,  we find that the above system admits \textbf{no} real solution for $t>t_0$ except the case that both $\psi_0=0$ and $\xi_0=0$ are satisfied.
Now we turn to the negative case in \eqref{eq:anal_sol_neg} and \eqref{eq:anal_der_neg} with $\xi_1\leq 0$. Similarly, this asks for a real-valued solution to the system
\begin{equation*}\left\{
	\begin{aligned}
		t-t_1 + \psi_1 +\eta \xi_1	=0,\\
		1 + (\psi_1\eta-1)e^{\eta (t_1-t)}  =0,
	\end{aligned}\right.
\end{equation*}
which again admits a real solution only when $\psi_1=0$ and $\xi_1=0$.
This shows that there are infinitely many zeros ($\psi_n=0$) on the trajectory of \eqref{eq:ode_atvf} if the initial values are not zero, which is not surprising since there is no finite extinction time.

We plot in Figure \ref{fig:ode_plot1} some trajectories of \eqref{eq:acc_tvf} with $\xi_0=1$, $\psi_0=0$ used for $\eta=10$, $\eta=1$  and $\eta=0.1$ over a fixed time interval, as well as their first- and second-order derivatives, respectively. 
We notice that $\xi$ is a continuous function with oscillations, and its first-order derivative $\dot{\xi}$ is Lipschitz continuous, while the second-order derivative $\ddot{\xi}$ is piecewise continuous. The behaviors of the flow are distinguished with respect to different $\eta$. Larger $\eta$ results in more frequent oscillations.
\begin{figure}[ht!]
	\includegraphics[width=0.327\textwidth]{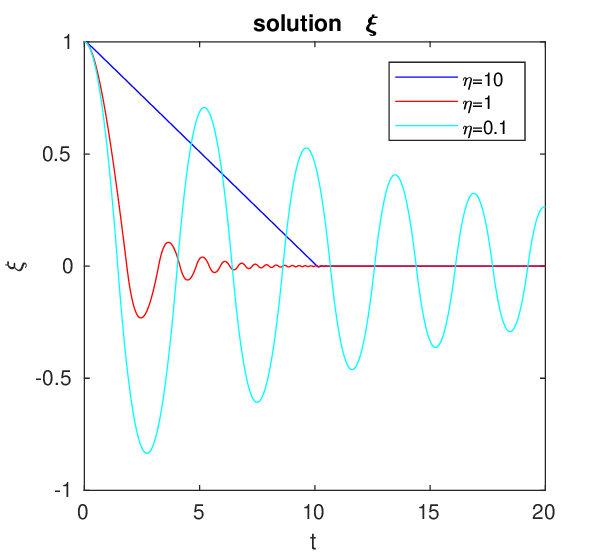} 
	\includegraphics[width=0.327\textwidth]{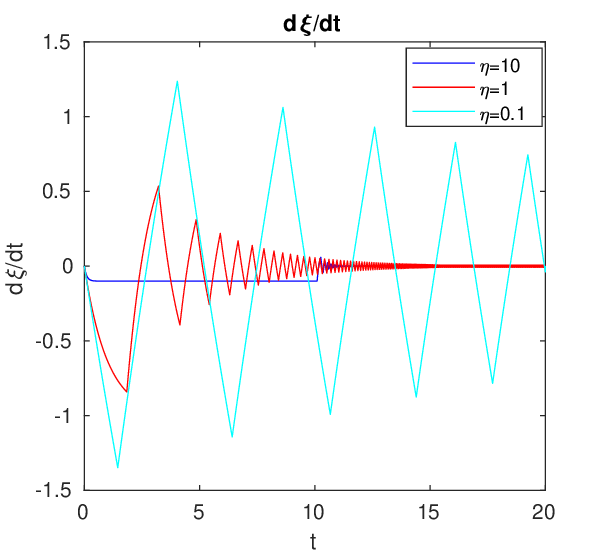} 
	\includegraphics[width=0.327\textwidth]{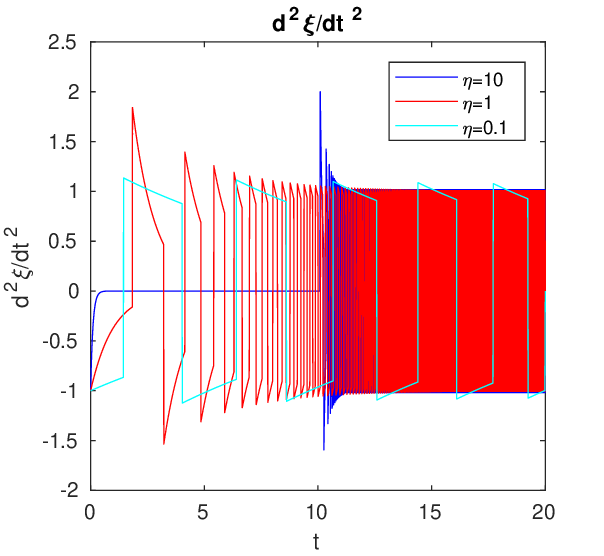} 
	\caption{Trajectories of equation \eqref{eq:ode_atvf} with respect to $\eta=10$ (blue) $\eta=1$ (red) and $\eta=0.1$ (cyan) over the time interval $[0,20]$ with homogeneous initial velocity. The vertical lines in the right figure show the discontinuity of the plotted functions.  }
	\label{fig:ode_plot1}
\end{figure}
The initial velocity $\dot{\psi}(0)$ significantly influences the solution trajectory. 
In Figure \ref{fig:ode_plot2}, we plot the solutions of the second order ODEs under different values for $\eta$ with respect to the initial velocity $\psi_0=0$ and the initial velocity $\psi_0=-\eta \xi_0$, respectively, and also compare them with the first-order flow using the same initial value $\xi_0=255$.
\begin{figure}[ht!]
	\centering
	\includegraphics[width=0.44\textwidth]{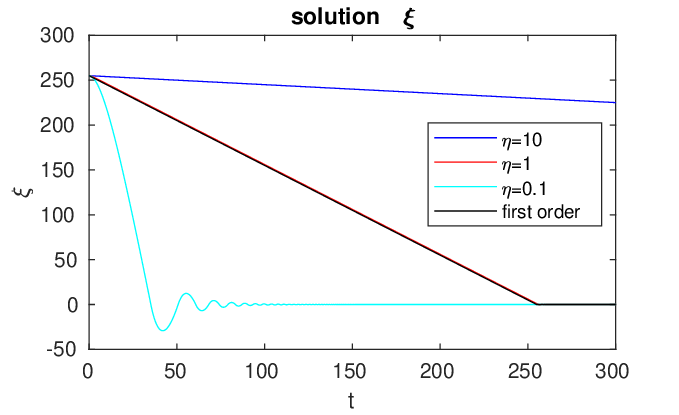} \quad
	\includegraphics[width=0.44\textwidth]{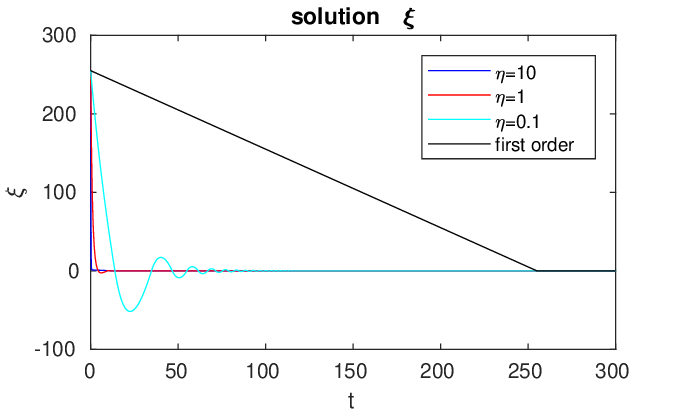} 
	\caption{Trajectories of equation \eqref{eq:ode_atvf} with respect to initial velocities $\psi_0=0$ (left) and $\psi_0=-\eta \xi_0$ (right), respectively, for $\xi_0=255$ over the time interval $[0,300]$. Different $\eta$ values as in Figure \ref{fig:ode_plot1} are compared also with the first-order flow (black color).}
	\label{fig:ode_plot2}
\end{figure}
Note that the decaying behavior for the cases of $\eta=10$ and $\eta=1$ are quite different in case of the homogeneous versus inhomogeneous initial velocity.
In fact, notice that $\psi_0=-\eta \xi_0$ accelerates the decay of the magnitude of $\xi$. This is in particular true for $\eta=10$ and $\eta=1$, which distinguishes them from the first-order flow and homogeneous initial velocity. We provide more comparisons in Figure \ref{fig:ode_plot3} of the initial velocity $\psi_0=-\eta \xi_0$ with some other initial velocities, where we can see that the trajectory with $\psi_0=-\eta \xi_0$ stabilizes around $0$ much faster than the other conditions in the cases of $\eta=1$ and $\eta=10$.
\begin{figure}[ht!]
	\centering
	\includegraphics[width=0.327\textwidth]{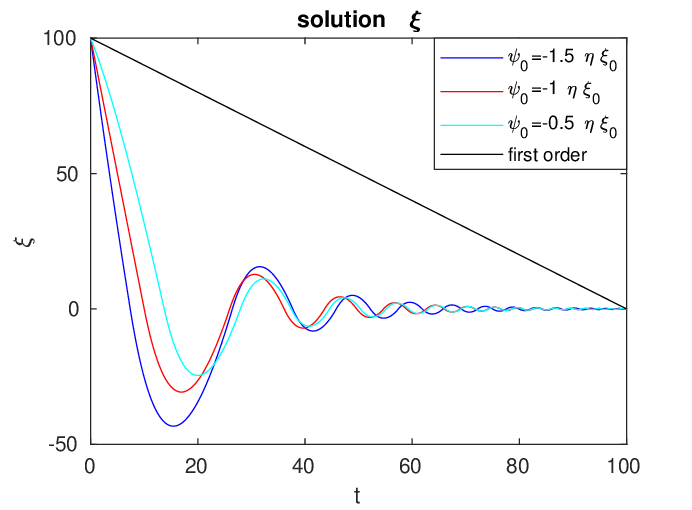} 
	\includegraphics[width=0.327\textwidth]{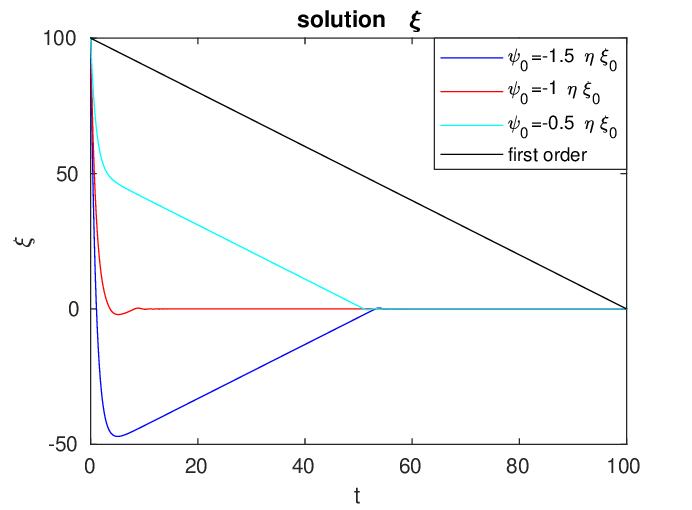} 
	\includegraphics[width=0.327\textwidth]{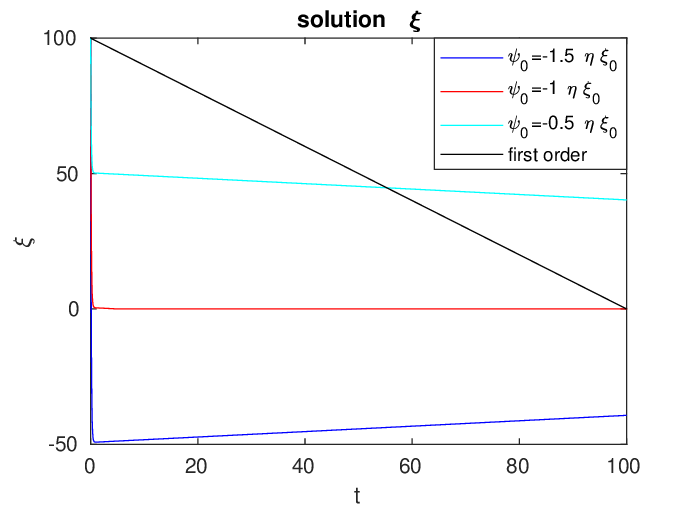}
	\caption{Trajectories of equation \eqref{eq:ode_atvf} with respect to inhomogeneous initial velocity $\psi_0=-1.5\eta \xi_0$, $\psi_0=-\eta \xi_0$ and $\psi_0=-0.5\eta \xi_0$ for $\eta=0.1$ (left), $\eta=1$ (middle) and $\eta=10$ (right) over the time interval $[0,100]$ corresponding to $\xi_0=100$, respectively.}
	\label{fig:ode_plot3}
\end{figure}
This phenomenon has been further studied in Section \ref{sec:algorithm}.
Another interesting case is when $u_0\in \xi_0 \partial \text{TV}(u_0)$, $v_0\in \psi_0 \partial \text{TV}(v_0)$, but $u_0$ and $v_0$ belong to different families of eigenfunctions. This is not discussed here but we show some numerical evidence in Section \ref{subsec:numerical}. 

We remark here that the first-order TV flow has been used in nonlinear spectral decomposition, see e.g., \cite{Gil14,BurGilMoeEckCre16,BunBurChaNov20}. In fact, its second-order derivative of the solution gives a pulse, which defines a nonlinear spectral representation of the signal, see \cite{Gil14} for details. There, the pulse is interpreted as some type of signal spectral which can identify the eigenfunctions for representing the signal.
However, for the second-order flow, $\ddot{\xi}(t) \to  -\text{sign}(\xi)$ is alternating before $\xi$ dying out to $0$ since $\dot{\xi}\to 0$ as $t\to \infty$. Its third order derivative gives countably many pulses, which might also be able to identify signal spectral. However, the details are still subject to further investigation.
We thus conclude our theoretical study on the second-order TVF \eqref{eq:acc_tvf}.
In the next section, we will study another family of nonlinear flows which are also able to decrease the total variation of a function, albeit in a somewhat different manner.
It is motivated from the application of correcting displacement errors, which is different to TVFs.

\section{Mean curvature motion of level sets}
\label{sec:mcf}
Mean curvature flow for level sets of scalar functions has been analyzed first in \cite{CheGigGot91,EvaSpr91}. The associated equation reads
\begin{equation}
\label{eq:mcf_u}
\left\{\begin{array}{ll}
\dot{u}(t)   = \abs{\nabla u(t)}\operatorname{div}\left(\frac{\nabla u(t)}{\abs{\nabla u(t)}} \right), & \text{ in } \R^n \times (0,\infty );\\
u(0)=u_0, & \text{ in } \R^n \times 0.
\end{array}
\right.
\end{equation}
In such a form, the flow can overcome the singularity and can be extended after topological changes which may be generated during the evolution of standard hypersurface mean curvature flow.
It finds many applications in surface processing and also in image processing.
A particular application which has become a research focus recently is concerned with correcting displacement errors in image data \cite{LenSch11}.
Also, a connection between the level-set MCF and a non-convex energy functional has been identified in \cite{ElbGraLenSch10}.

As described in the introduction, a displacement error in image data can be mathematically modeled as follows:
\begin{equation*}
u^d(x):=u(x+d(x)), \text{ for }d: \R^n \rightarrow \R^n,\; \text{ and } \norm{d}_{L^\infty} \leq M,
\end{equation*}
where $u^d: \R^n \to \R$ is the measured image, $M$ is some positive real number, and $u$ is the ideal physical acquisition of the image.
Assuming that the magnitude of the error bound $M$ is small, following \cite{LenSch11,DonPatSchOek15,DonSch17} we may consider a first-order Taylor expansion of the function $u$ along the normal direction of the level sets of $u$:
\begin{equation}
\label{eq:taylor}
u^d(x) = u(x+d(x)) \approx u(x) + \abs{d(x)}\left\langle\frac{\nabla u(x)}{\abs{\nabla u(x)}}, \nabla u(x) \right\rangle ,
\end{equation}
assuming $\abs{\nabla u(x)} >0$.
Then the magnitude of the displacement error can be approximated as follows:
\begin{equation}
\label{eq:disp}
\abs{d(x)} \approx \frac{\ud(x)-u(x)}{\abs{\nabla u(x)}}.
\end{equation}
In \cite{LenSch11}, a generalized total variation regularization is employed to recover $u$ given $\ud$, which leads to a non-convex and non-smooth energy functional
\begin{equation}\label{eq:non_convex}
\mathcal{E}(u;\ud) := \frac{1}{2}\int_{\R^n}
\frac{(u(x)-\ud(x))^2}{\abs{\nabla u(x)}^q}dx
+\alpha  \int_{\R^n}\abs{\nabla u(x)} dx,
\end{equation}
where $\alpha>0$ is a regularization parameter.
The parameter $q\in (0,2]$ is introduced in order to simultaneously control the displacement error $d(x)$ and also the intensity error $\delta(x)$ in the image data.
In the case of $q=1$ we observe that the first term of $\mathcal{E}(u;\ud)$ in \eqref{eq:non_convex} is a measure reflecting both the displacement error  and the intensity error by using the geometric mean of $\frac{(u(x)-\ud(x))^2}{\abs{\nabla u(x)}^2}$ and $(u(x)-\ud(x))^2$.

Using formal calculations, e.g., the semi-group techniques of \cite{ElbGraLenSch10,LenSch11,SchGraGroHalLen09} or the semi-implicit iterative scheme in \cite{DonSch17} by identifying $\alpha$ as a discrete time step (the latter is  analogous to our explanation involving the ROF model and the first-order TVF in the previous section), we formally derive the following nonlinear flows:
\begin{equation} \label{eq:formal_flow}
\left\{ \begin{array}{lll}
\dot{u} (t)&= \abs{\nabla u(t)}^q \operatorname{div} \left(\frac{\nabla u(t)}{\abs{\nabla u(t)}} \right)&
\text{ in } \R^n \times (0,\infty),\\
u(0) &=\ud  & \text{ in } \R^n  \times 0.
\end{array}\right.
\end{equation}
We can see that \eqref{eq:mcf_u} emerges for $q=1$ in \eqref{eq:formal_flow}.
In \cite{LenSch11,DonPatSchOek15,DonSch17}, it was documented that the nonlinear flows \eqref{eq:formal_flow} are able to  correct small displacement errors and also to denoise image data where the discrete time step and the stopping time play the roles of regularization parameters.
This motivates us to consider a second-order damping flow based on the first-order flow, which is the equation  below:
\begin{equation}
\label{eq:acc_mcf_u}
\left\{  \begin{array}{lll}
\ddot{u}(t) +\eta \dot{u}(t) & = \abs{\nabla u(t)}\operatorname{div}\left(\frac{\nabla u(t)}{\abs{\nabla u(t)}} \right),& \text{ in } \R^n\times (0,\infty),\\
u( 0)=u_0 , & \quad \dot{u}(0)=v_0 \;& \text{ in } \R^n  \times 0.
\end{array}
\right.
\end{equation}
We refer to this new equation as the damped second-order level-set MCF.

\subsection{Heuristic observation on  the damped second-order level-set MCF}\label{SecHeuristics}
Suppose that each level set of the function $u$ is a hypersurface which is well-defined in $\R^n$.
In this case, it has been verified in \cite{EvaSpr91} that the level-set MCF \eqref{eq:mcf_u} is equivalent to the gradient flow of the volume (perimeter) functional  for  the hypersurface of every single level set, that is the standard hypersurface mean curvature flow.
More precisely, let $\Gamma(t)$ be the immersion of the hypersurface into $\R^n$. Without loss of generality, we consider it to be the zero level set of $u(t)$ that is $u(\Gamma(t),t)\equiv 0$.
Assume $\Gamma(t)$ to be smooth. Then the evolution of $\Gamma(t)$ governed by the first-order equation \eqref{eq:mcf_u} is in fact characterized by the following hypersurface mean curvature flow:
\begin{equation}
\label{eq:mc_flow_ls}
\left\{\begin{array}{lll}
\dot{\Gamma}(t)   =-\mathcal{H} \nu(t)  \;,\\
\Gamma(0) = \Gamma_0,
\end{array}
\right.
\end{equation}
where $\nu(t)$ is the outer unit normal vector associated to the hypersurface of the level set $\Gamma(t)$, and $\mathcal{H} $ is the mean curvature of $\Gamma(t)$.
Note here that $\mathcal{H} \nu(t)=\partial \mathcal{V}(\Gamma(t))$, where $\partial \mathcal{V}$ denotes  the functional gradient of the volume functional $\mathcal{V}$ (e.g., the length of the level set curves in a two dimensional setting). Particularly, for $n$-dimensional hypersurfaces, $\mathcal{H} =-\frac{1}{n} \operatorname{div}\left(\nu(t) \right) $. In the following, we stick to this expression in our discussion, and ignore the constant $\frac{1}{n}$.

The mean curvature flow \eqref{eq:mc_flow_ls} for hypersurfaces or  general manifolds has been a central topic in geometric analysis.
In the level set setting, if the spatial gradient $\nabla u(\Gamma(t),t) \neq 0$, the normal field of the hypersurface of every level set can be represented by $\nu(t)=\frac{\nabla u(\Gamma(t),t)}{\abs{\nabla u(\Gamma(t),t)}}$.

In this context, a relevant question is connected to identifying an evolutionary equation for the hypersurfaces given by the level sets of $u(t)$ associated to the second-order level-set MCF \eqref{eq:acc_mcf_u}.
In the following, we give some heuristics based on formal calculations.

Let us again take $\Gamma(t)$ to be the immersion of the zero level set of the function $u(t)$, and consider the following equation:
\begin{equation}
\label{eq:acc_mcf}
\left\{\begin{array}{lll}
\ddot{\Gamma}(t) +  \left( \eta \operatorname{Id}+\nu(t)\otimes \frac{\nabla\dot{u}(\Gamma(t),t)}{\abs{\nabla u(\Gamma(t),t)}} \right) \dot{\Gamma}(t)  = -  \operatorname{div}\left(\nu(t) \right) \nu(t)  \;;\\
\dot{\Gamma}(0)  =  \gamma_0, \;\quad \Gamma(0) =  \Gamma_0,
\end{array}
\right.
\end{equation}
where $\operatorname{Id}$ represents the $n\times n$ identity matrix, and $\otimes$ denotes the tensor product of vectors, and $\Gamma_0$ is a close hypersurface.
While the tensor product term may appear surprising in the context of \eqref{eq:acc_mcf} at first glance, its role will soon become clear.
Using the fact that $\nu$ is a unit normal vector field, we have
\[P_\tau \frac{\nabla\dot{u}(\Gamma(t),t)}{\abs{\nabla u(\Gamma(t),t)}}= \dot{\nu}(t),\]
where $P_\tau$ is the projection operator onto the tangent space of $\Gamma(t)$.

Now, we look for the connection between \eqref{eq:acc_mcf} and \eqref{eq:acc_mcf_u} for the evolution of the level sets of the function $u$.
We first notice that $u(\Gamma(t),t) \equiv 0$ (or any other constant) which gives
\begin{equation}
\label{eq:total_time_der_u}
\dot{u}(\Gamma(t),t)=-\langle\nabla u(\Gamma(t),t),\dot{\Gamma}(t)\rangle.
\end{equation}
Differentiating with respect to time on both sides of \eqref{eq:total_time_der_u}, we get:
\begin{equation}
\label{eq:derivative}
\ddot{u}(\Gamma(t),t) =  -\langle (\partial_t \nabla u(\Gamma(t),t) ) , \dot{\Gamma}(t)\rangle - \langle\nabla u(\Gamma(t),t), \ddot{\Gamma}(t)\rangle.
\end{equation}
Note here that we are not calculating the total time derivative of $u$ but rather the partial derivative with respect to $t$.
Since $\Gamma(t)$ now follows the trajectory given by \eqref{eq:acc_mcf}, we observe that
\[\left(\nu(t)\otimes \frac{\nabla \dot{u}(\Gamma(t),t)}{\abs{\nabla u(\Gamma(t),t)}}\right) \dot{\Gamma}(t)=\left\langle\frac{\nabla \dot{u}(\Gamma(t),t)}{\abs{\nabla u(\Gamma(t),t)}},\dot{\Gamma}(t)\right\rangle \nu(t), \]
leading to
\begin{equation}
\label{eq:geometry_equ}
\begin{aligned}
\ddot{u}(\Gamma(t),t) & =  -\langle(\partial_t \nabla u(\Gamma(t),t) )  , \dot{\Gamma}(t)\rangle -\eta\dot{u}(\Gamma(t),t) +\langle\nabla u(\Gamma(t),t),\operatorname{div}\left(\nu(t) \right)  \nu(t) \rangle \\
& \quad +\left\langle \frac{\nabla \dot{u}(\Gamma(t),t) }{\abs{\nabla u(\Gamma(t),t) }}, \dot{\Gamma}(t)\right\rangle \langle\nabla u(\Gamma(t),t)  , \nu(t)\rangle  ,
\end{aligned}
\end{equation}
where we also use \eqref{eq:derivative}.
Using the fact that $\nu(t)=\frac{\nabla u(\Gamma(t),t) }{\abs{\nabla u(\Gamma(t),t) }} $ and $\abs{\nabla u(\Gamma(t),t) } \neq 0$ we verify that
\[
\left\langle\frac{\nabla \dot{u}(\Gamma(t),t) }{\abs{\nabla u(\Gamma(t),t) }}, \dot{\Gamma}(t)\right\rangle \langle\nabla u(\Gamma(t),t) , \nu(t)\rangle=\langle \nabla \dot{u}(\Gamma(t),t)  , \dot{\Gamma}(t)\rangle \;.
\]
Assuming for the moment that $u$ has sufficient regularity and interchanging the order of the time and spatial derivatives in the first term of the right hand side of \eqref{eq:geometry_equ}, that is $\dot{\nabla u}(\Gamma(t),t) =(\partial_t \nabla u(\Gamma(t),t) )$,  equation \eqref{eq:geometry_equ} turns into \eqref{eq:acc_mcf_u} restricted to the level set $\Gamma(t)$, i.e.,
\[\ddot{u}(\Gamma(t),t) +\eta \dot{u}(\Gamma(t),t)  = \abs{\nabla u(\Gamma(t),t)}\operatorname{div}\left(\frac{\nabla u(\Gamma(t),t)}{\abs{\nabla u(\Gamma(t),t)}} \right).\]
This indicates that every smooth level set $\Gamma(t)$ of the solution of \eqref{eq:acc_mcf_u} evolves according to equation \eqref{eq:acc_mcf}.
Basically, \eqref{eq:acc_mcf} is a vectorial form of second-order dynamics for the mean curvature flow of hypersurfaces. However, the damping coefficient has a matrix form and involves the external function $u$. This shows that \eqref{eq:acc_mcf} is not an independent geometric PDE. Rather it needs to be coupled to  \eqref{eq:acc_mcf_u}.
This is further expanded in the following remark.
\begin{remark}\label{rem:monotonical_geometric_flow}
Consider following second-order geometric flow for general smooth hypersurfaces of some immersion function $\Gamma(t)$:
	\begin{equation}
	\label{eq:second_level_set}
	\ddot{\Gamma}(t) +  \left( \eta \operatorname{Id}+\nu(t)\otimes \dot{\nu}(t) \right)\dot{\Gamma}(t)  = -  \operatorname{div}\left(\nu(t) \right) \nu(t).
	\end{equation}
	Define an entropy (Lyapunov function) for \eqref{eq:second_level_set} through
	\begin{equation}
	\label{eq:entropy}
	\mathcal{M}(t): = \frac{1}{2}\norm{\langle \nu(t),\dot{\Gamma}(t)\rangle}^2 + \mathcal{V}(\Gamma(t)),
	\end{equation}
	where $\mathcal{V}(\Gamma(t))$ presents the volume functional of the hypersurface, and $\nu(t)$ is the outward unit normal field of the hypersurface.
	Since
	\begin{eqnarray*}
		\frac{d \mathcal{M}(t)}{d t} = \left(\langle \nu(t),\dot{\Gamma}(t)\rangle   ,  (\langle \dot{\nu}(t) ,\dot{\Gamma}(t)\rangle + \langle \nu(t) ,\ddot{\Gamma}(t)\rangle)\right) + \left( \operatorname{div}(\nu(t))\nu(t),\dot{\Gamma}(t) \right),
	\end{eqnarray*}
	taking into account \eqref{eq:second_level_set} and by direct calculations, we deduce that
	\[\frac{d \mathcal{M}(t)}{d t} =-\eta  \norm{\langle \nu(t), \dot{\Gamma}(t)\rangle }^2\leq 0.\]
	This shows that the entropy $\mathcal{M}(t)$ is monotonically decreasing following the trajectory of the flow \eqref{eq:second_level_set}.
	Now assume $\Gamma(t)$ again to be the hypersurfaces of the level sets of a scalar function $u$.
	Isolating the term $\ddot{\Gamma}$ in \eqref{eq:second_level_set} and inserting it into \eqref{eq:derivative}, and also taking into account \eqref{eq:total_time_der_u} we derive a new equation corresponding to \eqref{eq:second_level_set} as follows:
	\[\ddot{u}(\Gamma(t),t) +\left(\eta - \frac{\langle\nabla u ,\nabla \dot{u}\rangle(\Gamma(t),t)}{\abs{\nabla u}^2(\Gamma(t),t)} \right) \dot{u}(\Gamma(t),t)  = \abs{\nabla u(\Gamma(t),t)}\operatorname{div}\left(\frac{\nabla u(\Gamma(t),t)}{\abs{\nabla u(\Gamma(t),t)}} \right) .\]
	Formally, for $\nabla u \neq 0$, this suggests the following equation for the scalar function $u$
	\begin{equation}
	\label{eq:new_acc_mcf_u}
	\ddot{u}  +\left(\eta -\frac{1}{2}\partial_t \log (\abs{\nabla u}^2)\right) \dot{u}  = \abs{\nabla u}\operatorname{div}\left(\frac{\nabla u}{\abs{\nabla u}} \right).
	\end{equation}
	Equations \eqref{eq:new_acc_mcf_u} and \eqref{eq:second_level_set} appear novel and they seem to be geometrically meaningful to study.
	The new equation \eqref{eq:new_acc_mcf_u} looks rather more complicated than \eqref{eq:acc_mcf_u}.
	Since our motivation here is to develop algorithms for image applications, we will skip detailed discussions on the equations \eqref{eq:new_acc_mcf_u} and \eqref{eq:second_level_set}, but rather focus on \eqref{eq:acc_mcf_u} in this paper.
We notice that hyperbolic mean curvature flow for hypersurfaces without the damping term has been pioneered as open problem in \cite{Yau00}. Later in \cite{LefSmo08,HeKonLiu09}, its local well-posedness has been studied, and in \cite{GinSva16}, its numerical simulation using level set approach has been implemented. 
\end{remark}
Using the entropy \eqref{eq:entropy} and following the orbit of equation \eqref{eq:acc_mcf},  we infer
\begin{eqnarray*}
	\frac{d\mathcal{M}(t)}{dt}
	&=&-\eta \norm{\langle \nu(t), \dot{\Gamma}(t)\rangle }^2 - \left(\langle \nu(t),\dot{\Gamma}(t)\rangle,\langle P_\nu \frac{\nabla \dot{u}(\Gamma(t),t)}{\abs{\nabla u(\Gamma(t),t)}} , \dot{\Gamma}(t)\rangle \right)\\
	&=& -\eta \norm{\langle \nu(t), \dot{\Gamma}(t)\rangle }^2 - \frac12 \left(\partial_t \log(\abs{\nabla u(\Gamma(t),t)}^2)\langle \nu(t),\dot{\Gamma}(t)\rangle,\langle \nu(t),\dot{\Gamma}(t)\rangle \right),
\end{eqnarray*}
where $P_\nu$ is the normal projection operator onto the hypersurface $\Gamma(t)$.
The last term makes the monotonicity of $\mathcal{M}$ unclear.
It implies that large $\eta$ are preferred for monotonicity, a practical point which we pick up in Section \ref{sec:algorithm} along with the algorithmic development.

\subsection{On the solvability of the damped second-order level-set MCF}
In order to study the solvability of the equation \eqref{eq:acc_mcf_u}, we rewrite it to obtain the following explicit form (note that $\nabla u(t)=(u_{x_1}(t), u_{x_2}(t))^\top$):
\begin{eqnarray}\label{eq:acc_mcf_u_ex}
\left\{\begin{array}{ll}
\ddot{u}(t)  +\eta \dot{u}(t)   = \sum_{i,j}\left( \delta_{ij} - \frac{u_{x_i}(t)u_{x_j}(t)}{\abs{\nabla u(t)}^2} \right) u_{x_i x_j}(t), &\textrm{~in~} \R^n\times (0,\infty), \\
\dot{u}(0)=v_0,\quad  u(0)=u_0,  & \textrm{~in~} \R^n\times 0,
\end{array}\right.
\end{eqnarray}
where $\delta_{ij}$ is the Kronecker delta function, i.e., $\delta_{ij}=1$ for $i=j$, and $\delta_{ij}=0$ for $i\neq j$.
For this problem, because of its geometric meaning, it is natural to study the flow in the domain $\R^n\times (0,\infty)$ instead of $\Omega\times (0,\infty)$ as in the TVF case with bounded $\Omega$.
For the latter, that is $u_0$ is compactly supported in $\R^n$, the results developed below will still hold by imposing Neumann boundary conditions on $\partial \Omega$ for sufficiently regular boundary.
As the right hand sides of \eqref{eq:acc_mcf_u} and \eqref{eq:acc_mcf_u_ex}, respectively, are not related to gradient (or subgradient) of a convex functional, the standard techniques using test functions are not applicable. Thus our previous approach to the second-order TVF is not suitable for this problem.
Also, the concept of the viscosity solution, which has been developed for the first-order level-set MCF \cite{EvaSpr91}, is not applicable either because of the degenerate hyperbolic structure of the equation \eqref{eq:acc_mcf_u_ex}.
Moreover, it can be checked that the nonlinear coefficients in \eqref{eq:acc_mcf_u_ex}, namely for $\abs{p}>0$,
\[\bar{a}_{i,j}(p):= \left( \delta_{ij} - \frac{p_ip_j}{\abs{p}^2} \right), \quad p_i:= u_{x_i},\]
satisfy $\sum_{i,j}\bar{a}_{i,j}(p)\zeta_i\zeta_j\geq 0$ for almost all $\zeta=(\zeta_1,\cdots,\zeta_n)^\top\in \R^n$.
Therefore  \eqref{eq:acc_mcf_u_ex} is a fully degenerate hyperbolic equation (sometimes also referred to as weakly hyperbolic) in the domain $\R^n\times (0,\infty)$.

The singularity of $\bar{a}_{i,j}$ at $p=0$ is another issue which has to be considered.
For the purpose of avoiding singularities and also to eliminate the degeneracy in  the equation \eqref{eq:acc_mcf_u_ex}, we construct a regularized version which is quasilinear but strictly hyperbolic.
This is also motivated by the numerical realization of \eqref{eq:acc_mcf_u_ex} from a practical algorithmic point of view.

\subsection{Solution of a regularized equation}
We concentrate on the following quasilinear but strictly hyperbolic equation as an approximation of \eqref{eq:acc_mcf_u_ex}:
\begin{eqnarray}\label{eq:acc_mcf_u_exApproximate}
\left\{\begin{array}{ll}
\ddot{u}^{\epsilon}(t)  +\eta \dot{u}^{\epsilon}(t)   =\sum_{i,j} \left( \delta_{ij} - \frac{u^{\epsilon}_{x_i}(t) u^{\epsilon}_{x_j}(t) }{ \abs{\nabla u^{\epsilon}(t) }^2+ \epsilon^2} \right) u^{\epsilon}_{x_i x_j}(t), & \textrm{~in~} \R^n\times (0,\infty), \\
\dot{u}^{\epsilon}(0)=v_0, \quad u^{\epsilon}(0)=u_0, & \textrm{~in~} \R^n\times 0,
\end{array}\right.
\end{eqnarray}
where $0<\epsilon  \ll 1$ is fixed.

The approximation \eqref{eq:acc_mcf_u_exApproximate} can be interpreted  as follows. Consider the function $w^{\epsilon}(y,t):= u^{\epsilon}(x,t) - \epsilon x_{n+1}$, where $y=(x,x_{n+1}) \in \R^{n+1}$. Since $|\nabla w^{\epsilon}|^2 = |\nabla u^{\epsilon}|^2 + \epsilon^2$, the equation in \eqref{eq:acc_mcf_u_exApproximate} becomes
\begin{eqnarray*}
	\left\{\begin{array}{ll}
		\ddot{w}^{\epsilon} (t) +\eta \dot{w}^{\epsilon} (t)  = \sum_{i,j}\left( \delta_{ij}  - \frac{w^{\epsilon}_{y_i}(t) w^{\epsilon}_{y_j}(t)}{\abs{\nabla w^{\epsilon}(t) }^2 }\right) w^{\epsilon}_{x_i x_j}(t), & \textrm{~in~} \R^{n+1}\times (0,\infty), \\
		\dot{w}^{\epsilon}(0)=v_0 , \quad w^{\epsilon}(0)=w^{\epsilon}_0, & \textrm{~in~} \R^{n+1}\times 0,
	\end{array}\right.
\end{eqnarray*}
where $w^{\epsilon}_0(y)=u_0(x)- \epsilon x_{n+1}$.
A geometric meaning for this approximation of first-order level-set MCF has been given in \cite{EvaSpr91}.
There, it is depicted that $v^\epsilon$ is a function defined on a higher dimensional domain, whose zero-level set is a graph given by \[\Gamma^\epsilon(t)=\set{y=(x,x_{n+1})|x_{n+1}=u^{\epsilon}(x,t) /\epsilon}.\]
Then it is argued that the complicated and possibly singular evolution of the level sets $\Gamma(t)$ of $u$ is approximated by a family of well behaved smooth evolutions of the level sets $\Gamma^\epsilon(t)$ of a function $v^\epsilon$ from a higher dimensional space, in the sense that $\Gamma^\epsilon(t) \simeq \Gamma(t)\times \R$ for sufficiently small $\epsilon$ at given $t>0$. We adopt the same geometric intuition for the solution of \eqref{eq:acc_mcf_u_exApproximate} as \cite{EvaSpr91} did for the first-order level-set MCF. 
This observation justifies  the use of such a regularization properly approximating the original solution when $\epsilon$ is small.

To study the existence and uniqueness of solutions to \eqref{eq:acc_mcf_u_exApproximate}, we rely on the results on linear hyperbolic equations (e.g., those in \cite{Lax06}). In particular, we consider the following equation
\begin{eqnarray}\label{eq:linear_hyper}
\left\{\begin{array}{ll}
\ddot{\zeta}(t)  +\eta \dot{\zeta}(t)   -\sum_{i,j} a_{i,j}(\nabla w(t)) \zeta_{x_i x_j}(t) =f(t), & \textrm{~in~} \R^n \times (0,\infty), \\
\dot{\zeta}(0)=\zeta_1, \quad \zeta(0)=\zeta_0, & \textrm{~in~} \R^n \times 0,
\end{array}\right.
\end{eqnarray}
where $w(t)$ is a given function.

To simplify notations, we write $H^k:=W^{k,2}(\R^n)$ for $k\in \N$, and
use $D^k$ to represent all derivatives with respect to temporal and spatial variables of differential orders between $[0, k]$.
In the following we summarize assumptions for existence results on linear hyperbolic PDEs. 
\begin{assumption}\label{assu:linear_hyper}
	\begin{itemize}
		\item[(i)] The functions $a_{i,j}$ are smooth and satisfy $a_{i,j}=a_{j,i}$ for $i,j \in \set{1,\cdots,n}$. Moreover there exists a continuous function $a(p)>0$ such that $a(p)\abs{q}^2 \leq  \sum_{i,j} a_{i,j}(p)q_iq_j\leq \sigma_1 \abs{q}^2 $ for  all $p\in \R^n$ and $q\in \R^n$, for some constant $\sigma_1>0$.
		\item[(ii)]  $w\in C^0([0,\infty);H^k)$, and the initial data are properly bounded, i.e. \[\sum_{l=1}^k \norm{D^l \zeta(0)}_{L^2}\leq M_0, \quad \text{ for some }  M_0>0.\]
	\end{itemize}
\end{assumption}
The first assumption implies that, for all bounded $p\in \R^n$ with $\norm{p}\leq P_0$, there exists a fixed $\sigma>0$ such that $a(p)\abs{q}^2\geq \sigma \abs{q}^2$ for all $q\in \R^n$. Note that $\zeta_0\in H^k$ and $\zeta_1\in H^{k-1}$ have been integrated into $D^k \zeta(0)$ because of the second assumption.
The existence and uniqueness of the solution for linear hyperbolic equations of type \eqref{eq:linear_hyper} have been established in \cite[Chapter 5]{Dio62}. One may also refer to \cite{Lax06,Eva10} for more results. We summarize the existence, uniqueness and energy estimate here:
\begin{theorem}\label{thm:linear_hyper}
	Let Assumption \ref{assu:linear_hyper} hold, and
	\[f \in C^s([0,\infty);H^{k-s-1})\quad  \text{ for }  k\in \N  \text{ and } s\in [0,k-1] ( k\geq s+1\geq 1).\]
	 Then the linear  hyperbolic equation \eqref{eq:linear_hyper} admits a unique solution $\zeta \in C^0([0,\infty);H^k)\cap  C^s([0,\infty);H^{k-s})$
	and the following estimate holds true:
	\begin{equation}\label{eq:hyper_est}
	\norm{D^l\zeta(t)}_{L^2}^2\leq e(t)\left(\norm{D^l\zeta(0)}_{L^2}^2 + \norm{D^{l-2}f(0)}_{L^2}^2 +\int_0^t\norm{D^{l-1}f(\tau)}_{L^2}^2d\tau \right)\; 
	\end{equation}
	 for all $ t\in (0,\infty)$ and every integer $l \in [1,s+1]$, where the case $l=1$ applies if $f(0)\equiv 0$.
	Here $e(t)$ is an exponential function of $t$.
	In particular, if $f\equiv 0$, then there exists some positive constant $\bar{t}_0>0$ and $c_0<1$, such that for $\sum_{l=1}^k \norm{D^l \zeta (0)}_{L^2}\leq c_0 M$, it holds that
	\begin{equation}
	\label{eq:constantM}
   \sum_{l=1}^k	\norm{D^l \zeta(t)}_{L^2} \leq M \quad \text{ for all } t\in (0,\bar{t}_0].
	\end{equation}
\end{theorem}
Connecting to equation \eqref{eq:acc_mcf_u_exApproximate}, where $a_{i,j}(p)=\delta_{ij}- \frac{p_i p_j}{\abs{p}^2+\epsilon^2}$, we have $a(p)=\frac{\epsilon^2}{\abs{p}^2+\epsilon^2}$.
Moreover, it is not hard to verify that all the assumptions on $a_{i,j}$ are fulfilled for this choice.
Also it can  be checked that $a(p)\abs{q}^2\geq \sigma \abs{q}^2$ for some $\sigma>0$ as soon as $\abs{p}$ (or in another words the norm $\norm{u^\epsilon(t)}_{H^k}$) is uniformly bounded from above for $t\in [0,\bar{t}_0]$.
With this preparation, we establish now a local (short time) existence and uniqueness of solutions of \eqref{eq:acc_mcf_u_exApproximate}.
To simplify the presentation, we omit the superscript $\epsilon$ in the following theorem as it is a fixed parameter.

\begin{theorem}
	\label{thm:short_time_approx}
	For every fixed $\epsilon\in(0,1)$, given $u_0\in H^4$ $v_0\in H^3$, a constant $0<c_0<1$, and $\sum_{l=1}^4\norm{D^l u_0}_{L^2}  \leq c_0 M$, there exists $t_0>0$ such that  equation \eqref{eq:acc_mcf_u_exApproximate} admits a unique solution $u\in \bigcap_{s=0}^2 C^s([0,t_0];H^{4-s}) \cap L^{\infty}([0,t_0];H^4) $. Moreover, $\sum_{l=1}^4 \norm{D^l u(t)}_{L^2}\leq M$ holds for all $ t\in [0,t_0]$.
\end{theorem}
\begin{proof}
	The main idea is borrowed from the proof of \cite[Theorem 4]{Mat77}. We first define
	a differential operator of the following form:
	\begin{equation}\label{eq:linearlization}
	L_{w} u(t):= \ddot{u}(t) +\eta \dot{u}(t) - \sum_{i,j}\left( \delta_{ij}  - \frac{w_{y_i}(t) w_{y_j}(t)}{\abs{\nabla w(t) }^2 +\epsilon^2}\right) u_{x_i x_j}(t).
	\end{equation}
	Then , we construct some initial function $u^0(t)\in \bigcap_{s=0}^3 C^s([0,\bar{t}_0];H^{4-s})$ which satisfies $u^0(0)=u_0$, $\dot{u}^0(0)=v_0$, and
	\[D_x^4 u^0(0)=D_x^4 u_0,\quad  \quad  c_0 \norm{u^0(t)}_{H^4} \leq \norm{ u_0}_{H^4}\quad \text{ for all } t\in [0,\bar{t}_0]. \]
	Here $D_x$ denotes the spatial derivative.
	Next, for $m\geq 1$, we consider the following equation recursively:
	\begin{equation}\label{eq:recursive_form}
	\left\{
	\begin{array}{l}
	L_{u^{m-1}} (u^m)=0,  \\
	u^m(0)=u_0, \quad \dot{u}^m(0)=v_0.
	\end{array}\right.
	\end{equation}
	Using Theorem \ref{thm:linear_hyper} with $f\equiv 0$, we find that $u^m \in \bigcap_{s=0}^3 C^s([0,\bar{t}_0];H^{4-s})$ for all $m\geq 0$,  and $\sum_{l=1}^4\norm{D^l u^m(t)}_{L^2}\leq M$ for all $t\in [0,\bar{t}_0]$.
	As $L_{u^{m}} (u^{m+1})=0$, $L_{u^m}(u^{m+1}-u^m)=-L_{u^m}(u^m)$ and
	\[\ddot{u}^m+\eta \dot{u}^m= \sum_{i,j}\left( \delta_{ij}  - \frac{u^{m-1}_{y_i} u^{m-1}_{y_j}}{\abs{\nabla u^{m-1} }^2 +\epsilon^2}\right) u^m_{x_i x_j} \]
	we arrive  at the following equation:
	\[L_{u^m}(u^{m+1}-u^m)= \sum_{i,j}\left( \frac{u^{m-1}_{x_i} u^{m-1}_{x_j}}{\abs{\nabla u^{m-1} }^2 +\epsilon^2} - \frac{u^m_{x_i} u^m_{x_j}}{\abs{\nabla u^m }^2 +\epsilon^2}\right) u^m_{x_ix_j}=:A^m(u^m-u^{m-1}).\]
	Let $f(t):=A^m(u^m(t)-u^{m-1}(t))$.
	Using Sobolev embedding (see, e.g., \cite{Eva10}) it is not hard to check that $f \in C^1( [0,\bar{t}_0];H^1)$, as both 
	\[u^m, \; u^{m-1} \in C( [0,\bar{t}_0];H^4)\bigcap C^1( [0,\bar{t}_0];H^3),  \text{ and } H^4,H^3,H^2 \hookrightarrow  L^\infty(\R^n).\]
	Using the estimate \eqref{eq:hyper_est} from Theorem \ref{thm:linear_hyper} again for the above equation (note $f(0)\equiv 0$),
	we have the following estimate:
	\[\norm{D(u^{m+1}(t)-u^m(t))}^2_{L^2}\leq e(t)\int_0^t \norm{A^m(u^m(\tau)-u^{m-1}(\tau))}^2_{L^2} d\tau.\]
	Note that $u^m(t)\in H^4$ is uniformly bounded for $t\in [0,\bar{t}_0]$. We also have $u^m_{x_i}(t)\in H^3$ and $u^m_{x_ix_j}(t)\in H^2$, and the fact that $H^2,H^3 \hookrightarrow L^\infty(\R^n)$ . Then using the following relation
	\begin{eqnarray*}
		& & A^m(u^m-u^{m-1})= \sum_{i,j}\left( \frac{u^{m-1}_{x_i}  u^{m-1}_{x_j} }{\abs{\nabla u^{m-1}  }^2 +\epsilon^2} - \frac{u^m_{x_i}  u^m_{x_j} }{\abs{\nabla u^m  }^2 +\epsilon^2}\right) u^m_{x_ix_j} \\
		&= & \sum_{i,j} \frac{ u^{m-1}_{x_i}u^{m-1}_{x_j}- u^m_{x_i}u^m_{x_j}  }{\abs{\nabla u^{m-1}}^2 +\epsilon^2} u^m_{x_ix_j}
		+ \sum_{i,j} \frac{ u^m_{x_i}u^m_{x_j} \left(\abs{\nabla u^{m}}^2 -\abs{\nabla u^{m-1}}^2 \right) }{(\abs{\nabla u^m}^2 +\epsilon^2)(\abs{\nabla u^{m-1}}^2 +\epsilon^2)}u^m_{x_ix_j} \\
		&= & \sum_{i,j} \frac{ u^{m-1}_{x_i}(u^{m-1}_{x_j} - u^m_{x_j})}{\abs{\nabla u^{m-1}}^2 +\epsilon^2}u^m_{x_ix_j}  + \frac{(u^{m-1}_{x_i} -u^m_{x_i}) u^m_{x_j}  }{\abs{\nabla u^{m-1}}^2 +\epsilon^2}u^m_{x_ix_j}  \\
		& & + \sum_{i,j} \frac{ u^m_{x_i}u^m_{x_j} \left(\abs{\nabla u^{m}} +\abs{\nabla u^{m-1}} \right)\left(\abs{\nabla u^{m}} -\abs{\nabla u^{m-1}} \right) }{(\abs{\nabla u^m}^2 +\epsilon^2)(\abs{\nabla u^{m-1}}^2 +\epsilon^2)}u^m_{x_ix_j} ,
	\end{eqnarray*}
	and applying the triangle inequality, we conclude that
	\begin{equation}\label{eq:regularity_estimate}
	\norm{A^m(u^m(\tau)-u^{m-1}(\tau))}^2_{L^2} \leq C_{m} \norm{D(u^m(\tau)-u^{m-1}(\tau))}_{L^2} \quad \text{ for } \tau \in [0,\bar{t}_0].
	\end{equation}
	Here $C_{m}$ is a constant depending on semi-norms of $u_m$ and $u_{m-1}$, but not on their difference.
	Since both $\norm{u_{m-1}(t)}_{H^4},\; \norm{u_m(t)}_{H^4}$ are uniformly bounded by $M$ for all $t\in [0,\bar{t}_0]$, there exists a constant $c$ independent of $m$ such that
	\[\norm{D(u^{m+1}(t)-u^m(t))}_{L^2}^2\leq c\int_0^t \norm{D(u^m(\tau) -u^{m-1}(\tau))}^2_{L^2}d \tau \; \text{ for all } m\geq 1.\]
	As $c$ is a fixed constant for all $m>1$, and $\sum_{l=1}^4\norm{D^l u_m(t)}_{L^2}\leq M$ for all $t\in [0,\bar{t}_0]$, there exists a sufficiently small $t_0\in [0,\bar{t}_0]$ such that the right hand side is always strictly smaller than $1$.
	A recursive application of this technique then shows that
	\[\lim_{m\to \infty}\norm{D(u^{m+1}(t)-u^m(t))}_{L^2}^2 \to 0 \; \text{ for all } \; t\in (0,t_0].\]
	Because $u_m\in \bigcap_{s=0}^3C^s([0,t_0];H^{4-s}) \hookrightarrow  L^\infty([0,t_0];H^{1}) $ for all $m$, and the latter is a Banach space, there exists a function $u(t)\in L^{\infty}([0,t_0];H^1)$ such that 
	\[\lim_{m\to\infty} \nabla u_m(t)=\nabla u(t) \text{ for } t\in [0,t_0]\]. 
	That is $u^m\to u$ strongly in $L^\infty([0,t_0];H^1)$.
	When we return to \eqref{eq:linearlization} with $\nabla w(t)=\nabla u(t)$, we see that it actually constructs a convergent sequence $(u_m)_{m\in \N} \in L^\infty([0,t_0];H^4)$ to $u\in L^\infty([0,t_0];H^4)$ which is a solution of the nonlinear equation \eqref{eq:acc_mcf_u_exApproximate}.
	Now we show that $u$ satisfies the stated regularity.
	Use the estimate \eqref{eq:hyper_est} for every $u^m$ of the equation \eqref{eq:recursive_form}, and consider $l=2,3$ there.
	Then we get that the sequence $(\dot{u}_m)_{m\in \N}$ and $(\ddot{u})_{m\in \N}$  are uniformly bounded over $[0,t_0]$ for all $m\geq 0$, respectively, and in particular they are equicontinuous due to the fact that their derivatives $(\ddot{u}_m)_{m\in \N}\in  C([0,\bar{t}_0];H^{2})$ and $(\dddot{u})_{m\in \N} \in C([0,\bar{t}_0];H^{1})$ are  uniformly bounded, again derived from the estimate \eqref{eq:hyper_est}. 
	This allows us to apply the Arzel\'a-Ascoli theorem to show that there are subsequences of $(u_m)_{m\in \N}$, still denoted by $(u_m)_{m\in \N}$, so that $\dot{u}_m \to \dot{u}$ in $C^1([0,t_0];H^3)$, and $\ddot{u}_m \to \ddot{u}$ in $C^0([0,t_0];H^2)$,  respectively. 
	Using similar argument, we can conclude that the spatial derives $D_x^2 u_m \to D_x^2 u$ in $C^0([0,t_0];H^2)$ as well, since $(D_x^2 u_m)_m$ is uniformly bounded in $L^\infty([0,t_0];H^2)$ and they are equicontinuous as $(D^3 u_m)_m$ is uniformly bounded using \eqref{eq:hyper_est}.
	This yields the uniform convergence $u_m\to u$ so that $u\in \bigcap_{s=0}^2 C^s([0,t_0];H^{4-s})$ is a strong solution to \eqref{eq:acc_mcf_u_exApproximate}.
	
	The proof for uniqueness is rather similar.
	If there exists another solution $\bar{u}$, then let $u^{m-1}=\bar{u}$, $u^{m}=u$ in \eqref{eq:recursive_form}. Using the estimate in \eqref{eq:regularity_estimate} we find then \[\norm{D(u(t)-\bar{u}(t))}_{L^2}=0.\]
	Taking into account the initial and boundary conditions, we conclude that \[\norm{u(t)-\bar{u}(t)}_{L^2}=0 \text{ for }  \; t\in [0,t_0].\]
	Therefore the solution is unique over $[0,t_0]$.
\end{proof}
We point out that the $H^4$ regularity of the initial value $u_0$ seems necessary for using our current strategy of  proof. Particularly, this regularity is required in order to have the estimate \eqref{eq:regularity_estimate}.

\begin{remark}\label{rem:global_solution}
	Based on the short time solution, we now comment on how to achieve  a global solution of \eqref{eq:acc_mcf_u_exApproximate} for arbitrary $T\in (0,\infty)$ of the  domain $ (0,T]$ by assuming a sufficiently regular initial value.
	The idea is to make sure that for arbitrary $T>0$, one has the estimate
	\begin{equation}\label{eq:hyper_properties}
	\sum_{l=1}^4\norm{D^l u(t)}_{L^2} \leq c_0 M \quad  \text{ for all } \quad t\in [0,T], 
	\end{equation}
	which is an assumption on the the initial data in Theorem \ref{thm:short_time_approx}. 
	If this is fulfilled, we see that $u(t_0)$ satisfies the requirement for the initial data of Theorem \ref{thm:short_time_approx}.
	Let $u(t_0)$ again be the initial data. Then one can derive the solution for the time domain $[t_0,2t_0]$ using the same technique as in Theorem \ref{thm:short_time_approx} and \eqref{eq:hyper_properties}.
	The procedure can be repeated for the whole time domain $[(n-1)t_0,nt_0]$ for every $n\in \N$ and $n\leq T/t_0$.
	This idea has been realized in \cite{Mat77} where more general equations have been considered.
	It is proven in \cite[Lemma 6]{Mat77} that for sufficiently regular initial data,  that is $\sum_{l=1}^4 \norm{D^l u_0(0)}_{L^2} \leq \varepsilon$ for some $\varepsilon>0$ depending on $M$ in \eqref{eq:constantM} sufficiently small,  if  \eqref{eq:acc_mcf_u_exApproximate} has the solution $u\in C([0,T],H^4)$ for arbitrary $T>0$, and if the estimate $\sum_{l=1}^4\norm{D^l u(t)}_{L^2} \leq M $ holds true over the whole temporal domain $[0,T]$, then \eqref{eq:hyper_properties} holds true.
	Note that $\varepsilon$ does depend on $M$, but not on $T$.
\end{remark}

We mention that for problems with higher spatial dimension, i.e. $x\in \R^d$ ($d\geq 3$) there are energy decay rates of general quasilinear strictly hyperbolic equations available in \cite{CheMil11}.

In order to study the solution of the original equation \eqref{eq:acc_mcf_u_ex}, there are certain restrictions using the current framework. First we can not pass $\epsilon \to 0$ as there is no uniform estimate on the approximating solutions.
Second, when $\epsilon=0$, it is a degenerated hyperbolic equation in the entire domain $\R^n\times (0,\infty)$.
Using the current procedure needs some energy estimates for the corresponding degenerated linear hyperbolic PDEs.
In the parabolic problem or first-order hyperbolic problem case, viscosity solutions \cite{CheGigGot91,EvaSpr91} has been introduced to overcome the difficulty. 
However, for degenerated second-order hyperbolic PDEs, because of the lack of maximum principle, such an application seems out of reach.
The literature appears very sparse on such and related issues up to the best of our knowledge. 
There is some work in this direction for degenerate linear hyperbolic equations, such as, e.g., \cite{Dan88,Asc06}, but definitely further efforts, or maybe even completely new concepts are required in order to successfully solve the nonlinear problem.
We notice that, without introducing the regularization parameter $\epsilon$, a local existence of solutions seems can be provable using the techniques from \cite{Joh76,Kla80}. However, it is still not clear how to proceed with the global existence and uniqueness.
We leave this for future work.

\section{Applications and numerical results}
\label{sec:algorithm}
In this section, we consider numerical aspects of the two proposed damped second-order nonlinear flows and their applications.
In particular, we focus on applications in image denoising and  correcting displacement errors in image data that motivate this study. The purpose of this section is to illustrate the behavior of the solutions of two novel hyperbolic PDEs, and show the applicability of the new PDEs to imaging problems. Nevertheless, we first provide a common algorithmic framework based on discretizations of the two PDEs. The convergence analysis of the algorithm is deferred to the Appendix. Note that, here we mostly aim at understanding the behavior and new aspects of the second-order PDEs, rather than developing new numerical algorithmic schemes. Therefore, we use simply the explicit Euler scheme and a finite difference method. 
One may find more sophisticated approaches for the discretizations of PDEs with respect to the nonsmoothness; see for instance \cite{ChaPoc10,HinRauHah14,BarDieNoc18}.

\subsection{An algorithm accommodating both types of PDEs}
Considering the evolutionary PDEs as regularization methods, the stopping time is important as it plays the role of the regularization parameter.
In principle, the stopping criterion for image problems typically depends on the noise level and the initial data, just as in standard regularization theory \cite{Engl:1996}, and the regularization parameters are chosen according to the magnitude of noise.
We provide here an automatic stopping rule based on thresholds on the high frequency in Fourier space.
We denote now $\mathbf{u}\in \mathbb{R}^{M}\times \mathbb{R}^{N}$ a discrete image matrix.
In order to set up our stopping criteria, we  adopt a frequency domain threshold method based on the fact that noise is usually represented by high frequencies in the frequency domain \cite[Chapter 4]{Gonzalez2002}.
An associated high frequency energy is defined by
\begin{equation*}
E_{N_0} (\mathbf{u}) = \sum_{(i,j)\in N_0} \left| \mathcal{F}(\mathbf{u}) (i,j) \right|,
\end{equation*}
where $ \mathcal{F}(\mathbf{u})$ denotes a 2D discrete Fourier transform of an image $\mathbf{u}$, and $N_0$ presents a selected set which contains high frequency indices. For instance, if one uses the Matlab function $\operatorname{fft2}$ as the discrete Fourier transform operator $\mathcal{F}$, then the high frequency coefficients will be the central part of the Fourier domain, and we take $N_0:=[\lfloor \rho M \rfloor,M-\lfloor \rho M \rfloor ] \times [\lfloor \rho N \rfloor,N-\lfloor \rho N \rfloor ]$, where $\rho \in (0,0.5)$ and $\lfloor \cdot \rfloor $ denotes the floor function. We define a function called Relative Denoising Efficiency (RDE)  as follows:
\begin{equation*}
\text{RDE}(k)=E_{N_0} (\mathbf{u}^k) )/  \max_{i,j} \left| \mathcal{F}(\mathbf{u}^k) (i,j) \right|.
\end{equation*}
Then, the value of  RDE  at every iteration can be used in a stopping criterion. Based on the above preparation, we propose the following algorithm. 
Note here, we use the operator $\mathbf{F}^{k} \mathbf{u}^{k}$ to uniformly represent the nonlinear parts resulted from both TVFs and MCFs, namely $\text{div}\left(\frac{\nabla  \mathbf{u}^{k}}{\abs{ \nabla  \mathbf{u}^{k}}}\right)$
and $\abs{ \nabla  \mathbf{u}^{k}}\text{div}\left(\frac{\nabla  \mathbf{u}^{k}}{\abs{ \nabla  \mathbf{u}^{k}}}\right)$. In the Appendix, we give a theoretical criteria on adaptive time-step size in \eqref{eq:parametersDtSteady}, which for simplicity of our image examples, can also be chosen as small constants.The estimate there gives a theoretical bound on the size of $\Delta t_k$ for the convergence analysis.
\renewcommand{\thealgorithm}{\arabic{algorithm}}
\setcounter{algorithm}{0}
\begin{algorithm}
	\begin{itemize}
		\item[] Input: Image data $\mathbf{u}^\delta_0$. Parameters $\eta>0$ and $\epsilon>0$. Tolerance $tol>0$.
		\item[] Initialization: $ \mathbf{u}_0\gets u_0$, $ \mathbf{v}_{0} \gets v_0$, $\Delta t_0>0 $, $\text{RDE}(0) \gets E_{N_0}(\f u_0)$, $k \gets 0$.
		\item[] While: $\text{RDE}(k)> tol $
		\begin{enumerate}
			\item[i.] $\mathbf{v}^{k+1} \gets (1-\eta \Delta t_k) \mathbf{v}^{k} + \Delta t_k  \mathbf{F}^{k} \mathbf{u}^{k}$;
			\item[ii.] $ \mathbf{u}^{k+1} \gets \mathbf{u}^{k} + \Delta t_k  \mathbf{v}^{k+1}$;
			\item[iii.] Update $\Delta t_k$ according to \eqref{eq:parametersDtSteady} ;
			\item[iv.] $k \gets k+1$;
			\item[v.] $\text{RDE}(k) \gets E_{N_0} (\mathbf{u}^k) ) /  \max_{i,j} \left| \mathcal{F}(\mathbf{u}^k) (i,j) \right|$.
		\end{enumerate}
		\item[] Output: A corrected image $\hat{\mathbf{u}} \gets \mathbf{u}^{k}$.
	\end{itemize}	
\caption{Explicit Euler scheme for discretizing damped second-order dynamics}
	\label{alg:symplectic}
\end{algorithm}
\begin{remark}
In practice, uniform time steps can be used and Step $iii$ in Algorithm \ref{alg:symplectic} can be ignored. However the criteria in \eqref{eq:parametersDtSteady} (see Appendix) of chosen step size is a kind of guideline for stable convergence of the algorithm.
\end{remark}
If we look deeper into the iterations of the algorithm, then we find
\[ \mathbf{u}^{k+1} =  \mathbf{u}^{k} + \Delta t_k \left( (1-\eta \Delta t_k) \mathbf{v}^{k} + \Delta t_k  \mathbf{F}^{k} \mathbf{u}^{k}\right).\]
As $ \mathbf{v}^{k}= \frac{\mathbf{u}^{k} - \mathbf{u}^{k-1} }{ \Delta t_{k-1}}$, this turns out to be
\[\mathbf{u}^{k+1} = \mathbf{u}^{k} + \frac{\Delta t_k}{\Delta t_{k-1}}  (1-\eta \Delta t_k)( \mathbf{u}^{k} - \mathbf{u}^{k-1})+ (\Delta t_k)^2  \mathbf{F}^{k}  \mathbf{u}^{k}.\]
Note that if we use a uniform time step, that is $\Delta t_k \equiv \Delta t_0$ , and choose $\eta=\frac{1}{\Delta t_0}$, then the algorithm is equivalent to the steepest descent method with step size $(\Delta t_0)^2$ whenever $\mathbf{F}^{k} \mathbf{u}^{k}$ can be interpreted as the negative gradient direction for an associated energy.
On the other hand, we can see that the term $\frac{\Delta t_k}{\Delta t_{k-1}}  (1-\eta \Delta t_k)(\mathbf{u}^{k} - \mathbf{u}^{k-1})$ plays a similar role as the correction step in Nesterov's scheme \cite{Nes04}, by which it is supposed to accelerate the steepest descent method.
In the following examples, in order to draw comparisons, we will always implement the first-order method using Algorithm \ref{alg:symplectic} by setting $\eta\equiv \frac{1}{\Delta t}$.
In this sense, we will find that Algorithm \ref{alg:symplectic} for  second-order flows includes acceleration and can be favoured over the first-order flows in applications.
One may note that the time step of the first-order method and the second-order method are not equal, as the former is $(\Delta t)^2$, but the latter is $\Delta t$.
However, we should be aware that for the discretization of evolutionary PDEs, the Courant–Friedrichs–Levy (CFL) condition needs to be taken into account for explicit time discretizations.
This puts restrictions on the length of the discrete time steps of the numerical implementations, e.g. for $2$D linear equations, second-order flows can have $\Delta t\sim h_x$ while first-order flows usually have $\Delta t\sim h_x^2$.
Here $h_x$ is the discretization mesh size of spatial variables.
In this sense, we argue that discrete time steps of order $(\Delta t)^2$ for first-order flows and $\Delta t$ for second-order ones are justified.
We shall notice that the costs of iterations for first-order flows and second-order flows are rather identical as Step (ii) in Algorithm \ref{alg:symplectic} consists very cheap operations. 

\subsection{Simulation on solutions and numerical results}
\label{subsec:numerical}
\subsubsection*{Evolution of characteristic functions}
In order to study some fundamental effects of our PDEs on the contour and scale of images, in the first example we test with an image resulting from a characteristic function.
We start with an image that is a scaled indicator function of a square $Q\subset \Omega$ (left image in Figure \ref{fig:square})
\[u_0=255\cdot I_{Q}=
\left\{\begin{array}{ll}
255, & x\in Q; \\
0, & x \in \Omega/Q.
\end{array}\right.\]
It is well-known that the first-order TVF will decrease the intensity value of the region $Q$, tends to preserve its shape, while the first-order level-set MCF will slowly shrink the square $Q$ to a circular shape, thus reducing its perimeter, but it will preserve the intensity value.
Figure \ref{fig:square} shows  numerical results on the evolution of the second-order flows \eqref{eq:acc_tvf} and \eqref{eq:acc_mcf_u}.
We use a $205\times 205$ square grid, and fix the domain $\Omega=(0,1)\times (0,1)$; therefore the spatial step size is $\Delta x = \Delta y= 1/204$. We use uniform step size for the time discretization, and choose  $\Delta t_1=0.001$, and $\eta=\frac1{50 \Delta t_1}$ for the TVF methods, and choose $\Delta t_2=0.0001$, and $\eta=\frac1{20 \Delta t_2}$ for the level-set MCF methods. For both we take $v_0=0$, and run $50000$ iterations.
From Figure \ref{fig:square}, we find that the resulting algorithm from \eqref{eq:acc_tvf} and \eqref{eq:acc_mcf_u} present exactly the same behavior as their first-order counterparts. In the three images in Figure \ref{fig:square}, we take the same pixel at the position $(106,100)$. The intensity of the original square is $255$, and it is the initial value, which decreases to $242.2$ in the image evolved with respect to the second-order TVF \eqref{eq:acc_tvf}, but remains the same in the image evolved according to the  second-order level-set MCF \eqref{eq:acc_mcf_u}. On the other hand, the shape of the square is almost not changed by the second-order TVF except the sharp corners, while it has been shrunk to a circle by the second-order level-set MCF.
\begin{figure}[h!]
	\centering
	\includegraphics[width=0.3\textwidth]{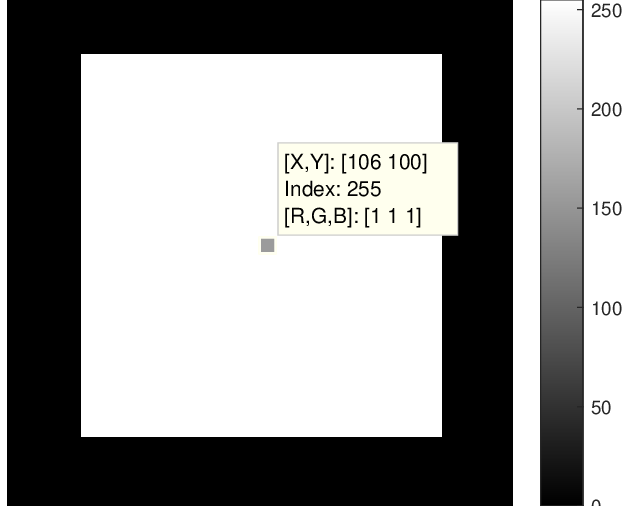} \;
	\includegraphics[width=0.3\textwidth]{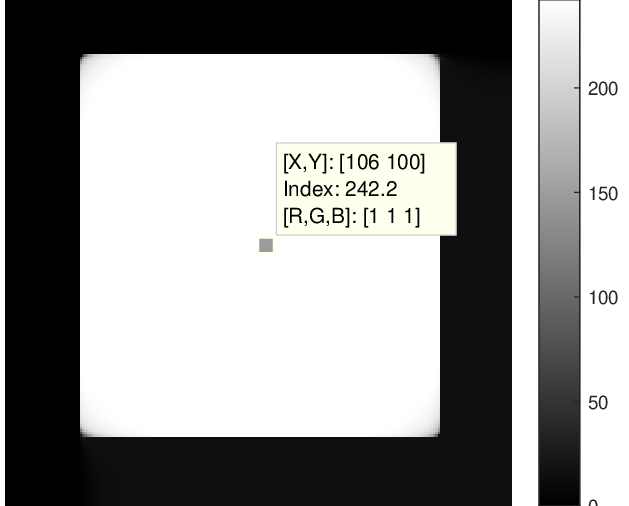} \;
	\includegraphics[width=0.3\textwidth]{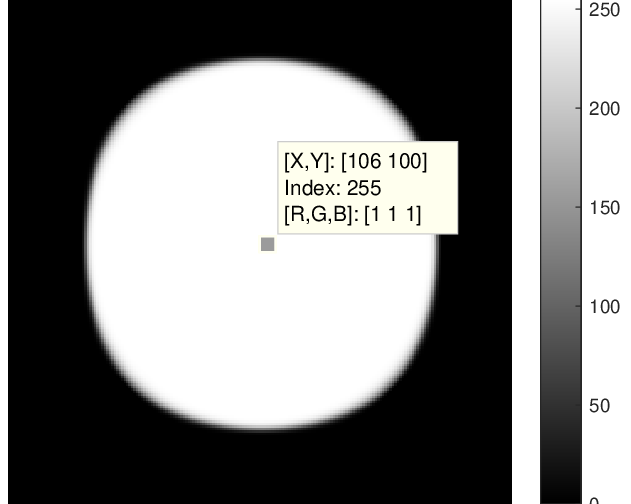}
	\caption{An example that distinguishes the damped second-order MCF and the damped second-order TVF. From left to right: the square image; the result of damped second-order TVF (middle) and  the result of damped second-order MCF (right). }
	\label{fig:square}
\end{figure}
Note here and also in the other examples that we take $\epsilon=10^{-16}$ which is already much smaller then the temporal and spatial mesh sizes, respectively. However, it seems sufficient for the numerical examples we considered.
Numerical diffusion is observed in the second and third images in Figure \ref{fig:square}, and the effect grows as time step and iteration numbers get larger. However, in our following image applications, only a small number of iterations is needed, and thus we will not investigate this issue here further since it is out of the scope of this paper. 

We have initialized the discussion in Section \ref{subsec:Ana_solution_stvf} that the initial velocity can influence the evolution of the trajectory. 
Here we show some numerical observation when $u_0$ and $v_0$ connect to different families of eigenfunctions of $\partial \text{TV}$. In Figure \ref{fig:diff_eigenfunctions}, the left and middle images at the above row are corresponding to $u_0$ and $u_1$ functions, respectively. Both $u_0$ and $u_1$ are taken to be $256\times 256$ images with intensity value $255$ for the white and $0$ for the black. We take $v_0=-\eta u_1$, use the same discretization as in the previous example, and choose the parameters $\Delta t=0.001$ and $\eta=10$. The second-order TV flow under such initial conditions shows interesting behaviors.  Note that the second-order MC flow has a different perspective and is not tested here. 
 \begin{figure}[h!]
	\includegraphics[width=0.32\textwidth]{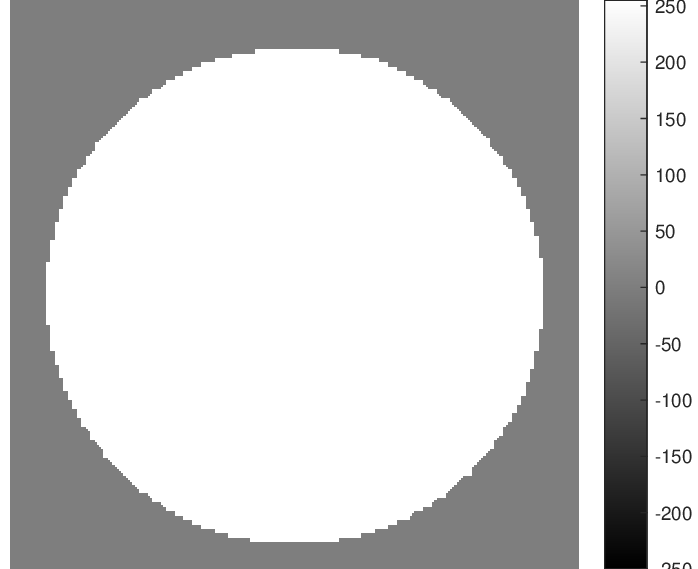} 
	\includegraphics[width=0.32\textwidth]{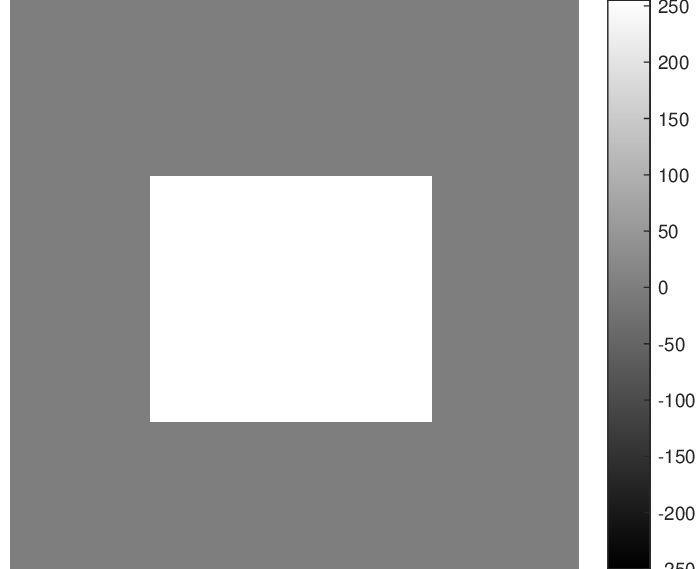}
	\includegraphics[width=0.32\textwidth]{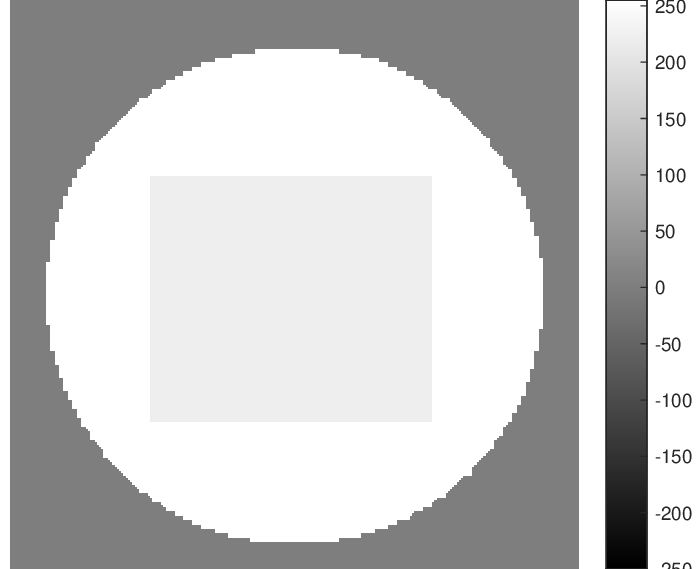}\\ 	\vskip 0.1pt
	\includegraphics[width=0.32\textwidth]{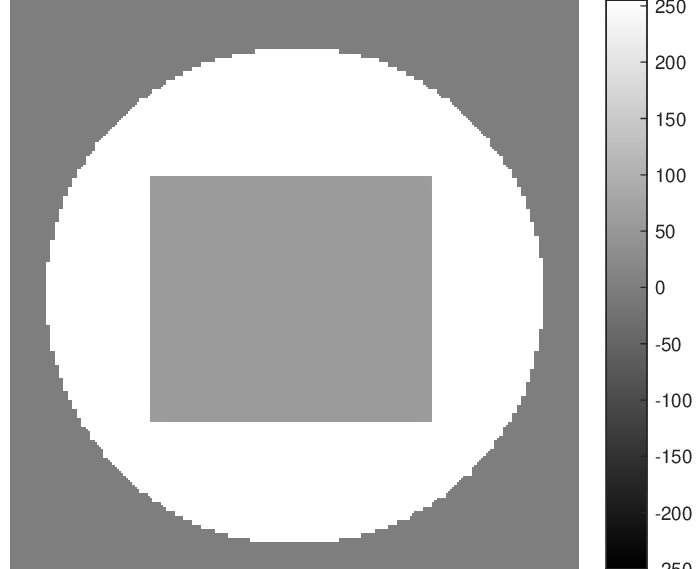}
	\includegraphics[width=0.32\textwidth]{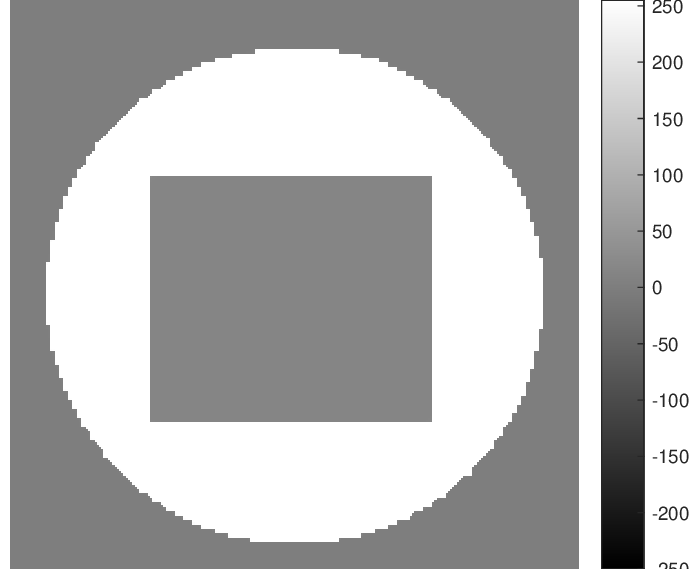}
	\includegraphics[width=0.32\textwidth]{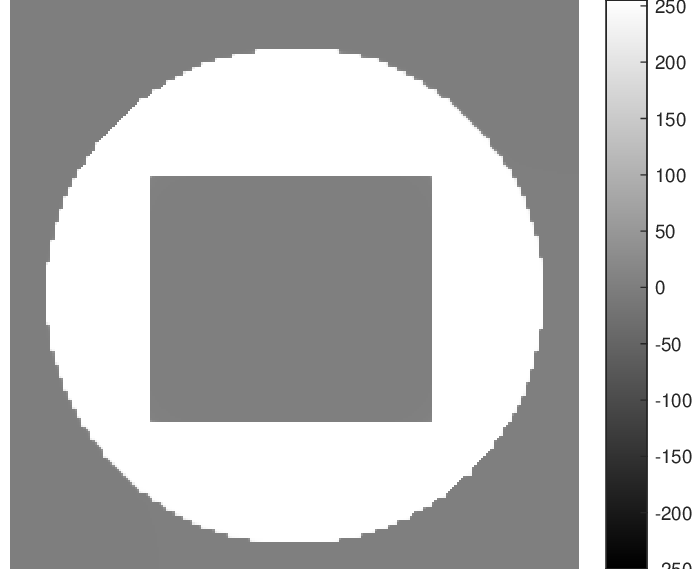}
	\caption{Initial value $u_0$ and initial velocity $v_0$ with respect to different eigenfunctions. From left to right, and from above to bottom: $u_0$, $u_1$ and the result from the second-order TV flow after 50, 500, 1000, 5000 iterations, respectively. Note that we choose the initial velocity $v_0=-\eta u_1$ ($\eta=10$).}
	\label{fig:diff_eigenfunctions}
\end{figure}	
 \begin{figure}[h!]
	\centering
	\includegraphics[width=0.32\textwidth]{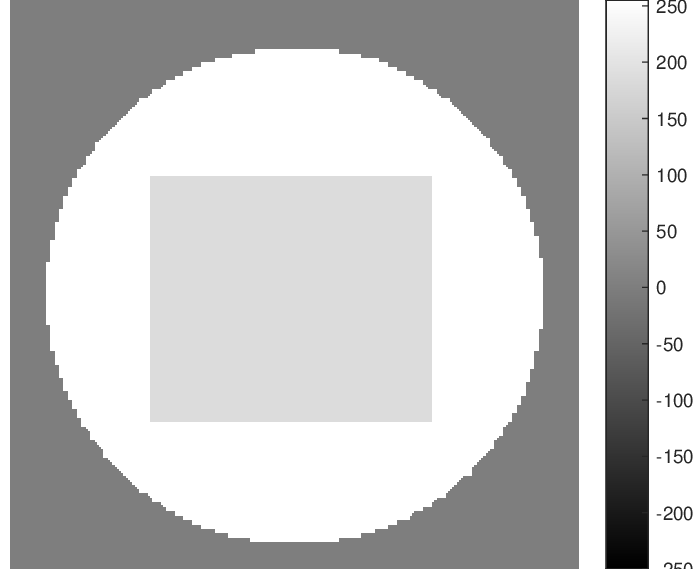}
	\includegraphics[width=0.32\textwidth]{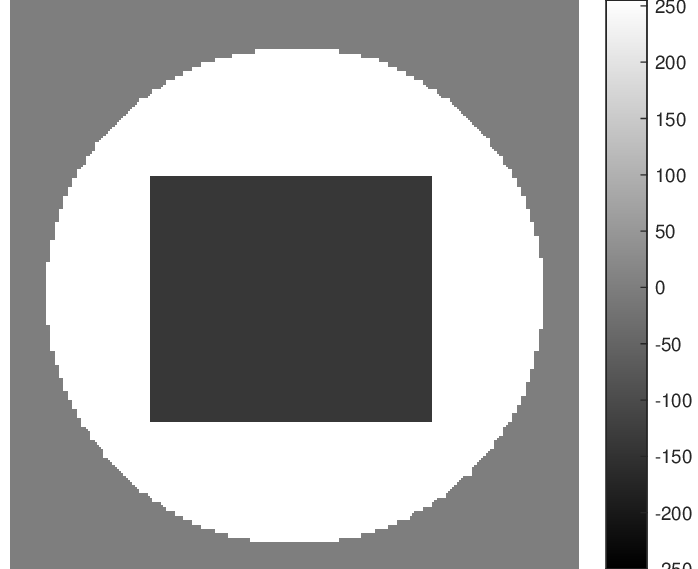}
	\includegraphics[width=0.32\textwidth]{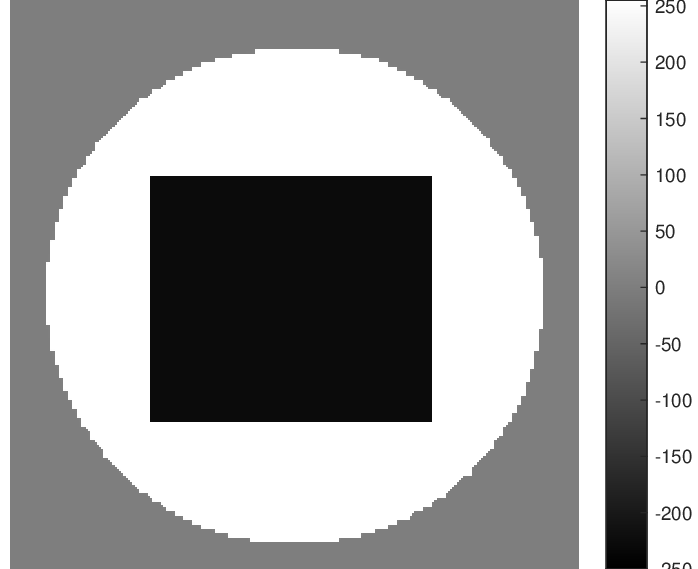}
	\caption{Initial value $u_0$ and initial velocity $v_0$ with respect to different eigenfunctions. From left to right: Result from the second-order TV flow after 50, 500, 1000 iterations, respectively. Initial velocity $v_0=-2\eta u_1$ where $u_1$ is the same as Figure \ref{fig:diff_eigenfunctions1} for $\eta=10$.}
	\label{fig:diff_eigenfunctions1}
\end{figure}	
In Figure \ref{fig:diff_eigenfunctions}, the square intensity gradually decreases, and then it stays around $0$, while the intensity value of the circle with $0$ initial velocity region is only very mildly decreased. 
In Figure \ref{fig:diff_eigenfunctions1}, we re-scale the initial velocity as $v_0=-2\eta u_1$, but use the same $u_0$ as the one in Figure \ref{fig:diff_eigenfunctions}.  Now we see a different phenomenon in the evolution of the solutions when compared to the one in Figure \ref{fig:diff_eigenfunctions}. Here, the intensity of the square decreases continuously  after it reaches $0$ and eventually become negative.

\subsubsection*{Image de-noising}
We notice that the translation of Nestorov's acceleration algorithm to the continuous second-order dynamics gives naturally rise to a homogeneous initial velocity. 
In this section, we first depict the examples under such an initial condition.
However, inhomogeneous initial velocity, as it is shown in the last section, can also be used and it provides some new aspects in imaging. We demonstrate its application potential by a simple toy case in the second part of this section. 
\paragraph{Using homogeneous initial velocity}
Our second set of examples concern image de-noising using both the TVFs and the level-set MCFs.
For the reason of comparing the evolutionary effect, we will always consider uniform time steps $\Delta t_k=\Delta t$ for each four of the discretized flows.
The image we tested here are of pixel size $400\times 400$, and we fix it to be in the domain $\Omega=(0,1)\times (0,1)$; therefore the spatial step size is $\Delta x = \Delta y= 1/399$. We choose  $\Delta t_1=0.003$, and $\eta=\frac1{50 \Delta t_1}$ for the TVF methods, and choose $\Delta t_2=0.0001$, and $\eta=\frac1{10 \Delta t_2}$ for the level-set MCF methods.
Using $\eta\equiv \frac{1}{\Delta t}$ in Algorithm \ref{alg:symplectic} for the first-order flows, the first-order TVF and the first-order MCF have step sizes $(\Delta t_i)^2$ for $i=1,2$, respectively. For setting the final iteration of the algorithm we choose $\rho=0.2$, and $tol=1.0$ in this example.
The noise is generated using a random variable $\delta$ follow Gaussian distribution of mean $0$ and standard deviation $20$. Adding it to the pepper image results in an image with approximately $15$ percentage of noise.
\begin{figure}[h!]
	\begin{center}
		\includegraphics[width=0.3\textwidth]{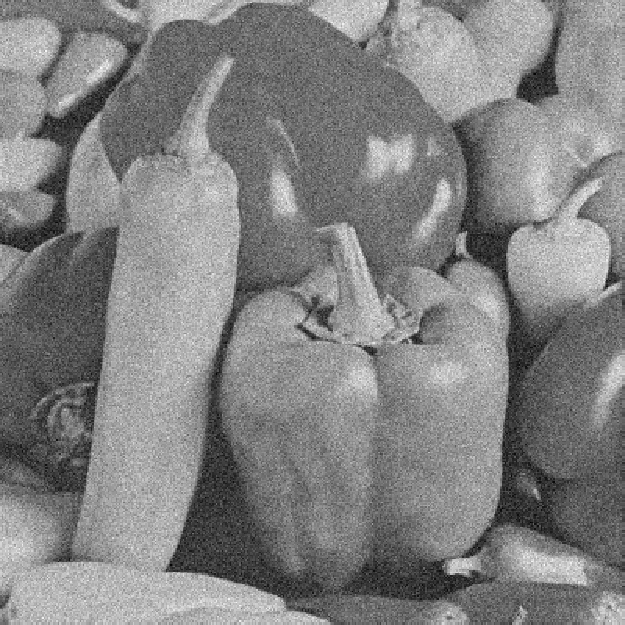}\quad
		\includegraphics[width=0.3\textwidth]{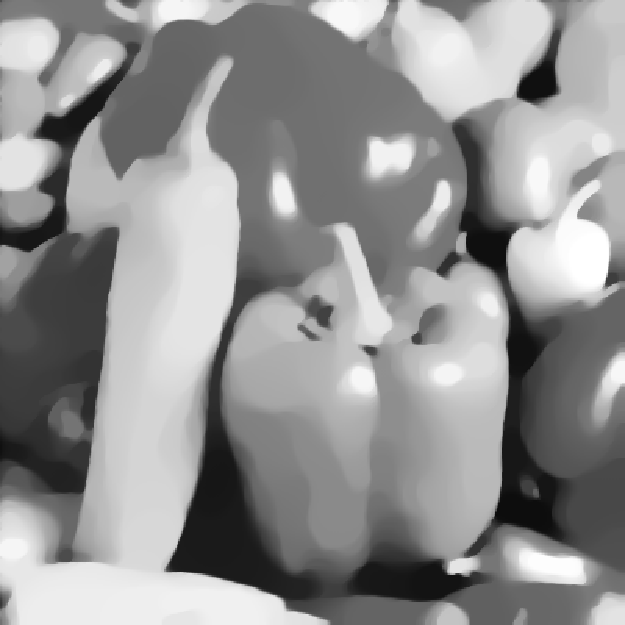}\quad
		\includegraphics[width=0.3\textwidth]{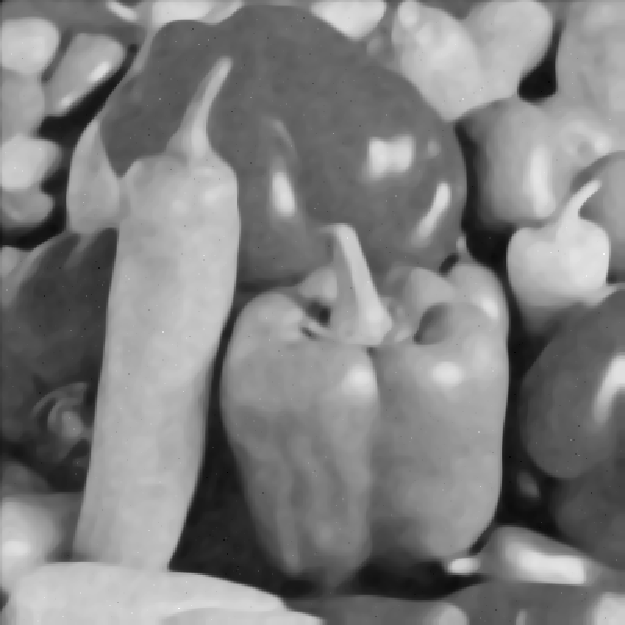} \\ \vskip 0.1pt
		\includegraphics[width=0.3\textwidth]{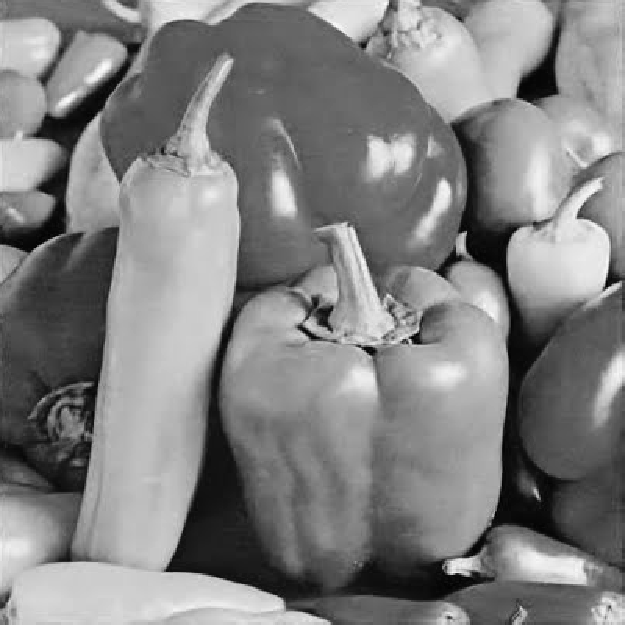}\quad
		\includegraphics[width=0.3\textwidth]{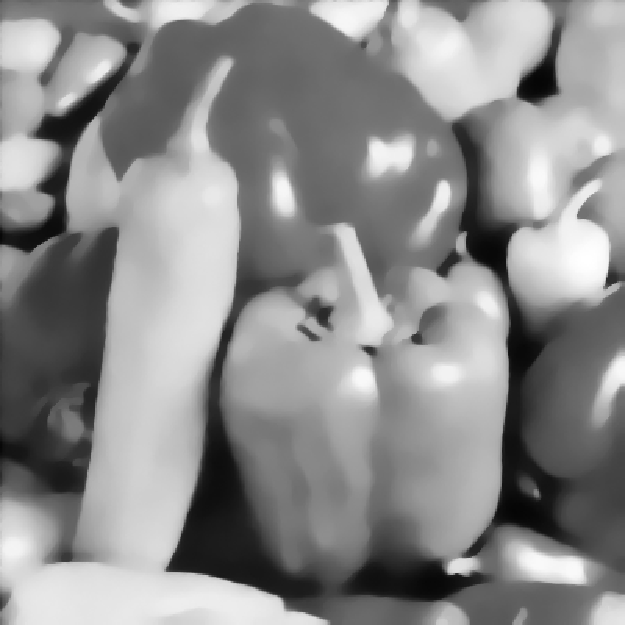}\quad
		\includegraphics[width=0.3\textwidth]{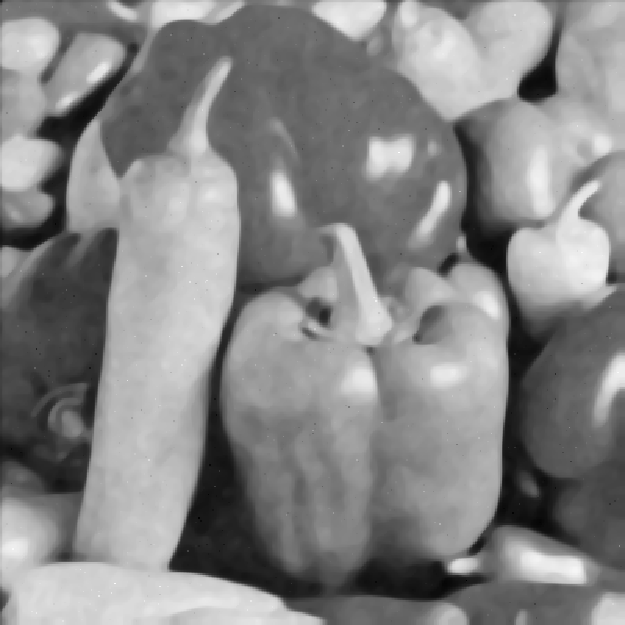}
	\end{center}
	\caption{Comparison of denoising for pepper image using the algorithms by TVFs and also level-set MCFs. From left to right: The first column: the noisy image; the noise-free image; the second column: the result of first-order TVF, the result of  second-order TVF; the third column: the result of first-order level-set MCF, the result of  second-order level-set MCF. }
	\label{fig:denoising}
\end{figure}

The results in Figure \ref{fig:denoising}  indicate that all four methods yield competitive results for denoising, while there is a big difference between the first-order flows and the damped second-order flows in terms of CPU time.
Note that this has been reflected in the different iteration times to reach the stopping rule. For the first-order TVF and the second-order one, it is $19189  \text{ vs. }  338$ iterations, while for the first-order level-set MCF and the second-order one, it is $4299  \text{ vs. }  331$ iterations.
Each scheme is run on a computer with Intel 3687U CPU, 2.10GHz$\times$4,  15.5GB RAM, and using Matlab 2017b.
One may notice that the algorithm by level-set MCFs is capable of denoising almost as good as the TVFs.
\paragraph{Using inhomogeneous initial velocities }
Our numerical example here shows that if we choose the initial velocity properly, the second-order flow can behave significantly different to the one with homogeneous initial velocity. This property distinguishes second-order dynamics from their first-order counterparts.
Based on the observation from Section \ref{subsec:Ana_solution_stvf}, when the initial value $u_0$ is considered positive with respect to some eigenfunctions, and if we properly scale the initial velocity, then the decay of the total variation can be accelerated in comparison to the homogeneous initial velocity case. This inspires the next experiments.
Let $g$ be the noise in the initial data $u_0$, i.e., $u_0=u+g$.
The idea is then if we can have the initial velocity to be negatively proportional to $g$, then the algorithm of second-order flow has the potential to eliminate the noise more efficiently than the flow with homogeneous initial velocity. 
To verify this, we design the following model example.
Our tests are again based on the pepper image with additive noise. 
We take a rather strongly degraded image $u_0=u+g$ with $g=5\delta$. Here $\delta$ is an instance of a random variable following a Gaussian distribution with mean $0$ and standard deviation $20$. The discretization settings are exactly the same for the homogeneous and inhomogeneous initial conditions, where $\Delta t=0.003$, and the spatial discretization is the same as in the previous example.
The inhomogeneous velocity is obtained by applying a simple highpass filter. More specifically, we apply a fast Fourier transform (fft2 in MATLAB) to $u_0$, and only keep the $19\%$ of  the Fourier coefficients from the high frequency region in Fourier space. Then applying the inverse Fourier transform (ifft2) we get the reconstructed image from these Fourier coefficients and denote it by $h_g$ which approximates $g$. We choose $v_0=-\eta h_g$, and parameter values $\eta=10$, $\rho=0.125$, $tol=1$. The results for homogeneous and inhomogeneous initial velocities are presented in Figure \ref{fig:denoising_nonhomo1}. The former consumes $1936$ iterations to stop while the latter terminates already after $922$ iterations. We notice that the right image also has slightly better quality in terms of the standard image measures. More precisely, we use the following two indexes: the mean square error (MSE), and the structural similarity (SSIM).
MSE measures the absolute differences of pixels between the corrected image and the ideal image, for which a smaller value indicates a better match; SSIM is more focus on the structure differences, and a larger value refers to a better structure preservation. The right image in Figure \ref{fig:denoising_nonhomo1} has MSE=246.91, and SSIM=0.35, while  it is: MSE=266.13, and SSIM=0.31 for the middle image.
	\begin{figure}[h!]
	\includegraphics[width=0.3\textwidth]{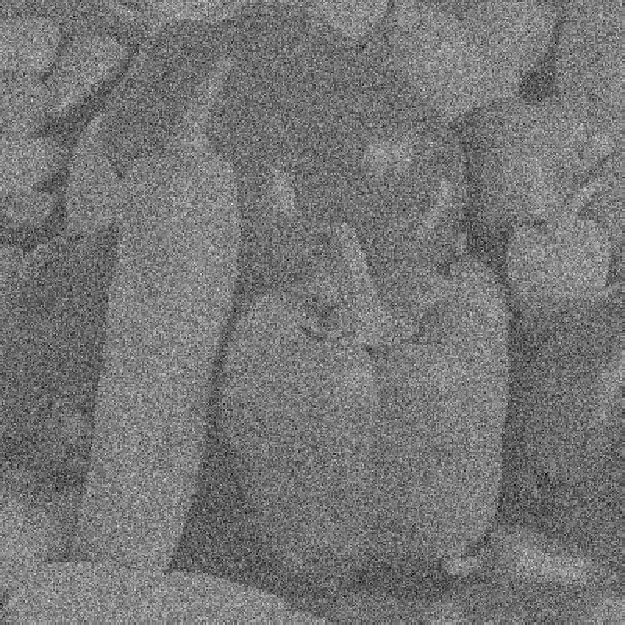}\quad
	\includegraphics[width=0.3\textwidth]{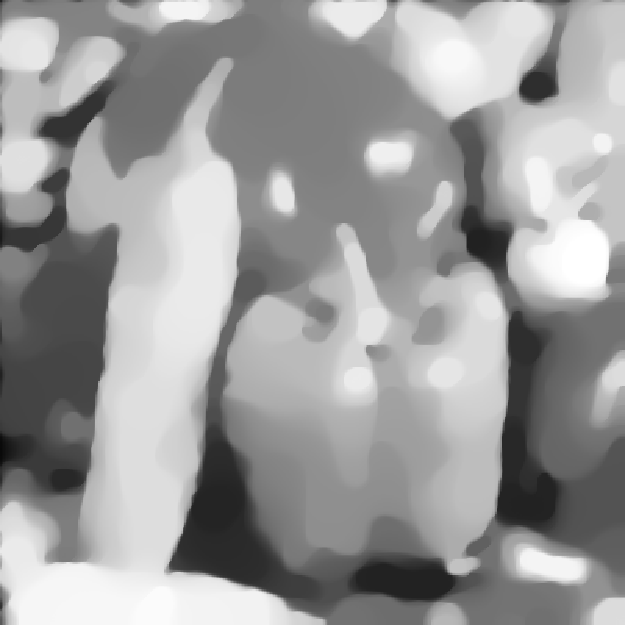}\quad
	\includegraphics[width=0.3\textwidth]{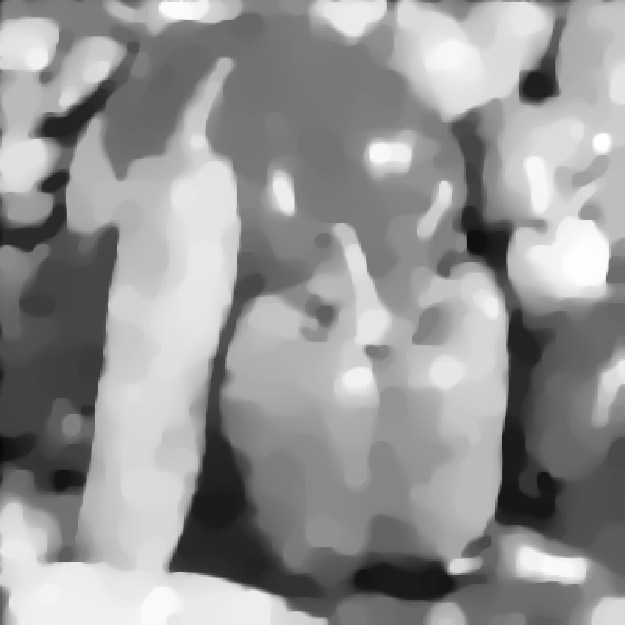}
	\caption{From left to right: Noisy pepper $u_0=u+5\delta$; The result using homogeneous initial velocity; The results using inhomogeneous initial velocity $v_0=-\eta h_g$ for $\eta=10$.}
	\label{fig:denoising_nonhomo1}
\end{figure}
 In order to highlight the effect of the choice of $\psi_0=-\eta g$, we show an ideal example by amplifying the noise $100$ times (i.e., $g=100\delta$)  with the same $\delta$ as before. We choose parameter values $\Delta t=0.001$, $\eta=10$, $\rho=0.25$, $tol=1$, and compare the outcome from Algorithm \ref{alg:symplectic} by the first-order TV flow, the second-order TV flow with homogeneous initial velocity $\dot{u}(0)=0$, and inhomogeneous initial velocities $\dot{u}(0)=-0.5\eta g$, $\dot{u}(0)=-2\eta g$ and $\dot{u}(0)=-\eta g$, respectively. These results are presented in Figure \ref{fig:denoising_nonhomo2}. 
	\begin{figure}[h!]
		\includegraphics[width=0.3\textwidth]{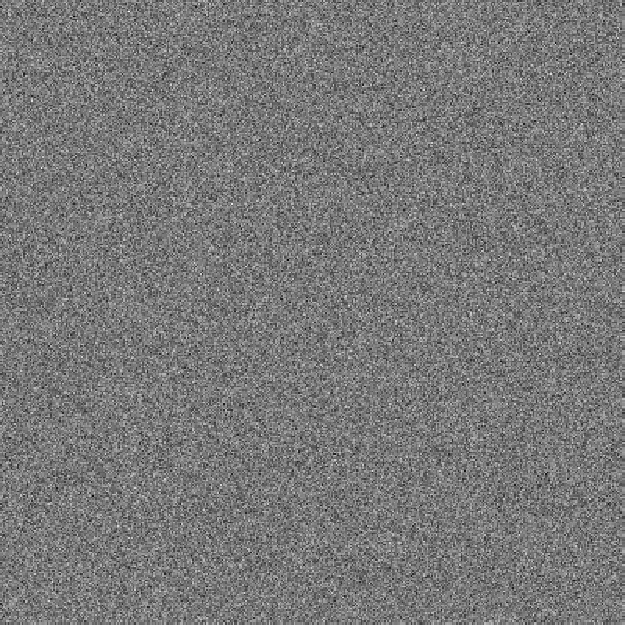}\quad
		\includegraphics[width=0.3\textwidth]{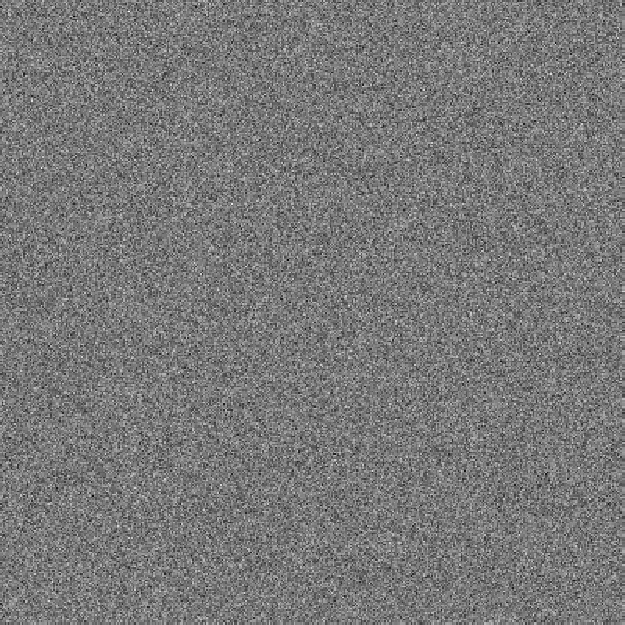}\quad
		\includegraphics[width=0.3\textwidth]{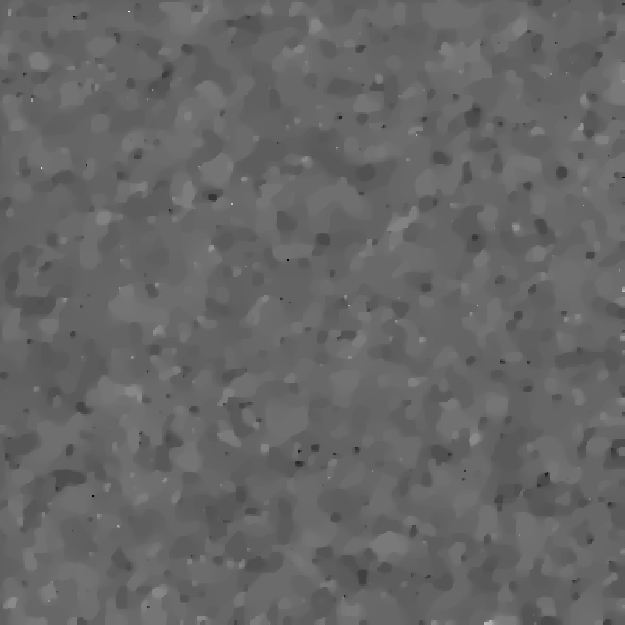}\\	\vskip 0.1pt
		\includegraphics[width=0.3\textwidth]{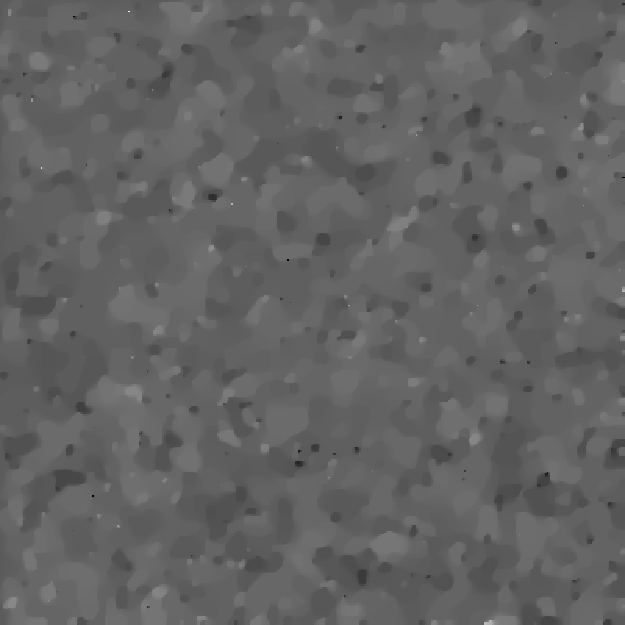}\quad	\includegraphics[width=0.3\textwidth]{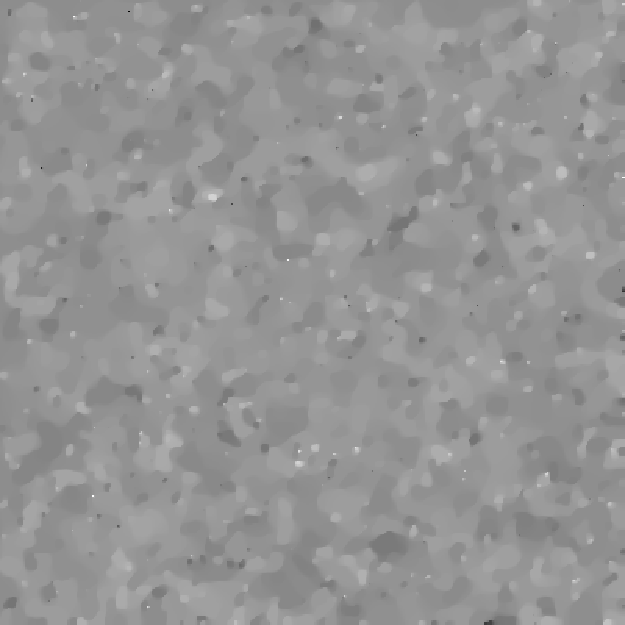}\quad
		\includegraphics[width=0.3\textwidth]{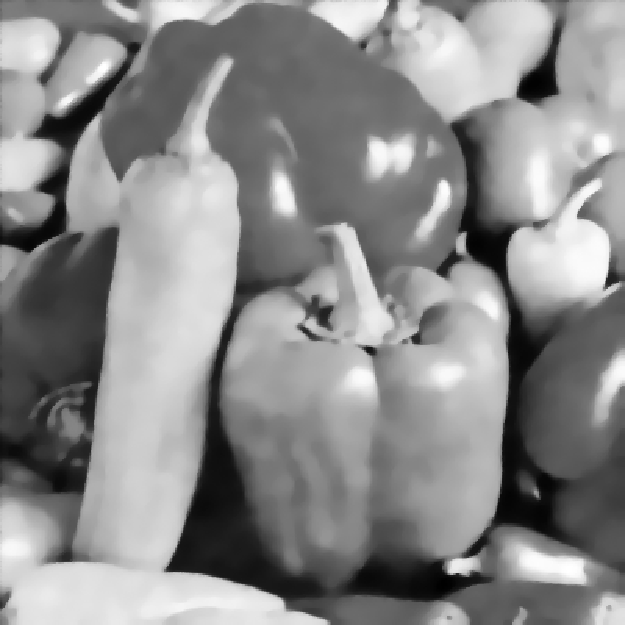}
		\caption{From left to right and from up to down: A completely noisy pepper image; Results corresponding to the cases of first-order TV flow, second-order TV flow with homogeneous initial velocity, inhomogeneous initial velocity $\dot{u}(0)=-0.5\eta g$ $\dot{u}(0)=-2\eta g$, and $\dot{u}(0)=-\eta g$, respectively, where  $\eta=10$.}
		\label{fig:denoising_nonhomo2}
	\end{figure}
Using the same stopping rule for all the cases, the algorithms with the first-order TV flow, the second-order TV flow with homogeneous initial velocity and  inhomogeneous initial velocities $\dot{u}(0)=-0.5\eta g$ and $\dot{u}(0)=-2\eta g$ have eventually all ran out of $50000$ iterations but still returned nothing meaningful. While using the initial velocity $\dot{u}(0)=-\eta g$, the algorithm of second-order TV flow terminates after $1221$ iterations with an almost perfect image.
The experiment here indicates that the setting $\dot{u}(0)=-\eta g$ seems to be critical among other inhomogeneous initial velocities. 
This would need more investigation both theoretically and numerically, which we do not pursue further in this paper.
For the second-order MCF, it is not clear how to identify a meaningful inhomogeneous initial velocity. Therefore, in the following examples, we restrict to homogeneous ones for both cases.
Nevertheless, the above example shows some interesting potential of second-order flows which seems to be capable of decomposing an image by providing some prior information in the initial velocity.

\subsubsection*{Displacement error correction}
Now, we show the results of correcting displacement errors in Figure \ref{fig:dejittering}.
We again consider the pepper image of pixel size $400\times 400$ to be in the domain $\Omega=(0,1)\times (0,1)$. We degrade the image by some displacement error yielding a line-jittered pepper image. 
This is done by shifting every horizontal row of the image randomly, that is
\[u^d(x_1,x_2)=u(x_1+d(x_2),x_2), \text{ where } d(\cdot) \text{ is a  random process, and } \norm{d}_\infty \leq \sigma.\]
In the experiments, we set $\sigma$ to be $8$ pixels, and choose the time-step-size to be $\Delta t_1=0.003$, and $\eta=\frac1{50 \Delta t_1}$ for the damped second-order TVF, and the step size $\Delta t_2=0.0001$, and $\eta=\frac1{30 \Delta t_2}$ for the damped second-order level-set MCF. For the stopping rule of the algorithm, we set $\rho=0.2,\; tol=0.3$.
Notice that the iteration numbers are again different between the algorithms due to the first-order flows and the second-order flows, in order to reach the stopping criteria.
The number of iterations of the first-order TVF vs. the second-order TVF is $47769 \text{ vs. } 952$, and the number of iterations of the first-order level-set MCF vs. the second-order level-set MCF is $5614 \text{ vs. } 243$.
We see that the results of the TVFs and the level-set MCFs turn out to be quite different in this example. This is not surprising,  as we have observed from Figure \ref{fig:square} that the second-order TVF has very limited effect in changing the curvature of the level lines in comparison with the second-order level-set MCF.
Consequently, the algorithm with TVFs pays a bigger price to correct the jitter error as it sacrifices the contrast of the images which we can see clearly from Figure \ref{fig:dejittering}.
\begin{figure}[h!]
		\includegraphics[width=0.3\textwidth]{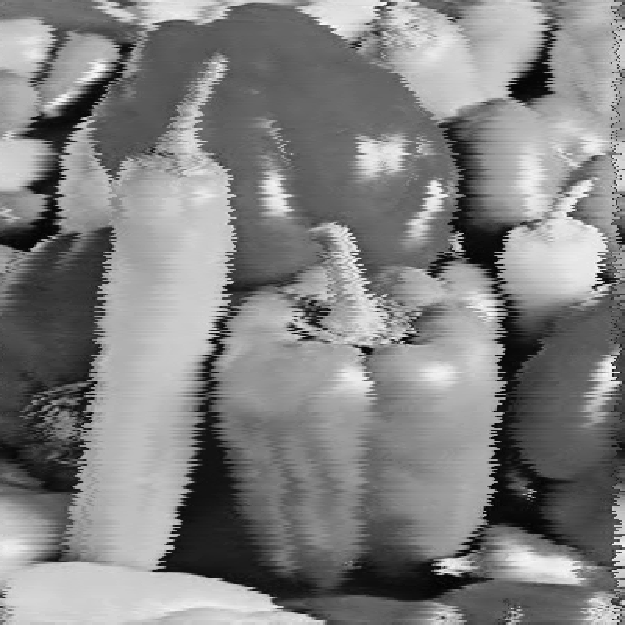} \quad
		\includegraphics[width=0.3\textwidth]{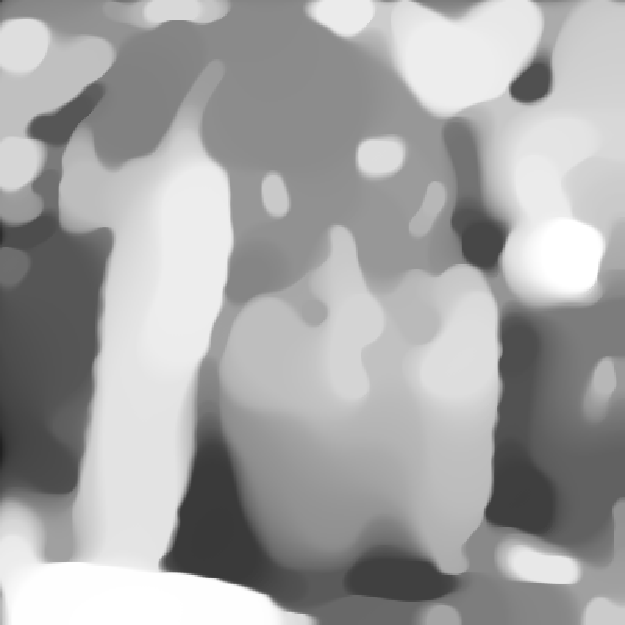} \quad
		\includegraphics[width=0.3\textwidth]{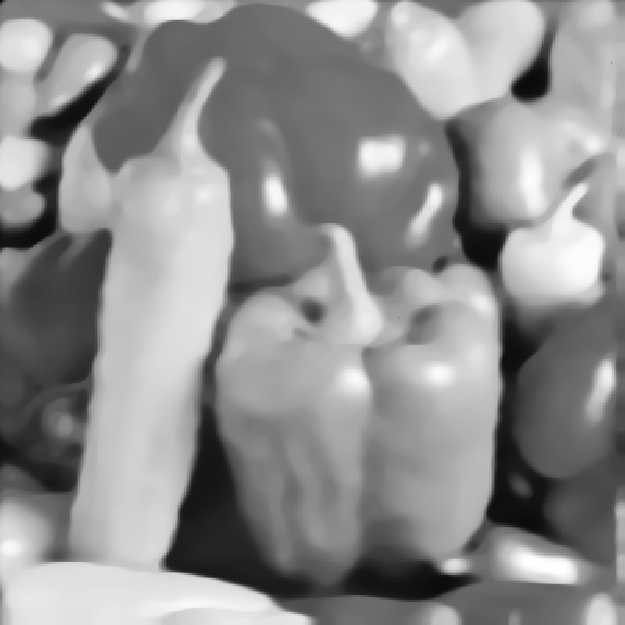}\\ \vskip 0.1pt
		\includegraphics[width=0.3\textwidth]{fig/pepper.eps} \quad
		\includegraphics[width=0.3\textwidth]{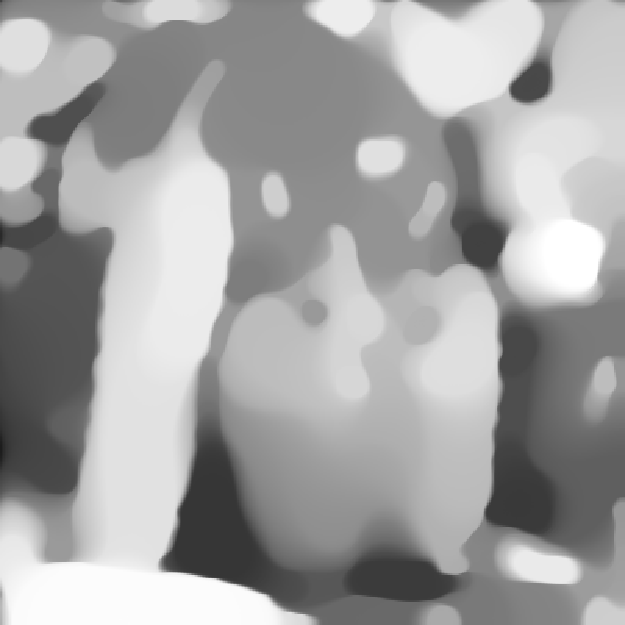} \quad
		\includegraphics[width=0.3\textwidth]{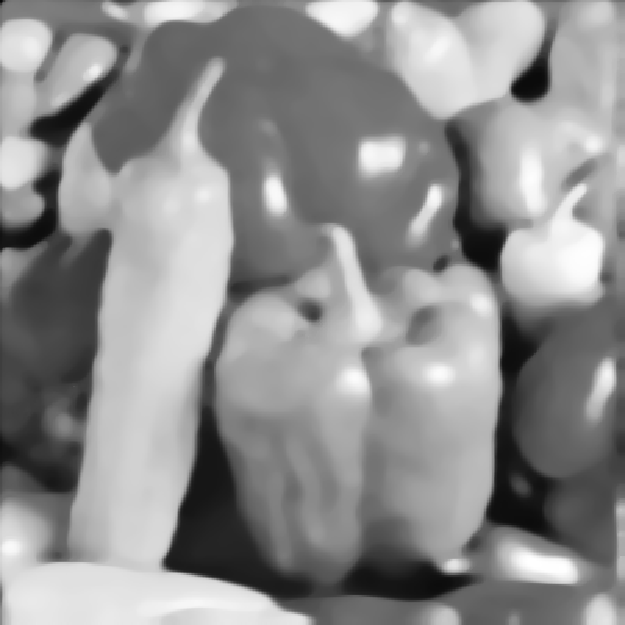}
		\caption{Dejittering of pepper image. Above: from left to right: the jittered image;
			the result of first-order TVF and the result of first-order level-set MCF.	
			Below: from left to right: the ground truth; the result of second-order TVF and the result of second-order level-set MCF. }
		\label{fig:dejittering}
\end{figure}

\subsubsection*{Simultaneous denosing and correcting displacement errors}
In Figure \ref{fig:denoising+dejittering}, we show the results on dejittering and denoising simultaneously using the algorithm by the damped second-order level-set MCF and comparing it with the results from the damped second-order TVF.
We use the jittered pepper image from the last example and then add the same amount of Gaussian noise as we did in the denoising example (appr. $15$ percentage of noise), and also we select the same disretization parameters as before.
We do not show the results given by the first-order flows as they are similar to the second-order ones but requiring larger iteration times.
From  Figure \ref{fig:denoising+dejittering},  we see that, up to the stopping threshold with ($\rho=0.2, \;tol=0.5$), the second-order level-set MCF performs better in correcting the displacement error than the second-order TVF, while the latter does better for denoising than the former. Overall, the algorithm by second-order level-set MCF outperforms in this example as both the noise and the jitter are significantly reduced simultaneously.
The observed iteration times are $581$ for the second-order TVF and $241$ for  second-order level-set MCF, respectively.
\begin{figure}[h!]
	\begin{center}
		\includegraphics[width=0.3\textwidth]{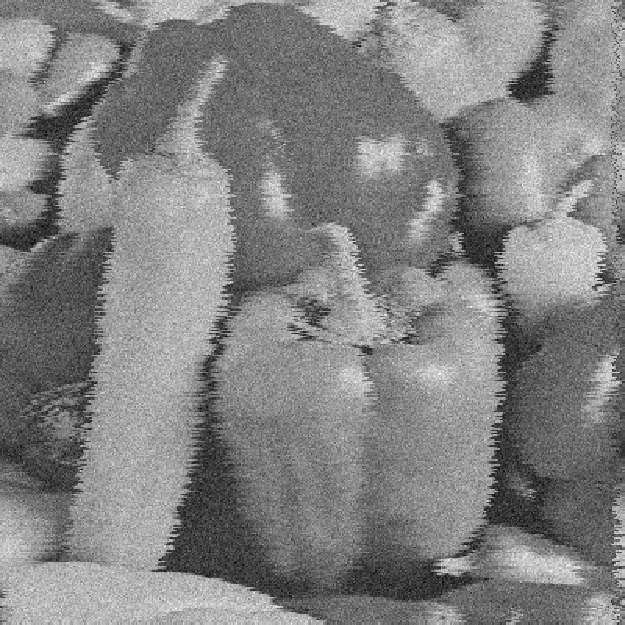} \quad
		\includegraphics[width=0.3\textwidth]{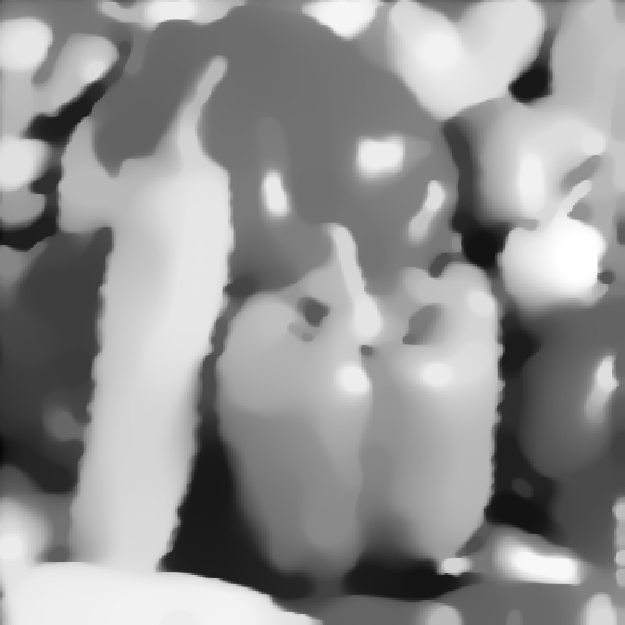} \quad
		\includegraphics[width=0.3\textwidth]{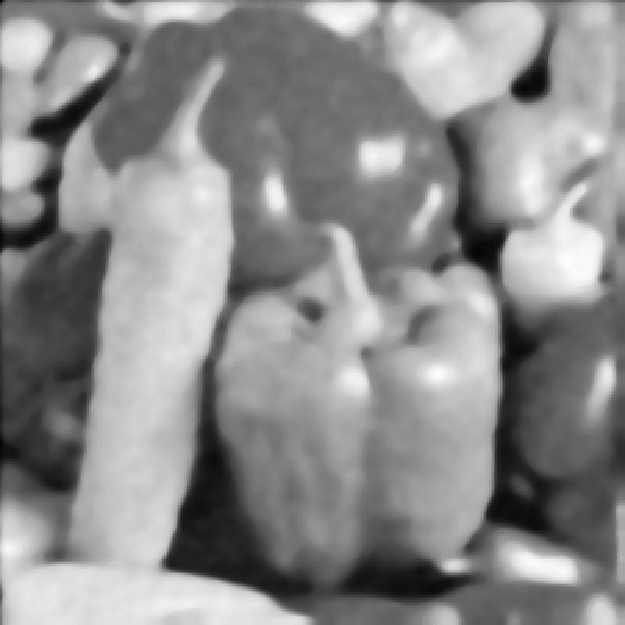}
		\caption{Dejittering and denoising of pepper image simultaneously. From left to right: the jittered and noisy image; the result of  second-order TVF and the result of  second-order level-set MCF. }
		\label{fig:denoising+dejittering}
	\end{center}
\end{figure}
\begin{remark}
	Notice that a smaller value of $\eta$ in Algorithm \ref{alg:symplectic} usually results in better efficiency (less number of iterations to achieve the same outcome with respect to the same discrete time step size) in both denoising and dejittering tasks for both the second-order flows. However, there is a trade off as too small $\eta$ causes unstable evolution. This is particularly relevant for the second-order level-set MCF algorithm as explained in Section \ref{SecHeuristics}. There we have argued that $\eta$ needs to be sufficiently large to provide an energy decay with respect to the level sets evolution.
\end{remark}

	\paragraph{Summary of quantitative comparison using standard image measures for homogeneous initial velocity}
	\begin{table}[h!]
		\begin{center}
			\renewcommand{\arraystretch}{1.2}
			\resizebox{\textwidth}{!}{
				\begin{tabular}{|l|l|l|l|l|l|l|l|l|l|l|l|l|}
					\hline
					Methods  &	Iterations & MSE   & SSIM &	Iterations & MSE   & SSIM&	Iterations & MSE   & SSIM \\\hline
					$\rho=0.2$ &\multicolumn{3}{c|}{Denoising (tol$=1.0$)} &\multicolumn{3}{c|}{Dejittering (tol$=0.3$)} &\multicolumn{3}{c|}{Denoising \& jitering (tol$=0.5$)} \\			\hline
					SO-TVF & $338 $ &  $ 92.68 $  &    $0.52  $ & $952  $ &  $762.72  $  &     $0.18  $& $ 584 $ &  $ 695.37 $   &  $ 0.25 $   \\	\hline
					TVF & $ 19189 $ &  $128.79  $   &  $  0.44$ & $ 47769 $ &  $ 779.74 $   &  $ 0.17 $& $ 31305 $ &  $ 714.68 $   &  $ 0.21 $   \\	\hline
					SO-MCF & $331  $ &  $ 89.39 $    &  $ 0.58 $ & $ 243 $ &  $ 739.07 $   &  $ 0.34 $& $242  $ &  $ 688.63 $    &  $0.31  $   \\	\hline
					MCF  & $4299  $ &  $105.95  $   &  $ 0.57 $ & $5614  $ &  $728.83  $    &  $ 0.35 $& $6952  $ &  $ 700.28 $   &  $0.32  $   \\	\hline
			\end{tabular}}
			\caption{Quantitative comparison (second-order flows here are with homogeneous initial velocity).}\label{tab:comparison}
		\end{center}
	\end{table}
	In Table \ref{tab:comparison}, we summarize the quantitative performance of the above visualized results using again standard computational measures MSE and SSIM.
	We only present the result using homogeneous initial velocity.
	The same conclusion as the visual comparisons can be made here.
	For the same set of experiments, using the same parameters $\rho$ and tol for the stopping criteria in Algorithm \ref{alg:symplectic}, the algorithms by MCFs produce either equally or better results as by TVFs for almost all the tasks, and the algorithm from second-order MCF shows better numerical efficiency than the one from the first-order MCF.

\section{Concluding remarks}
\label{sec:Conclusion}
This paper has studied two geometric quasilinear hyperbolic partial differential equations, namely the second-order total variation flow and the second-order level set mean curvature flow.
For the former equation, we have a relatively complete analytical result on its well-posedness, which is attributed to the convexity of the total variational functional. However, for the latter, we have only obtained a very preliminary result on its well-posedness  by considering a regularized version.
The main difficulty there comes from the degeneracy of the hyperbolic structure of the equation. Particularly, different to the former problem, there has no associated convex functional been found out for the latter one.
Instead, we identified some novel geometric PDEs evolving hypersurfaces to understand the behavior of the solution of the second-order level-set MCF. 
From an application point of view, we have observed that both second-order flows are able to generate efficient numerical algorithms for the motivating tasks in imaging sciences.
The two types of flows have different behaviors, and this has been verified by numerical examples. As a consequence, they have different strengths in our imaging applications. The TVFs are able to remove additive noise efficiently, but cannot properly treat displacement errors, while the MCFs seem to simultaneously deal with these two tasks, at least to some extent.
Based on the above observation, several interesting theoretical problems have been identified. For instance, the well-posedness of the original second-order level-set MCF \eqref{eq:acc_mcf_u}, and asymptotic analysis of its solution.
On the other hand, it would also be interesting to conduct a systematic investigation of the newly derived geometric PDE and its corresponding level sets equation pointed out in equations \eqref{eq:second_level_set} and  \eqref{eq:new_acc_mcf_u}, respectively.
We notice that applications for the respective first-order flows are quite rich in the literature, which are far beyond denoising and dejittering as we studied here. It remains an open problem to see whether the second-order flows can be generalized to take care of these applications as well. 
Moreover, inhomogeneous initial velocity and spatial dependent parameter $\eta$ for second-order dynamics seems to open interesting avenues for further investigation, particularly, in imaging applications.

\subsection*{Acknowledgement}
GD thanks M{\aa}rten Gulliksson for his hospitality when he was visiting \"Orebro University, and he also thanks Otmar Scherzer for some stimulating discussions.
GD and MH acknowledge support of Institut Henri Poincar\'e (IHP) (UMS 839 CNRS-Sorbonne Université), and LabEx CARMIN (ANR-10-LABX-59-01), while they both got funded for visiting `The Mathematics of Imaging' research trimester held at IHP in 2019, where a first version of this manuscript was finalized.
The authors are extremely grateful to the reviewers for their thorough and constructive comments which have significantly improved the paper, particularly on the suggestion of constructing an analytical solution for the second-order TV flow, discussed in Section \ref{subsec:Ana_solution_stvf} and the question on extinction time.

\appendix
\section{Convergence analysis of Algorithm \ref{alg:symplectic}}
Discrete images consist of pixels and are stored as matrices. For the convenience of  theoretical analysis of the algorithm, we represent the matrices by column vectors in the following. However, in practice, direct matrix operations can be performed and they are preferable in terms of computational efficiency.
\begin{definition}
	Given a matrix $\mathbf{u}\in \mathbb{R}^{M}\times \mathbb{R}^{N}$, one can obtain a vector $\vec{\mathbf{u}}\in \mathbb{R}^{MN}$ by stacking the columns of $\mathbf{u}$. This defines a linear operator $vec: \mathbb{R}^{M}\times \mathbb{R}^{N} \to \mathbb{R}^{MN}$, 	$\vec{\mathbf{u}} = vec(\mathbf{u})$, where
	\begin{equation*}\label{vec}
	 vec(\mathbf{u}) = (\mathbf{u}_{1,1}, \mathbf{u}_{2,1}, \cdot\cdot\cdot, \mathbf{u}_{M,1}, \mathbf{u}_{1,2}, \mathbf{u}_{2,2}, \cdot\cdot\cdot, \mathbf{u}_{M,2}, \cdot\cdot\cdot, \mathbf{u}_{1,N}, \mathbf{u}_{2,N}, \cdot\cdot\cdot, \mathbf{u}_{M,N})^\top.
	\end{equation*}
\end{definition}
Note that $vec(\mathbf{u})$ corresponds to a lexicographical column ordering of the components in the matrix $\mathbf{u}$. The symbol $array$ denotes the inverse of the $vec$ operator, i.e.,
\begin{equation*}\label{array}
array(vec(\mathbf{u})) = \mathbf{u} \; \text{ and }\;  vec(array(\vec{\mathbf{u}})) = \vec{\mathbf{u}},
\end{equation*}
whenever $\mathbf{u}\in \mathbb{R}^{M}\times \mathbb{R}^{N}$ and $\vec{\mathbf{u}}\in \mathbb{R}^{MN}$.
Denote by $\mathbf{u}^{k}$ the reconstructed image at iteration $k$. Then, based on the above definition and the discretization formula \eqref{linearApproximation0}, the right-hand side of our damped second-order flows, i.e.,
\[\abs{\nabla \mathbf{u}^{k}}\operatorname{div}\left(\frac{\nabla \mathbf{u}^{k}}{\abs{\nabla \mathbf{u}^{k}}+ \epsilon} \right) \quad  \text{ or } \quad  \operatorname{div}\left(\frac{\nabla \mathbf{u}^{k}}{\abs{\nabla \mathbf{u}^{k}}+ \epsilon} \right),\]
can be rewritten in an abstract matrix form as $\mathbf{F}^{k} \vec{\mathbf{u}}^{k}$, where the matrix $\mathbf{F}^{k}$ depends on $\vec{\mathbf{u}}^{k}$. The precise form of $\mathbf{F}^{k}$ is given in \eqref{eq:Fk}, where its spectral properties and their usage in convergence considerations are discussed.
For the sake of simplicity and clarity of statements, let us consider a uniform grid $\Omega_{MN}=\{(x_i,y_j)\}^{M,N}_{i,j=1}$, discretizing $\Omega$ with the uniform step size $h=x_{i+1}-x_i=y_{j+1}-y_j$. Define $\mathbf{u}(t)=[u(x_i,y_j,t)]^{M,N}_{i,j=1}$, and denote by $\mathbf{u}^{k}$ the projection of $u(x,y,t)$ onto the spatial grid $\Omega_{MN}$ and time point $t=t_k$. Denote by 
\[c^{\epsilon,k}_{i,j}(\mathbf{u}^{ k})= 1/ \left( \epsilon + \frac{1}{h} \sqrt{( \mathbf{u}^{k}_{i+1,j} -  \mathbf{u}^{k}_{i,j})^2 + ( \mathbf{u}^{k}_{i,j+1} -  \mathbf{u}^{k}_{i,j})^2}  \right).\]
Using forward-backward differences, we obtain
\begin{equation}\label{linearApproximation0}
\begin{aligned}
&\left[ \textmd{div}\left( c^{\epsilon,k}(\mathbf{u}^{ k}) \nabla \mathbf{u}^{k} \right) \right]_{i,j} 
=\frac{1}{h^2} \left( c^{\epsilon, k}_{i-1,j} \mathbf{u}^{k}_{i-1,j} +  c^{\epsilon, k}_{i,j-1} \mathbf{u}^{k}_{i,j-1}    \right) \\
&\qquad- \frac{1}{h^2} \left( 2 c^{\epsilon, k}_{i,j} + c^{\epsilon, k}_{i-1,j} + c^{\epsilon, k}_{i,j-1} \right) \mathbf{u}^{k}_{i,j} +  \frac{1}{h^2} \left(c^{\epsilon, k}_{i,j} \mathbf{u}^{k}_{i,j+1} +  c^{\epsilon, k}_{i,j} \mathbf{u}^{k}_{i+1,j}\right) .
\end{aligned}
\end{equation}
To put TVFs and MCFs under a common umbrella, we use $b(\mathbf{u}^{k}) \textmd{div}\left( c^{\epsilon,k}(\mathbf{u}^{ k}) \nabla \mathbf{u}^{k} \right)$ to represent the nonlinear part of the equations: $b(\mathbf{u}^{k})\equiv 1$ for TVFs, and 
\[b(\mathbf{u}^{k})=\abs{\nabla \mathbf{u}^{k}}=\left( \left[ \frac{1}{2h} \sqrt{( \mathbf{u}^{k}_{i+1,j} -  \mathbf{u}^{k}_{i-1,j})^2 + ( \mathbf{u}^{k}_{i,j+1} -  \mathbf{u}^{k}_{i,j-1})^2} \right]_{i,j} \right)\]
for MCFs. 
By applying lexicographical column ordering of $\mathbf{u}^{k}_{i,j}$ and assuming the Neumann boundary condition, we obtain $b(\mathbf{u}^{k}) \textmd{div}\left( c(\mathbf{u}^{\epsilon, k}) \nabla \mathbf{u}^{k} \right)$ in matrix representation, denoted as $\mathbf{F}^{k} \vec{\mathbf{u}}^{k}$, where  
\begin{equation}\label{eq:Fk}
\mathbf{F}^{k}=\mathbf{B}^{k} \mathbf{G}^{k}, \quad \mathbf{B}^{k}=\operatorname{diag}(b^k_{1,1}, b^k_{2,1}, \cdots, b^k_{M,1}, b^k_{1,2}, \cdots, b^k_{M,N}),
\end{equation}
and $\mathbf{G}^{k}$ is the $MN\times MN$ matrix with $N\times N$ block entries given by
\begin{equation}\label{eq:Gk}
\mathbf{G}^{k} = \left(
\begin{array}{cccccc}
L^k_1 &  I^k_1 & \mathbf{0} & \cdots & \mathbf{0} & \mathbf{0} \\
I^k_1 & L^k_2 & I^k_2  & \mathbf{0} & \ddots & \mathbf{0} \\
\mathbf{0} & I^k_2 & L^k_3  & I^k_3 & \mathbf{0} & \vdots \\
\vdots &  & \ddots  & \ddots & \ddots &  \\
\mathbf{0} &  \ddots & \mathbf{0}  & I^k_{N-2} & L^k_{N-1} & I^k_{N-1} \\
\mathbf{0} &  \mathbf{0} & \cdots  & \mathbf{0} & I^k_{N-1} & L^k_{N} \\
\end{array}
\right) .
\end{equation}
Here $I^k_j$ is the $M\times M$ diagonal matrix $I^k_j = \operatorname{diag}\left( c^{\epsilon, k}_{1,j}, \cdots, c^{\epsilon, k}_{M,j}\right)$, $\mathbf{0}$ represents the $M\times M$ zero matrix, and $L^k_j$ is the $M\times M$ matrix of the form
\begin{equation*}
L^k_j = \frac{1}{h^2} \left(
\begin{array}{cccccc}
-\tilde{c}^{\epsilon, k}_{1,j} & c^{\epsilon, k}_{1,j} & 0 & \cdots & 0 & 0 \\
c^{\epsilon, k}_{1,j} & -\tilde{c}^{\epsilon, k}_{2,j} & c^{\epsilon, k}_{2,j}  & 0 & \ddots & 0 \\
0 & c^{\epsilon, k}_{2,j}  & -\tilde{c}^{\epsilon, k}_{3,j}  & c^{\epsilon, k}_{3,j} & 0 & \vdots \\
\vdots &  & \ddots  & \ddots & \ddots &  \\
0 &  \ddots & 0  & c^{\epsilon, k}_{M-2,j} & -\tilde{c}^{\epsilon, k}_{M-1,j} & c^{\epsilon, k}_{M-1,j} \\
0 &  0 & \cdots  & 0 & c^{\epsilon, k}_{M-1,j} & -\tilde{c}^{\epsilon, k}_{M,j}  \\
\end{array}
\right),
\end{equation*}
where $\tilde{c}^{\epsilon, k}_{i,j}: = 2 c^{\epsilon, k}_{i,j} + c^{\epsilon, k}_{i-1,j} + c^{\epsilon, k}_{i,j-1}$.

\begin{proposition}\label{NegativeEigen}
	All eigenvalues of $\mathbf{G}^{k}$ (for all $k\in \N $) are non-positive.
\end{proposition}

\begin{proof}
	By the definition of $\mathbf{G}^{k}$, i.e. \eqref{eq:Gk}, it is not difficult to show that $\mathbf{G}^{k}$ is a symmetric and diagonally dominant matrix. Then, all eigenvalues of $\mathbf{G}^{k}$ (for all $k\in \N $)  are real and, by Gershgorin's circle theorem \cite{Ger31}, for each eigenvalue $\lambda$ there exists an index $\nu$ such that
	\begin{equation*}
	\lambda \in \left[ [\mathbf{G}^{k}]_{\nu,\nu} - \sum^{MN}_{\imath \neq \nu} |[\mathbf{G}^{k}]_{\nu,\imath}|, [\mathbf{G}^{k}]_{\nu,\nu} + \sum^{MN}_{\imath \neq \nu} |[\mathbf{G}^{k}]_{\nu,\imath}| \right],
	\end{equation*}
	which implies, by definition of the diagonal dominance, $\lambda\leq 0$. Here, $[\mathbf{G}^{k}]_{\nu,\imath}$ denotes the element of the matrix $\mathbf{G}^{k}$ at the position $(\nu,\imath)$.
\end{proof}

Denote $\vec{\mathbf{v}}^k=\frac{d \vec{\mathbf{u}}^k}{dt}$, and recall Algorithm \ref{alg:symplectic} where the symplectic Euler scheme is applied to discretize the second-order flow \eqref{eq:acc_tvf} or \eqref{eq:acc_mcf_u}, i.e.,
\begin{equation}\label{symplectic}
\left\{\begin{array}{l}
\vec{\mathbf{v}}^{k+1} =(1- \Delta t_k\eta) \vec{\mathbf{v}}^{k} + \Delta t_k   \mathbf{F}^{k} \vec{\mathbf{u}}^{k} , \\
\vec{\mathbf{u}}^{k+1} = \vec{\mathbf{u}}^{k} + \Delta t_k \vec{\mathbf{v}}^{k+1}, \\
\vec{\mathbf{u}}_{0}=\vec{\mathbf{u}}^d, \vec{\mathbf{v}}_0=0,
\end{array}\right.
\end{equation}
where $\vec{\mathbf{u}}^d= vec(\mathbf{u}^d)$ and $\mathbf{u}^d$ is the project of $u^d(x)$ onto the grid $\Omega_{MN}$.

Now, we are in a position to give a numerical analysis for the scheme \eqref{symplectic} in Algorithm \ref{alg:symplectic}.

\begin{theorem}\label{ThmFixDamping}
	Let  $\eta>0$ be a fixed damping parameter. If the step size is chosen to fulfill
	\begin{equation}\label{eq:parametersDtSteady}
	\Delta t_k \leq \min \left( \frac{1}{\eta} , \frac{1}{\sqrt{b^{k}_{max} \lambda^{k}_{max}}} \right),
	\end{equation}
	where $\lambda^{k}_{max}$ is the maximal eigenvalue of $-\mathbf{F}^{k}$, and $b^{k}_{max}:= \max^{M,N}_{i,j=1} b^k_{i,j}$, then the scheme \eqref{symplectic} is convergent.
\end{theorem}

\begin{proof}
	
	Denote $\mathbf{z}_k=(\mathbf{v}^{k};\mathbf{u}^{k})$, and rewrite equation \eqref{symplectic} by
	\begin{equation}\label{ODE2Steady}
	\mathbf{z}_{k+1} = A_k \mathbf{z}_k,
\quad
\text{ 	where }\quad	A_k  = \left(
	\begin{array}{cc}
	I_n + \Delta t_k^2 \mathbf{F}^{k} & \Delta t_k \left( 1 - \Delta t_k \eta \right) I_n \\
	\Delta t_k \mathbf{F}^{k} &  \left( 1 - \Delta t_k \eta \right) I_n
	\end{array}
	\right).
	\end{equation}	
	It is well-known that a sufficient condition for the convergence of the iteration scheme \eqref{ODE2Steady} is that $A_k$ is a contractive operator, i.e., $\|A_k\|_2<1$.
	By the elementary calculation and the decomposition $-\mathbf{G}^{k}=Q \Lambda^{k} Q^\top$ with $\Lambda^{k}=\operatorname{diag}(\lambda^{k}_{i}), \lambda^{k}_{i}\geq0, i =1,...,MN$, we derive that the eigenvalues of $A_k$ are
	\begin{equation*}\label{Spectral}
	\mu^{k}_{i,\pm} = 1- \frac{\Delta t_k}{2} \left[ \left( \Delta t_k b^{k}_i \lambda^{k}_i + \eta \right) \pm  \sqrt{\left( \Delta t_k b^{k}_i \lambda^{k}_i + \eta \right)^2 - 4 b^{k}_i\lambda^{k}_i} \right] \quad i=1,2,\cdots, MN,
	\end{equation*}
	where $b^{k}_i\geq0$ represents the $i$-th element in the diagonal of matrix $\mathbf{B}^{k}$.
	Hence, in order to prove $\|A_k\|_2 \leq1$, it is sufficient to show that for all $i=1, ..., MN$: $|\mu^{k}_{i,\pm}| \leq1$ for the time step size $\Delta t_k$, defined in \eqref{eq:parametersDtSteady}.
	
	For each fixed $i$, there are three possible cases: 
(i)	$ \left( \Delta t_k b^{k}_i \lambda^{k}_{i} + \eta \right)^2 > 4 b^{k}_i \lambda^{k}_{i}$ overdamped case;
(ii) $\left( \Delta t_k b^{k}_i \lambda^{k}_{i} + \eta \right)^2 < 4 b^{k}_i \lambda^{k}_{i}$ underdamped case;
(iii)  $ \left( \Delta t_k b^{k}_i \lambda^{k}_{i} + \eta \right)^2 = 4 b^{k}_i \lambda^{k}_{i}$ critical damping case.	
We consider each case separately.

	For the overdamped case, define $a:= \frac{\eta + \Delta t_k b^{k}_i \lambda^{k}_{i}}{2 \sqrt{b^{k}_i \lambda^{k}_{i}}}$ ($a>1$). Then,
$	\mu^{k}_{i,\pm} = 1- \Delta t_k \sqrt{b^{k}_i \lambda^{k}_{i}} (a\pm \sqrt{a^2-1}).$
	Obviously, $\mu^{k}_{i,\pm}<1$ by noting the positivity of the second term on the right-hand side of the equation above. Now, let us show the inequality $\mu^{k}_{i,\pm}>-1$. By the choice of the time step size $\Delta t_k$ in \eqref{eq:parametersDtSteady}, we know that $\eta\Delta t_k<1$ and $b^{k}_{max} \lambda^{k}_{max}\Delta t^2_k <1$, which implies that $\eta\Delta t_k + b^{k}_{max} \lambda^{k}_{max}\Delta t^2_k \leq 2$. Therefore, we have
	\begin{equation}\label{bounded1Steady}
	\frac{2}{\eta + b^{k}_{max} \lambda^{k}_{max}\Delta t_k} \geq \Delta t_k.
	\end{equation}
	
	Since $a =\frac{\eta}{2 \sqrt{b^{k}_i\lambda^{k}_{i}}} + \frac{\Delta t_k}{2} \sqrt{b^{k}_i\lambda^{k}_{i}}$, using inequality \eqref{bounded1Steady}, we deduce that
	\begin{equation*}
	\begin{aligned}
	&\frac{2}{\sqrt{b^{k}_i\lambda^{k}_{i}} \left( a\pm\sqrt{a^2-1} \right)} \geq  \frac{2}{\sqrt{b^{k}_i\lambda^{k}_{i}} \left( a+\sqrt{a^2-1} \right)} > \frac{1}{\sqrt{b^{k}_i\lambda^{k}_{i}} a}\\
	 = &\frac{1}{\sqrt{b^{k}_i\lambda^{k}_{i}} \left( \frac{\eta}{2 \sqrt{b^{k}_i\lambda^{k}_{i}}} + \frac{\Delta t_k}{2} \sqrt{b^{k}_i\lambda^{k}_{i}} \right)} = \frac{2}{\eta+ b^{k}_i\lambda^{k}_{i} \Delta t_k }  \geq \frac{2}{\eta+ b^{k}_{max} \lambda^{k}_{max} \Delta t_k } \geq \Delta t_k,
	\end{aligned}
	\end{equation*}
	which implies that $\mu^{k}_{i,\pm} = 1- \Delta t_k \sqrt{b^{k}_i\lambda^{k}_{i}} (a\pm \sqrt{a^2-1}) > -1. $
	Therefore, we conclude that $|\mu^{k}_{i,\pm}|\leq1$ for the overdamped case.
	
	Now, consider the underdamped case. In this case, since $\left( \Delta t_k b^{k}_i \lambda^{k}_{i} + \eta \right)^2 < 4 b^{k}_i \lambda^{k}_{i}$, we have
	\begin{equation*}
	| \mu^{k}_{i,\pm} | = \left|  1- \frac{\Delta t_k}{2} \left( \Delta t_k b^{k}_i \lambda^{k}_i + \eta \right) \pm  i \frac{\Delta t_k}{2} \sqrt{4 b^{k}_i \lambda^{k}_i -\left( \Delta t_k b^{k}_i \lambda^{k}_i + \eta \right)^2}   \right| = \sqrt{1- \eta \Delta t_k} .
	\end{equation*}
	which implies $| \mu^{k}_{i,\pm} |<1$ for any fixed pair $(\eta,\Delta t_k)$ satisfying \eqref{eq:parametersDtSteady}.
	
	Finally, consider the critical damping case. In this case, the eigenvalue for $\mu^{k}_{i,\pm}$ is simply given by
	\begin{equation*}
	| \mu^{k}_{i,\pm} | = |1- \Delta t_k \sqrt{b^{k}_i \lambda^{k}_{i}}| = \sqrt{1- \eta \Delta t_k},
	\end{equation*}
	which yields the desired result according to the argument in the critical damping case.
	In conclusion, all the eigenvalues of matrix $A_k$ are smaller or equal than $1$, therefore it is a contractive operator. Then the scheme \eqref{symplectic} is convergent.
\end{proof}


\bibliographystyle{siamplain}

\end{document}